\input ./style/arxiv-general.cfg
\documentclass[aop,MSNbibl,dvips]{arximspdf}
\makeatletter
   \@ifpackageloaded{graphicx}{}{\usepackage{graphicx}}
\makeatother
\usepackage{mathrsfs}

\doi{10.1214/15-AOP1050}
\volume{45}
\issue{1}
\pubyear{2017}
\firstpage{147}
\lastpage{209}
\docsubty{FLA}

\makeatletter

\def\vfrac#1#2{(#1)/#2}
\def\afrac#1#2{#1/(#2)}
\def\vafrac#1#2{(#1)/(#2)}

\def\sklvfrac#1#2{((#1)/#2)}
\def\sklafrac#1#2{(#1/(#2))}

\newcommand{\rrvert}{\vert}
\newcommand{\llvert}{\vert}
\renewcommand{\mid}{|}
\newcommand{\eqref}[1]{(\ref{#1})}

\newtheorem{theorem}{Theorem}
\newtheorem{proposition}[theorem]{Proposition}
\newtheorem{lemma}[theorem]{Lemma}
\newtheorem{corollary}[theorem]{Corollary}
\newproclaim{rem}{Remark}

\def\tllbracket{[\![}
\def\trrbracket{]\!]}

\def\cc{{\mathcal C}}

\def\t{{\mathcal T}}
\def\n{{\mathcal N}}
\def\v{{\mathcal V}}
\def\ve{\varepsilon}

\def\bp{\mathbf{p}}
\def\bd{\mathbf{d}}
\def\D{\mathrm{d}}
\def\la{\longrightarrow}
\def\T{{\mathbb T}}
\def\R{{\mathbb R}}
\def\P{{\mathbb P}}
\def\E{{\mathbb E}}
\def\N{{\mathbb N}}
\def\Z{{\mathbb Z}}

\makeatother

\begin{document}
\begin{frontmatter}

\title{The harmonic measure of balls in random trees}
\runtitle{The harmonic measure of balls in random trees}

\begin{aug}
\author[A]{\fnms{Nicolas}~\snm{Curien}\ead[label=e1]{nicolas.curien@gmail.com}}
\and
\author[A]{\fnms{Jean-Fran\c cois}~\snm{Le Gall}\corref{}\ead
[label=e2]{jean-francois.legall@math.u-psud.fr}}
\runauthor{N. Curien and J.-F. Le Gall}
\affiliation{Universit\'e Paris-Sud}
\address[A]{Universit\'e Paris-Sud\\
Math\'ematiques, b\^at. 425\\
91405 ORSAY Cedex\\
France\\
\printead{e1}\\
\phantom{E-mail: }\printead*{e2}}
\end{aug}

%
\received{\smonth{3} \syear{2014}}
%
\revised{\smonth{7} \syear{2015}}

%
\begin{abstract}
We study properties of the harmonic measure
of balls in typical large discrete trees. For
a ball of radius $n$ centered at the root,
we prove that, although the size of the boundary is of
order $n$, most of the harmonic measure is supported on
a boundary set of size approximately equal to $n^\beta$,
where $\beta\approx0.78$ is a universal constant.
To derive such results, we interpret harmonic measure as
the exit distribution of the ball by simple random walk on the
tree, and we first deal with the case of critical Galton--Watson
trees conditioned to have height greater than $n$. An important
ingredient of our approach is the analogous
continuous model (related to Aldous' continuum random tree), where the
dimension of harmonic
measure of a level set of the tree is equal to
$\beta$, whereas the dimension of the level
set itself is equal to $1$. The constant $\beta$ is expressed in terms
of the asymptotic
distribution of the conductance of large critical Galton--Watson trees.
\end{abstract}

%
\end{frontmatter}

\section{Introduction} \label{sec:introduction}

The main goal of this work is to study properties of the harmonic measure
of balls in large discrete trees.
From a probabilistic point of view, the harmonic measure of a set is
the exit distribution of that set by random walk,
in the discrete setting, or by Brownian motion, in the continuous
setting. Harmonic measure has been studied in depth
both in harmonic analysis and in probability theory, and it would be
hopeless to try to survey the literature on this subject. It
has been observed in different contexts that the harmonic measure of a
set with a
fractal-like boundary is often supported on a subset of the boundary of
strictly smaller dimension. For example, the famous Makarov theorem
\cite{Mak95} states that harmonic measure on the boundary of a simply
connected planar domain is always supported on a subset of Hausdorff
dimension equal to $1$, regardless of the dimension of the boundary
(see \cite{Law93} for similar results in a discrete setting and \cite
{Bou87} for higher-dimensional analogs). This ``dimension drop''
phenomenon also appears in the context of (infinite) discrete random
trees. In \cite{LPP95}, Lyons, Pemantle and Peres studied the harmonic
measure at infinity for simple random walk on an
infinite supercritical Galton--Watson tree and proved that the harmonic
measure is supported on a boundary set of dimension strictly less than
the dimension of the whole boundary. The same authors then extended
this result to biased random walk on a supercritical Galton--Watson
tree \cite{LPP96}.

In the present work, we study a similar phenomenon in the context of
finite discrete trees. Our results apply to several
combinatorial classes of discrete trees, such as plane trees, binary
trees or Cayley trees in particular. For a typical
tree with a (fixed) large size
chosen in any of these classes, we obtain that
the harmonic measure of a ball of radius $n$ is supported, up to a
small mass, on a subset of about $n^\beta$ vertices, despite the fact
that the
boundary of the ball has of order $n$ vertices. Here, $\beta\approx
0.78$ is a universal constant that does not depend on the
combinatorial class.

In order to obtain these results for ``combinatorial trees'', we
interpret them as conditioned Galton--Watson trees. Recall that a
Galton--Watson tree
describes the genealogy of a population starting with an ancestor or
root, where each individual has, independently of the others, a
number of children distributed according to a given offspring
distribution (see Section~\ref{sec:tree-discrete} for a precise definition).
We first study harmonic measure on generation $n$ of a \emph{critical}
Galton--Watson tree, whose offspring
distribution has mean $1$ and finite variance, and which is conditioned
to have height greater than $n$. In this setting, we obtain that most
of the harmonic measure on generation $n$ is concentrated on a set of
approximately $ n^\beta$ vertices, with high probability.
Again, this should be contrasted with the fact that the generation $n$
of the tree has about $n$ vertices.
The constant $\beta$ has an explicit expression in terms of the law of
a random variable $ \mathcal{C}$, which is the limit in distribution of
the (scaled) conductance
of the tree between the root and
generation $n$---again this limiting distribution does not depend
on the offspring distribution. In the related continuous model,
we show that the Hausdorff dimension of the harmonic measure is almost
surely equal to $\beta$, whereas the dimension
of the boundary is known to be equal to $1$. Let us describe our
results in a more precise way.

\renewcommand{\t}{\mathsf{T}}
\subsection*{Discrete setting}
Let\vspace*{1pt} $\theta$ be a probability measure on
$\Z_+$, and assume
that $\theta$ has mean one and finite variance $\sigma^2>0$.
Under the probability $ \mathbb{P}$, for every integer $n\geq0$, we
let $\t^{(n)}$ be a
Galton--Watson tree with offspring distribution $\theta$, conditioned
on nonextinction at generation $n$. Conditionally on the tree $ \t
^{(n)}$, we then consider
simple random walk on $\t^{(n)}$, starting from the root, and we let
$\Sigma_{n}$ be the first hitting point of generation $n$ by
random walk. The harmonic measure $\mu_n$
is the law of $\Sigma_n$. Notice that $\mu_n$ is a random probability
measure supported on
the set $\t^{(n)}_n$ of all vertices of $\t^{(n)}$ at generation $n$.
By a classical theorem of the theory of branching processes, $n^{-1}
\# \t^{(n)}_n$ converges in distribution to
an exponential distribution with parameter $2/\sigma^2$.

%
\begin{theorem}
\label{thm:maindiscrete}
There exists a constant $\beta\in(0,1)$, which does not depend on the
offspring distribution $\theta$, such that,
for every $\delta>0$, we have the convergence in $ \mathbb{P}$-probability
\[
\mu_n \bigl( \bigl\{v\in\t^{(n)}_n:
n^{-\beta-\delta}\leq\mu _n(v) \leq n^{-\beta+\delta} \bigr\} \bigr)
\mathop{\longrightarrow}\limits
_{n\to\infty}^{( \mathbb{P})} 1.
\]
Consequently, for every $\ve\in(0,1)$, there exists, with $\P
$-probability tending
to $1$ as $n\to\infty$, a subset $A_{n,\ve}$ of $\t^{(n)}_n$ such that
$\# A_{n,\ve}\leq n^{\beta+\delta}$ and
$\mu_n(A_{n,\ve})\geq1-\ve$. Conversely, the maximal $\mu_n$-measure
of a set of cardinality bounded by $n^{\beta-\delta}$ tends to $0$ as
$n\to\infty$,
in $\P$-probability.
\end{theorem}

Although we have no exact numerical expression for $\beta$,
calculations using the formulas in
Proposition \ref{prop:valuebeta} below indicate that $\beta\approx
0.78$. See the discussion at
the end of Section~\ref{sec:proofoftheoremmain-harmonic}.
This approximate numerical value confirms simulations made in physics
\cite{Jon12}.

The last two assertions of the theorem are easy consequences of the
first one. Indeed,
$A_{n,\ve}:=\{v\in\t^{(n)}_n: \mu_n(v)\geq n^{-\beta-\delta}\}$ has
cardinality smaller than $n^{\beta+\delta}$, and the
first assertion of the theorem shows that the $\mu_n$-measure of the
latter set is
greater than $1-\ve$ with $\P$-probability tending
to $1$ as $n\to\infty$. On the other hand,
if $A$ is any subset of $\t^{(n)}_n$ with cardinality smaller than
$n^{\beta-\delta}$, we have
\[
\mu_n(A)\leq\mu_n \bigl( \bigl\{v\in
\t^{(n)}_n: \mu_n(v) > n^{-\beta+\delta/2} \bigr\}
\bigr) + n^{\beta-\delta}n^{-\beta+\delta/2}
\]
and the first term in the right-hand side tends to $0$ in $\P$-probability
by the first assertion of the theorem.

Theorem \ref{thm:maindiscrete} implies a similar result for
Galton--Watson trees conditioned
to have a fixed size. For every integer
$N\geq0$ such that this makes sense, let $\mathbf{T}(N)$
be distributed under the probability measure
$\P$ as a Galton--Watson tree with offspring distribution $\theta$
conditioned to have
$N$ edges.
For every integer $n\geq1$, let $\mathbf{T}_n{(N)}$ be the set
of all vertices of $\mathbf{T}(N)$ at generation
$n$.
The
harmonic measure $\mu_n^N$ is defined on the event $\{ \mathbf
{T}_{n}{(N)} \ne\varnothing\}$ as the hitting distribution
of $\mathbf{T}_n{(N)}$ by simple random walk on $\mathbf{T}(N)$
started from the root.

%
\begin{corollary} \label{cor:planetree} Let $\delta>0$ and $\ve>0$. Then,
\begin{eqnarray*}
&&\mathbb{P} \bigl( \bigl\{ \mu^{N}_n \bigl( \bigl\{v\in
\mathbf{T}_{n}{(N)}: \mu^{N}_n(v)\notin
\bigl[n^{-\beta
-\delta}, n^{-\beta+\delta} \bigr] \bigr\} \bigr) > \ve \bigr\}\cap
\bigl\{\mathbf {T}_{n}{(N)} \ne\varnothing \bigr\} \bigr)
\\
&&\qquad\mathop{\longrightarrow}\limits
_{n,N\to\infty} 0.
\end{eqnarray*}
\end{corollary}

As in Theorem \ref{thm:maindiscrete}, this implies that, with high
probability on the event
$\{\mathbf{T}_{n}{(N)} \ne\varnothing\}$, the harmonic measure $\mu^{N}_n$
is supported, up to a mass less than $\ve$, on a set of $n^{\beta
+\delta}$ vertices, and conversely
the maximal $\mu^N_n$-measure of a set of cardinality bounded above by
$n^{\beta-\delta}$ is small.

If $h(\mathbf{T}(N))$ denotes the height (maximal distance from the
root) of the tree $\mathbf{T}(N)$, it is well known
that $N^{-1/2}h(\mathbf{T}(N))$ converges in distribution to a
positive random variable; see
\eqref{asymp-maxi} below.
Therefore, if we let $n$ and $N$ tend to infinity in such a way that
$n=o( \sqrt{N})$, the probability $\mathbb{P}(\mathbf{T}_{n}{( N)} \ne
\varnothing)$ tends to $1$. It is worth pointing that Corollary \ref
{cor:planetree}
applies to balls of radius $n$ which is large but small in comparison
with the diameter of the tree---a similar extension would in fact hold
also for Theorem \ref{thm:maindiscrete}.

For specific choices of $\theta$, the tree $\mathbf{T}(N)$ is
uniformly distributed over certain classes
of combinatorial trees, and Corollary \ref{cor:planetree} yields the
results that were mentioned earlier
in this \hyperref[sec:introduction]{Introduction}. In particular, if $\theta$ is the geometric
distribution $\theta(k)=2^{-k-1}$,
$\mathbf{T}(N)$ is uniformly distributed over plane trees with $N$
edges. If $\theta$ is the
Poisson distribution with mean $1$, and if we assign
labels $1,\ldots,N+1$ to vertices in a random manner and then
``forget'' the ordering of $\mathbf{T}(N)$, we get a random tree
uniformly distributed over Cayley trees
on $N+1$ vertices. In a similar manner, for every integer $p\geq2$, we
can handle $p$-ary trees (where the number of children of every vertex
belongs to $\{0,1,\ldots,p\}$) or strictly $p$-ary trees (where each
vertex has $0$ or $p$ children).

\subsection*{Continuous setting}
A key ingredient of the proof of
Theorem~\ref{thm:maindiscrete} is a similar result in the continuous
setting. A critical Galton--Watson tree conditioned on having height
greater than $n$, viewed as a metric space for
the graph distance normalized by the factor $n^{-1}$, is close in the
Gromov--Hausdorff sense to a variant of
Aldous' Brownian continuum random tree \cite{Ald91}, also called the
CRT. So a continuous analog of the harmonic measure $\mu_n$
would be the hitting distribution of height $1$ by Brownian motion on
the CRT starting from the root.
Although the construction of Brownian motion on the CRT has been
carried out in \cite{Kre95} (see also \cite{Cro08} for a simpler
approach, and \cite{AEW} for a general construction of Brownian motion
on $\R$-trees), we will not follow this approach, because there is a
simpler way of looking at the
continuous setting.

%
\begin{figure}

\includegraphics{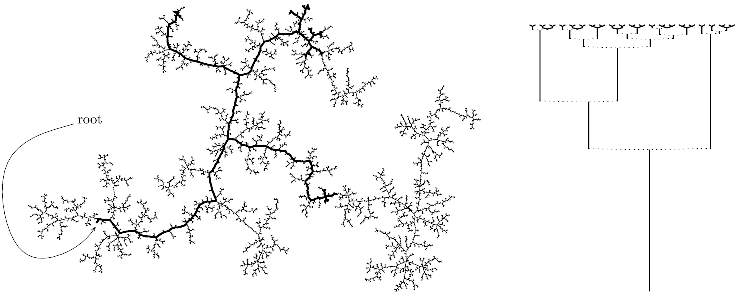}

\caption{A large (binary) Galton--Watson tree and the reduced tree at
a given level.}\label{fig1}
\end{figure}

The point is that properties of the harmonic measure $\mu_{n}$ on $ \t
^{(n)}_n$ can be read from the reduced tree $ \t^{*n}$
that consists only of vertices of $\t^{(n)}$ that have descendants at
generation $n$. In other words, we can chop off the branches of the
discrete tree that do not reach the level $n$. Indeed, a simple
argument shows that the hitting distribution
of generation $n$ is the same for simple random walk on $\t^{(n)}$ and
on the reduced tree $\t^{*n}$. See Figure~\ref{fig1} for a simulation of
a large Galton--Watson tree and the associated reduced tree.

The scaling limit of the discrete reduced trees $\t^{*n}$ (when
distances are scaled by the factor $n^{-1}$) is particularly simple.
We define a random compact $\R$-tree by the following device. We start
from an (oriented) line segment whose
length $U_\varnothing$ is uniformly distributed over $[0,1]$ and whose origin
will serve as the root of our tree. At the other end of this initial
line segment, we attach the initial point of two other line segments with
respective lengths
$U_1$ and $U_2$ such that, conditionally given $U_\varnothing$, $U_1$
and $U_2$ are independent and uniformly distributed over
$[0,1-U_\varnothing]$.
At the other end of the first of these segments, respectively, of the
second one, we attach two line segments whose lengths are
again independent and uniformly distributed over $[0,1-U_\varnothing
-U_1]$, respectively, over $[0,1-U_\varnothing-U_2]$,
conditionally on the triplet $(U_\varnothing,U_1,U_2)$. We continue the
construction by induction and after an infinite
number of steps we get a random (noncompact) rooted $\R$-tree, whose
completion is denoted by $\Delta$. This is the scaling limit of the
discrete reduced trees~$\t^{*n}$. See Section~\ref{sec:treedelta} for a
more precise construction.

The metric on $\Delta$ is denoted by $ \mathbf{d}$. By definition, the
boundary $ \partial\Delta$ consists of
all points of $\Delta$ at height $1$, that is, at distance $1$ from
the root: these are exactly the points that are added when taking
the completion in the preceding construction.

It is then easy to define Brownian motion on $\Delta$ starting from the
root and up to the first hitting
time of $\partial\Delta$
(it would be possible to extend the definition of Brownian motion
beyond the first hitting time
of $\partial\Delta$, but this is not relevant to our purposes). Roughly
speaking, this process behaves like linear Brownian motion as long as
it stays on an ``open interval'' of the tree. It is reflected at the
root of the tree and when it arrives at a branching point, it chooses
each of the three possible line segments incident to this point with
equal probabilities. The harmonic measure $\mu$ is then the (quenched)
distribution of the first hitting point of $ \partial\Delta$ by
Brownian motion (see Section~\ref{sec:treedelta} for details).

%
\begin{theorem}
\label{thm:main-harmonic}
With the same constant $\beta$ as in Theorem~\ref{thm:maindiscrete}, we
have $\P$ a.s.,
$\mu(\D x)$~a.e.,
\[
\lim_{r\downarrow0} \frac{\log\mu(\mathscr{B}_\bd(x,r))}{\log r} =\beta,
\]
where $\mathscr{B}_\bd(x,r)$ stands for the closed ball of radius $r$
centered at $x$
in the metric space $(\Delta,\bd)$.
Consequently,
the Hausdorff dimension
of $\mu$ is $\P$ a.s. equal to $\beta$.
\end{theorem}

The fact that the second assertion of the theorem follows from the
first one
is standard. See, for example, Lemma 4.1 in \cite{LPP95}.
The simulation in Figure~\ref{fig2} illustrates the fractal behavior of the measure~$\mu$.

\begin{rem*}
It is not hard to prove that the Hausdorff dimension of
$\partial\Delta$ (with respect to $\bd$) is a.s.  equal to $1$.
An exact Hausdorff measure function is given by
Theorem 1.3 in Duquesne and Le Gall \cite{DLG06}.
\end{rem*}

%
\begin{figure}[t]

\includegraphics{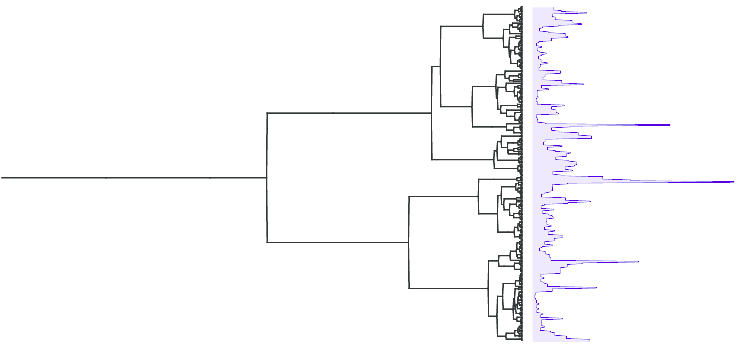}

\caption{A simulation of the reduced tree $\Delta$ and the harmonic
measure on its boundary. Clearly the measure is not uniformly spread
and exhibits a fractal behavior.}\label{fig2}
\end{figure}

Let us give some ideas of the proof of Theorem~\ref
{thm:main-harmonic}. It is well known that one can turn the tree
${\Delta}$,
or rather the subtree $\Delta\setminus\partial\Delta$, into a
``stationary'' object via a logarithmic transformation. Roughly
speaking, we introduce a new tree which has the
same binary branching structure as $\Delta$, such that each point of
$\Delta$ at height $s \in[0,1)$
corresponds to a point of the new tree at height $ -\log(1-s) \in
[0,\infty)$. The resulting noncompact tree is called the
Yule tree because it describes the genealogy of the classical Yule
process, where individuals have
(independent) exponential lifetimes with parameter $1$ and each
individual has exactly two offspring.
We define the boundary of the Yule tree as the collection of
all its geodesic rays, where a geodesic ray is just a semi-infinite
geodesic path starting from the root. This
boundary is easily identified with $\partial\Delta$.
An application of It\^o's formula shows that the
logarithmic transformation turns Brownian motion on $ {\Delta}$ into a
time-changed Brownian motion \emph{with {drift} $1/2$}
toward infinity on the Yule tree. Consequently, the probability
measure $\mu$ corresponds via
the preceding transformation to the distribution $\nu$ of the geodesic
ray that is ``selected'' by Brownian motion
with drift $1/2$ (i.e., the unique ray of the Yule tree that is
visited by Brownian motion at arbitrarily large times).
The first assertion of Theorem~\ref{thm:main-harmonic} is then
equivalent to proving that, $\P$ a.s., $\nu(\D y)$ a.e.,
\begin{equation}
\label{equivalent-asymp} \lim_{r\to\infty} \frac{1}{r} \log\nu \bigl({
\mathcal B}(y,r) \bigr) = -\beta,
\end{equation}
where ${\mathcal B}(y,r)$ denotes the set of all geodesic rays of the
Yule tree that coincide with $y$ up to height $r$.

The next step is then to identify a kind of ``stationary environment
seen from the particle''
for Brownian motion on the Yule tree. More precisely, we show in
Section~\ref{sec:subtreeabove} that the law of the subtree above level
$r \geq0$ that is selected by Brownian motion (with drift $1/2$)
converges as $r \to\infty$ to a
limiting probability measure that we explicitly describe. This allows
us to construct an ergodic invariant measure
for the natural shifts on the space of all pairs consisting of a
(deterministic) Yule-type tree and a distinguished
geodesic ray on this tree, and moreover this measure is absolutely
continuous with respect to the law of
the random pair formed by
the Yule tree and the ray selected by Brownian motion. The limiting
result \eqref{equivalent-asymp} then
follows from an application of Birkhoff's ergodic theorem to a
suitable functional. In this part of our work, we use several ideas
that have been
developed by Lyons, Pemantle and Peres \cite{LPP95} in a slightly
different setting.

\subsection*{The random conductance} The constant $\beta$ in Theorems~\ref{thm:maindiscrete} and~\ref{thm:main-harmonic} can be expressed in terms of the
(continuous) conductance of $ {\Delta}$. Roughly speaking, if one
considers $\Delta$ as a network of resistors with unit resistance per
unit length, then the effective resistance between height $0$ and
height $1$ is a random variable, which we denote by $ \mathcal{C}$.
With this interpretation, it is clear that $\mathcal{C}>1$ a.s.
Alternatively, $ \mathcal{C}$ is the mass under the Brownian excursion
measure from the root of those excursion paths that hit height $1$.
Note that $ \mathcal{C}$ is also the limit in distribution of the
(scaled) conductance between generations $0$ and $n$ in $ \t^{(n)}$.
The distribution of $ \mathcal{C}$ satisfies the following recursive equation:
\begin{eqnarray}
\label{eq:rde} \mathcal{C} & \stackrel{(\mathrm{d})} {=} & \biggl(U +
\frac
{1-U}{ \mathcal{C}_1+ \mathcal{C}_2} \biggr)^{-1},
\end{eqnarray}
where $ \mathcal{C}_1$ and $ \mathcal{C}_2$ are independent copies of $
\mathcal{C}$, and $U$ is uniformly distributed over $[0,1]$ and
independent of the pair $(\cc_1,\cc_2)$. Despite this rather simple
recursive equation, the law $\gamma(\D s)$ of $ \mathcal{C}$ is not
completely understood (in particular, its mean is unknown). We prove that,
although $\gamma$ has a continuous density $f$ over $[1,\infty)$, the
function $f$ is not twice continuously differentiable at the point $2$
(and we expect a similar singular behavior at all integer values). See
Figure~\ref{fig:density}.

%
\begin{figure}[t]

\includegraphics{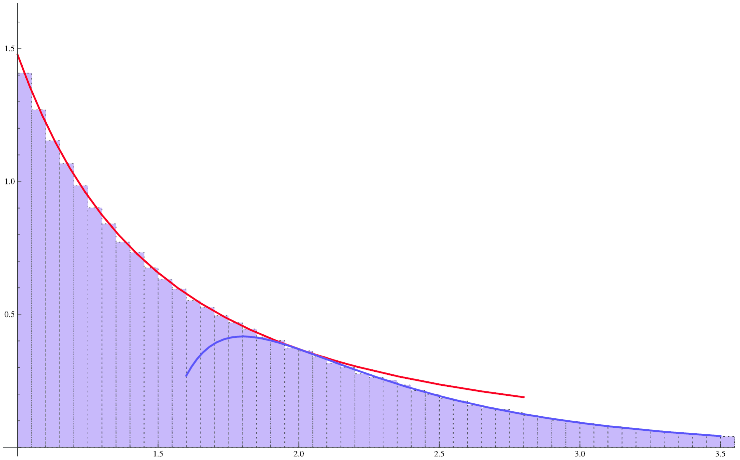}

\caption{A histogram of the distribution of $\gamma$ over $(1,\infty)$
from simulations based on
the recursive equation \protect\eqref{eq:rde}. There are explicit
formulas for the density of $\gamma$ over $[1,2]$
and over $[2,3]$, which however depend on the (unknown) density at $1$.
The red and the blue curves correspond to these
explicit formulas, with a numerical approximation of the density at $1$.}
\label{fig:density}
\end{figure}

In many respects, the distribution $\gamma$ governs the behavior of
harmonic measure.
In particular the constant $\beta$ has an explicit expression in terms
of $\gamma$.

%
\begin{proposition}
\label{prop:valuebeta}
The distribution $\gamma$ is characterized
in the class
of all probability measures on $[1,\infty)$ by the distributional
equation \eqref{eq:rde}. The constant $\beta$ appearing in Theorems~\ref
{thm:maindiscrete} and~\ref{thm:main-harmonic} is given by
\begin{eqnarray}
\label{value-beta} \beta &=& 2 \frac{\int\!\!\int\!\!\int\gamma(\D r)\gamma(\D
s)\gamma(\D t)\sklafrac{rs}{r+s+t-1}\log(\vfrac{s+t}{s})}{ \int\!\!\int\gamma
(\D s)\gamma(\D t) \sklafrac{st}{s+t-1}}
\nonumber
\\[-8pt]
\\[-8pt]
\nonumber
&=& \frac{1}{2} \biggl( \frac{(\int\gamma(\D s) s)^2}{ \int\!\!\int\gamma(\D
s) \gamma(\D t) \sklafrac{st}{s+t-1}}-1 \biggr).
\end{eqnarray}
\end{proposition}

We finally mention that some extensions of the results of the present
work have been obtained by Lin \cite{Lin,Lin2}.
A version of Theorem \ref{thm:maindiscrete} for Galton--Watson trees
whose offspring distribution belongs to the domain of attraction
of a stable law with index $\alpha\in(1,2)$ is derived in \cite{Lin}.
In the setting of the present work, the article \cite{Lin2}
gives an analog
of Theorem~\ref{thm:maindiscrete} for the harmonic measure of a vertex
chosen according to the {uniform} probability measure on generation
$n$ of the tree $\t^{(n)}$. This is another step toward a full
multifractal analysis of the harmonic measure $\mu_n$.

The paper is organized as follows. We start by studying the continuous
model. In Section~\ref{sec:continuoussetting}, we
introduce the basic set-up and we relate the random tree $ \Delta$ to
the Yule tree. The law of the random conductance $ \mathcal{C}$ is
studied in Section~\ref{sec:conductance}.
Section~\ref{sec:stationary+ergodic} gathers the
ingredients of the proof of Theorem~\ref{thm:main-harmonic}. In
particular, Section~\ref{sec:asymptotics} identifies
the limiting distribution of the subtree above level $r$ selected by
Brownian motion, and
Section~\ref{sec:ergodic} explains the application of the ergodic
theorem needed to derive \eqref{equivalent-asymp}. Section~\ref{sec:continuous->discrete} is devoted to the proof of Theorem~\ref
{thm:maindiscrete} and Corollary \ref{cor:planetree}. Let us emphasize
that Theorem \ref{thm:maindiscrete} is not a straightforward
consequence of
Theorem~\ref{thm:main-harmonic}, and that the proof of our results in
the discrete setting requires
a number of additional estimates, even though a key role is played by
Theorem~\ref{thm:main-harmonic}.
The last section is devoted to a few complements. In particular, we
comment on the connection between the present paper and the recent work
of A\"\i d\'ekon \cite{Aid11}.

\section{The continuous setting} \label{sec:continuoussetting}
\renewcommand{\t}{\mathcal{T}}
In this section, we give a formal definition of the (continuous)
reduced tree $\Delta$. We then explain the connection between the
reduced tree and the Yule tree. We finally introduce and study the
conductance of these trees, which plays a key role in the next sections.

\subsection{The reduced tree \texorpdfstring{${\Delta}$}{Delta}}
\label{sec:treedelta}

We set
\[
\mathcal{V} = \bigcup_{n=0}^\infty\{1,2
\}^n,
\]
where $\{1,2\}^0=\{\varnothing\}$. If $v=(v_1,\ldots,v_n)\in\mathcal
{V}$, we set $|v|=n$ (in particular, $|\varnothing|=0$), and if $n\geq1$,
we define the parent
of $v$ as $\widehat v=(v_1,\ldots,v_{n-1})$ (we then say that $v$ is a
child of $\widehat v$). If $v=(v_1,\ldots,v_n)$ and $v'=(v'_1,\ldots
,v'_m)$ belong to $\mathcal{V}$, the concatenation of
$v$ and $v'$ is $vv':=(v_1,\ldots,v_n,v'_1,\ldots,v'_m)$. The notions
of a descendant and an ancestor of an element of
$\v$ are defined in the obvious way, with the convention that a vertex
$v\in\v$ is both an ancestor and
a descendant of itself. If $v,w\in\v$, $v\wedge w$ is the unique
element of $\v$ that is an ancestor of both $v$ and $w$
and such that $|v\wedge w|$ is maximal.

We then consider a collection
\[
(U_v)_{v\in\mathcal{V}}
\]
of independent real random variables uniformly distributed over $[0,1]$
under the
probability measure $\P$. We set
\[
Y_\varnothing=U_\varnothing
\]
and then, by induction, for every $v\in\{1,2\}^n$, with $n\geq1$,
\[
Y_v = Y_{\widehat v} + U_v(1- Y_{\widehat v}).
\]
Note that $0\leq Y_v< 1$ for every
$v\in\mathcal{V}$, a.s. Consider then the set
\[
\Delta_0:= \bigl(\{\varnothing\}\times[0,Y_\varnothing] \bigr)
\cup \biggl(\bigcup_{v\in\mathcal{V}\setminus\{\varnothing\}} \{v\} \times(Y_{\widehat
v},
Y_v] \biggr).
\]
There is a straightforward way to define a metric $\bd$ on $\Delta_0$,
so that
$(\Delta_0,\bd)$ is a (noncompact) $\R$-tree and, for every
$x=(v,r)\in\Delta_0$, we have $\bd((\varnothing,0), x)=r$. To be
specific, let $x=(v,r)\in\Delta_0$ and $y=(w,r')\in\Delta_0$:
\begin{itemize}
\item If $v$ is a descendant of $w$ or $w$ is a descendant
of $v$, we set $\bd(x,y)= |r-r'|$.
\item Otherwise, $\bd(x,y)= \bd((v\wedge w,Y_{v\wedge
w}),x)+ \bd((v\wedge w,Y_{v\wedge w}),y)
=(r-Y_{v\wedge w})+(r'-Y_{v\wedge w})$.
\end{itemize}
See Figure~\ref{fig:Delta} for an illustration of the tree $\Delta_0$.

%
\begin{figure}[t]

\includegraphics{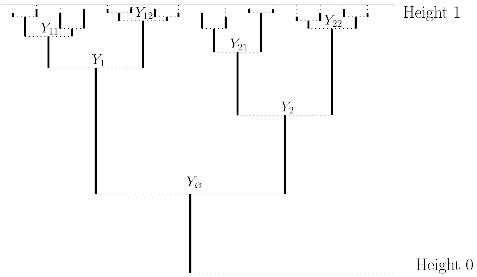}

\caption{The random tree $\Delta_0$.}\label{fig:Delta}
\end{figure}

We let $\Delta$ be the completion of
$\Delta_0$ with respect to the metric $\bd$. Then
\[
\Delta=\Delta_0 \cup\partial\Delta,
\]
where by definition $\partial\Delta=\{x\in\Delta:\bd((\varnothing,0),
x)=1\}$, which is canonically identified with $\{1,2\}^\N$
(here and below, $\N=\{1,2,\ldots\}$ is the set of all positive
integers). Note that
$(\Delta,\bd)$ is a compact $\R$-tree.

The point $(\varnothing,0)$ is called the root of $\Delta$. For every
$x\in\Delta$, we set $H(x)=\bd((\varnothing,0), x)$
and call $H(x)$ the height of $x$.
We can define a genealogical order on $\Delta$ by setting $x\prec y$ if
and only if
$x$ belongs to the geodesic path from the root to $y$.

For every $\ve\in(0,1)$, we set
\[
\Delta_\ve= \bigl\{x\in\Delta: H(x)\leq1-\ve \bigr\},
\]
which is also a compact $\R$-tree for the metric $\bd$.
The leaves of $\Delta_\ve$ are the points of the form $(v,1-\ve)$ for
all $v\in\mathcal{V}$ such that
$Y_{\widehat v}< 1-\ve\leq Y_v$. The branching points of $\Delta_\ve$
are the points of the form
$(v,Y_v)$ for all $v\in\mathcal{V}$ such that \mbox{$Y_v<1-\ve$}. We can then
define Brownian motion on $\Delta_\ve$ as a special case of a diffusion
on a graph (see in particular \cite{FS00,EK01} and the
references therein, and note that the definition of Brownian motion on
$\Delta_\ve$ can also be
viewed as a very special case of the construction of Brownian motion on
$\R$-trees given in \cite{AEW}).
Informally, this process behaves like linear Brownian motion as long as
it stays on an ``open interval'' of the
form $\{v\}\times(Y_{\widehat v},Y_v\wedge(1-\ve))$. It is reflected
at the root $(\varnothing,0)$ and at the leaves
of $\Delta_\ve$, and when it arrives at a branching point of the tree,
it chooses each of the three possible
line segments ending at this point with equal probabilities.

Write $B^\ve=(B^\ve_t)_{t\geq0}$ for Brownian motion on $\Delta_\ve$
starting from the root, which is defined
under the probability measure $P$ (for our purposes, it will be
important to carefully distinguish the
probability measure $ \mathbb{P}$ governing the random trees and the
one governing Brownian motions on
these trees). We
let
\[
T_\ve:= \inf \bigl\{t\geq0: H \bigl(B^\ve_t
\bigr)= 1-\ve \bigr\},
\]
be the hitting time of the set of all leaves of $\Delta_\ve$.

If we now set $\ve_n=2^{-n}$ for every $n\geq1$, we may define all
processes $B^{\ve_n}$ on the same
probability space, in such a way that $B^{\ve_n}_{t\wedge T_{\ve_m}}=
B^{\ve_m}_{t\wedge T_{\ve_m}}$
for every $t\geq0$ and every choice of $m\leq n$, $P$ a.s. Assuming
that the latter property holds, we set
\[
T=\lim_{n\uparrow\infty}\uparrow T_{\ve_n}
\]
and we define the process
$(B_t)_{t\geq0}$ by requiring that $B_t=\dagger$ if $t\geq T$ (where
$\dagger$ is a cemetery point)
and, for every $n\geq1$ and $t\geq0$, $B_{t\wedge T_{\ve_n}}=B^{\ve
_n}_{t\wedge T_{\ve_n}}$. It is easy to verify that the left limit
\[
B_{T-}= \lim_{t\uparrow T,t<T} B_t
\]
exists in $\Delta$ and belongs to $\partial\Delta$, $P$ a.s. The
harmonic measure
$\mu$ is the distribution of $B_{T-}$ under $P$, which is a (random)
probability measure on $\partial\Delta=\{1,2\}^\N$.

\subsection{The Yule tree}
\label{sec:yuletree}

For the proof of Theorem~\ref{thm:main-harmonic}, it will be more convenient
to reformulate the problem in terms of Brownian motion on the Yule tree.
To define the Yule tree, consider now a collection
\[
(V_v)_{v\in\mathcal{V}}
\]
of independent real random variables exponentially distributed with
mean $1$ under the
probability measure $\P$. We set
\[
\mathcal{Y}_\varnothing=V_\varnothing
\]
and then by induction, for every $v\in\{1,2\}^n$, with $n\geq1$,
\[
\mathcal{Y}_v = \mathcal{Y}_{\widehat v} + V_v.
\]
The Yule tree is the set
\[
\Gamma:= \bigl(\{\varnothing\}\times[0,\mathcal{Y}_\varnothing] \bigr) \cup
\biggl(\bigcup_{v\in\mathcal{V}\setminus\{\varnothing\}} \{v\} \times (\mathcal{Y}_{\widehat v},
\mathcal{Y}_v] \biggr),
\]
which is equipped with the metric $d$ defined in the same way as $\bd$
in the preceding
section. For this metric, $\Gamma$ is again a noncompact $\R$-tree.
For every $x=(v,r)\in\Gamma$, we keep the notation
$H(x)=r=d((\varnothing,0),x)$ for the height of the point $x$.

Now observe that if $U$ is uniformly distributed over $[0,1]$, the
random variable
$-\log(1-U)$ is exponentially distributed with mean $1$.
Hence, we may and will suppose that the collection $(V_v)_{v\in\mathcal
{V}}$ is constructed from the collection
$(U_v)_{v\in\mathcal{V}}$ in the previous section via
the formula $V_v=-\log(1-U_v)$, for every $v\in\mathcal{V}$. Then the mapping
$\Psi$ defined on $\Delta_0$ by $\Psi(v,r)=(v,-\log(1-r))$, for every
$(v,r)\in\Delta_0$,
is a homeomorphism from $\Delta_0$ onto $\Gamma$.

Stochastic calculus shows that we can write, for every $t\in[0,T)$,
\begin{equation}
\label{BM-Yule-reduced} \Psi(B_t) =W \biggl(\int_0^t
\bigl(1-H(B_s) \bigr)^{-2} \,\D s \biggr),
\end{equation}
where $(W(t))_{t\geq0}$ is Brownian motion with constant drift $1/2$ toward
infinity
on the Yule tree (this process is defined in a similar way as Brownian motion
on $\Delta_\ve$, except that it behaves like Brownian motion with drift $1/2$
on every ``open interval'' of the tree). Note that $W$ is again defined under
the probability measure $P$. From now on, when we speak about Brownian motion
on the Yule tree or on other similar trees, we will always mean Brownian
motion with drift $1/2$ toward infinity.

By definition, the boundary of $\Gamma$ is the set of all infinite
geodesics in $\Gamma$ starting from the root $(\varnothing,0)$ (these are
called geodesic rays). The boundary of $\Gamma$ is canonically
identified with $\{1,2\}^\N$. From the transience of
Brownian motion on $\Gamma$, there is an a.s. unique geodesic ray
denoted by
$W_\infty$ that is visited
by $(W(t),t\geq0)$ at arbitrarily large times. We sometimes say that
$W_\infty$ is the exit ray of Brownian motion on $\Gamma$. The
distribution of $W_\infty$ under $P$ yields
a probability measure $\nu$ on $\{1,2\}^\N$. Thanks to \eqref
{BM-Yule-reduced}, we have in
fact $\nu=\mu$, provided we view both $\mu$ and $\nu$
as (random) probability measures on $\{1,2\}^ \mathbb{N}$. The
statement of Theorem~\ref{thm:main-harmonic} is then reduced to
checking that \eqref{equivalent-asymp} holds $\nu(\D y)$ a.e., $\P$ a.s.

\subsection*{Yule-type trees}
Our proof of \eqref{equivalent-asymp}
makes a heavy
use of tools of ergodic theory applied to certain transformations on a space
of trees that we now describe.
We let $\T$ be the set of all collections $(z_v)_{v\in\v}$
of nonnegative real numbers such that the following properties hold:
\begin{longlist}[(ii)]
\item[(i)] $z_{\widehat v}< z_v$ for every $v\in\v\setminus\{
\varnothing\}$;
\item[(ii)] for every $\mathbf{v}=(v_1,v_2,\ldots)\in\{1,2\}^\N$,
\[
\lim_{n\to\infty} z_{(v_1,\ldots,v_n)} =+\infty.
\]
\end{longlist}
Notice that we allow the possibility that $z_\varnothing=0$.
We equip $\T$ with the $\sigma$-field generated by the coordinate mappings.
If $(z_v)_{v\in\v}\in\T$, we can consider the associated ``tree''
\[
\t:= \bigl(\{\varnothing\}\times[0,z_\varnothing] \bigr) \cup \biggl(\bigcup
_{v\in
\mathcal{V}\setminus\{\varnothing\}} \{v\} \times(z_{\widehat v},
z_v] \biggr),
\]
equipped with the distance defined as above. We will keep the notation
$H(x)=r$ if $x=(v,r)$ for the height of
a point $x\in\t$. The genealogical order on $\t$ is defined as
previously and will again be denoted by $\prec$.
If $\mathbf{u}=(u_1,u_2,\ldots,u_n,\ldots)\in\{1,2\}^\N$, and
$x=(v,r)\in\t$, we write $x\prec\mathbf{u}$
if $v=(u_1,u_2,\ldots,u_k)$ for some integer $k\geq0$.

We will often abuse notation and say that we consider a tree $\t\in\T
$: this really
means that we are given a collection $(z_v)_{v\in\v}$ satisfying the
above properties, and we consider the associated tree $\t$.
In particular, $\t$ has an order structure (in addition to the
genealogical partial order)
given by the lexicographical order on $\v$.
Elements of $\T$ will be called Yule-type trees.

Clearly, the Yule tree can be viewed as a random variable with values
in $\T$, and
we write $\Theta(\D\t)$ for its distribution.

Let us fix $ \t\in\mathbb{T}$. If $r>0$, the level set at height $r$ is
\[
\t_r= \bigl\{x\in\t: H(x)=r \bigr\}. %
\]
If $x\in\t_r$, we can consider the subtree
$\t[x]$ of descendants of $x$ in $\t$. Formally, we view $\t[x]$ as an
element of $\T$: we write $v_x$ for the unique element of $\v$ such
that $x=(v_x,r)$, and define $\t[x]$ as the
Yule-type tree corresponding to the collection $(z_{v_xv}-r)_{v\in\v}$.
Similarly, if $\tllbracket0, x\trrbracket$ denotes the
geodesic segment between the root and $x$, we can define the subtrees
of $\t$
branching off $\tllbracket0, x\trrbracket$. To this end, let $n_x=|v_x|$
and let
$v_{x,0}=\varnothing, v_{x,1},\ldots, v_{x,n_x}=v_x$ be the successive
ancestors of
$v_x$ from generation $0$ to generation $n_x$. Set $r_{x,i}=
z_{v_{x,i-1}}$ for every $1\leq i\leq n_x$. Then,
for every $1\leq i\leq n_x$, the $i$th subtree branching off
$\tllbracket0, x\trrbracket$, which is denoted by
$\t_{x,i}$, corresponds to the collection
\[
(z_{\tilde v_{x,i}v}- r_{x,i} )_{v\in\v},
\]
where $\tilde v_{x,i}$ is the child of $v_{x,i-1}$ that is not
$v_{x,i}$. To simplify notation, we introduce the
point measure
\[
\xi_{r,x}(\t)= \sum_{i=1}^{n_x}
\delta_{(r_{x,i},\t_{x,i})},
\]
which belongs to the set $\mathcal{M}_p(\R_+\times\T)$ of
all finite point measures on $\R_+\times\T$.

%
\begin{figure}[b]

\includegraphics{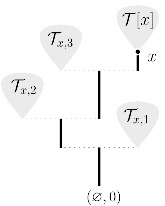}

\caption{The spine decomposition.}\label{fig5}
\end{figure}

We now state a ``spine'' decomposition of the Yule tree, which plays an
important role in our approach. See Figure~\ref{fig5} for an illustration of this decomposition.

%
\begin{proposition}[(Spine decomposition)]\label{spine-decomposition}
Let $F$ be a nonnegative measurable function on $\T$, and let $G$ be a
nonnegative
measurable function on $\mathcal{M}_p(\R_+\times\T)$. Let $r>0$. Then
\[
\E \biggl[ \sum_{x\in\Gamma_r} F \bigl(\Gamma[x] \bigr) G
\bigl(\xi_{r,x}(\Gamma) \bigr) \biggr] = e^r \E \bigl[F(
\Gamma) \bigr] \times\E \bigl[G(\mathcal{N}) \bigr],
\]
where ${\mathcal N}(\D s \,\D\t)$ is, under the probability measure $\P
$, a Poisson point measure on
$\mathcal{M}_p(\R_+\times\T)$ with intensity
$2 {\mathbf1}_{[0,r]}(s)\, \D s  \Theta(\D\t)$.
\end{proposition}

This result is part of the folklore of the subject (see Theorem 2 in
\cite{CRW} for essentially the same Palm decomposition in the
more general setting where branching is combined with spatial motion),
and is closely related to the spine decomposition
of size-biased Galton--Watson trees in the discrete setting (see, e.g.,
\cite{LP10}, Section~12.1). For the reader's convenience,
we sketch a proof of Proposition \ref{spine-decomposition} in the
\hyperref[appen]{Appendix} below. This proof is based on
a relation between the continuous reduced tree $\Delta$ and the
Brownian excursion conditioned
to hit level $1$, which is recalled in Section~\ref{sec:convergences}
below (see in particular Figure~\ref{redu-excursion}).

\subsection{The continuous conductance}
\label{sec:conductance}
Before we proceed to the proof of Theorem~\ref{thm:main-harmonic}, we
will define and study the continuous conductance
$\mathcal{C}$ of
the tree $ \Delta$, which plays a major role in this proof.
Informally, the random variable $\mathcal{C}$ is defined
by viewing the random tree $\Delta$ as a network of ideal resistors
with unit resistance per unit of length
and letting $\mathcal{C}$ be the conductance between the root and the
set $\partial\Delta$ in this network.
We will give a more formal definition using excursion measures of
Brownian motion. To this end, and
in view of further applications in the next section, we first define
the excursion measure
on a (deterministic) Yule-type tree.

So let $\t\in\T$, and consider the associated collection $(z_v)_{v\in
\mathcal{V}}$ as explained in
the preceding section. We suppose that $z_\varnothing>0$.
We write
$C(\R_+,\t)$ for the set of all
continuous functions from $\R_+$ into $\t$. We also let $\mathcal{E}_\t
$ be the subset of $C(\R_+,\t)$ consisting
of all ``excursions'' in $\t$: an element $\omega$ of $C(\R_+,\t)$
belongs to $\mathcal{E}_\t$ if and only if
$\omega(0)= (\varnothing,0)$ and there
exists a number $\zeta(\omega)\in(0,\infty]$ such that $\omega(t)\neq
(\varnothing,0)$ if and only
if $0<t<\zeta(\omega)$. For every $r\geq0$ and $\omega\in C(\R_+,\t)$, set
\[
T_r(\omega):=\inf \bigl\{t\geq0:H \bigl(\omega(t) \bigr)= r \bigr\},
\]
where we recall that $H(\omega(t))$ is the height (or distance from the
root) of $\omega(t)$, and we make the usual convention $\inf\varnothing
=\infty$.
For every
$\ve\in(0,z_\varnothing)$, there is a unique $y_\ve\in\t$ whose height is
equal to $\ve$. Let $n_{\t,\ve}$ be the law of Brownian motion on $\t$
with drift $1/2$
started from $y_\ve$ and stopped when it hits the root $(\varnothing
,0)$ (if this event occurs). Then
$n_{\t,\ve}$ is a probability measure on the space $C(\R_+,\t)$. If
$0<\ve'<\ve$ an application of the strong Markov property
shows that the distribution of $(\omega(T_\ve(\omega)+t))_{t\geq0}$
under $n_{\t,\ve'}(\cdot\mid T_\ve<\infty)$
is $n_{\t,\ve}$. Furthermore, $n_{\t,\ve'}(T_\ve<\infty)=(1-e^{-\ve
'})/(1-e^{-\ve})$, by the formula for the
scale function of linear Brownian motion with drift. From
these properties, it is an easy exercise to verify that the measures
$\ve^{-1}n_{\t,\ve}$
converge when $\ve\to0$ toward a $\sigma$-finite measure $n_\t$ on
the set $\mathcal{E}_\t$ of all excursions in $\t$. The convergence
holds in
the sense that
\[
\ve^{-1} n_{\t,\ve} \bigl(g_1 \bigl(
\omega(t_1) \bigr)\cdots g_p \bigl(
\omega(t_p) \bigr) \bigr)\mathop{\la}\limits_{\ve\to0}
n_\t \bigl(g_1 \bigl(\omega(t_1) \bigr)
\cdots g_p \bigl(\omega(t_p) \bigr) \bigr)
\]
for every choice of $0<t_1<\cdots<t_p$ and of the bounded continuous
functions $g_1,\ldots,g_p$ on $\t$ that vanish on
a neighborhood of $(\varnothing,0)$ in $\t$.
Alternatively, the measure $n_\t$ is the unique $\sigma$-finite measure
on $\mathcal{E}_\t$ such that, for every $\ve>0$, one has
$n_\t(T_\ve<\infty)= (1-e^{-\ve})^{-1}$ and the law of $(\omega(T_\ve
(\omega)+t))_{t\geq0}$
under $n_\t(\cdot\mid T_\ve<\infty)$ is $n_{\t,\ve}$.

The measure $n_\t$ is
called the excursion measure of Brownian motion (with drift $1/2$) in
the tree $\t$. The preceding construction is an analog
of a classical construction of the It\^o excursion measure of linear
Brownian motion; see, for example,
\cite{RY}, Chapter XII. Of course, it is also a special case of the
definition of the excursion measure of a general Markov process
from a regular point (see Blumenthal~\cite{Bl}).

The conductance $\mathcal{C}(\t)$ is then defined by
\[
\mathcal{C}(\t)=n_\t(\zeta=\infty) = \lim_{\ve\to0}
\ve^{-1} n_{\t,\ve
}(T_0=\infty).
\]
Note that we have $1\leq\cc(\t)\leq(1-e^{-z_\varnothing})^{-1}<\infty
$. The bound $\mathcal{C}(\t)\geq1$ is obtained by saying that
$\mathcal{C}(\t)$ is greater than the conductance of the trivial tree
that consists only of a half-line. The other
bound follows from
the form of the scale function
of Brownian motion with drift, which yields an explicit expression for
the probability under $n_{\t,\ve}$
that the process comes back to $0$ before hitting the first branching point.

To simplify notation, we set $\mathcal{C}=\mathcal{C}(\Gamma)$, which
is a random variable
with values in $[1,\infty)$.
Because of the relations between the Yule tree~$\Gamma$ and the reduced
tree~$\Delta$, the
random conductance $\mathcal{C}$
may also be defined
as the mass assigned by the excursion measure of Brownian motion on
$\Delta$ (away from the root), to the set of trajectories
that reach height $1$ before coming back to the root.

The distributional identity \eqref{eq:rde} is obvious from the electric
network interpretation: just view
$\Delta$ as a series of two conductors, the first one being a segment
of length $U$ and the
second one consisting of two independent copies of $\Delta$ (scaled by
the factor $1-U$) in parallel.
Alternatively, it is also easy to derive \eqref{eq:rde} from the
probabilistic definition in terms
of excursion measures, by applying the strong Markov property at the
hitting time of the
first node of the tree. We leave the details to the reader.

Let us now prove that \eqref{eq:rde} characterizes the law of $
\mathcal{C}$ and discuss some of the properties
of this law. For $u \in(0,1)$ and $x,y \geq1$, we define
\begin{equation}
\label{eq:G} G(u,x,y):= \biggl( u + \frac
{1-u}{x+y} \biggr)^{-1},
\end{equation}
so that \eqref{eq:rde}
can be rewritten as
\begin{equation}
\label{recursive-bis} \mathcal{C}\stackrel{\rm(\mathrm{d})}{=} G(U,
\mathcal{C}_{1}, \mathcal{C}_{2}),
\end{equation}
where $U,\mathcal{C}_1,\mathcal{C}_2$ are as in \eqref{eq:rde}. Let $
\mathscr{M}$ be the set of all probability measures on $[1, \infty]$
and let $ \Phi: \mathscr{M} \to \mathscr{M}$ map a distribution
$\lambda$ to
\[
\Phi( \lambda) = \mathsf{Law} \bigl(G(U,X_{1},X_{2})
\bigr),
\]
where $X_{1}$ and $X_{2}$ are independent and distributed according to
$\lambda$ and $U$ is uniformly distributed over $[0,1]$
and independent of the pair $(X_1,X_2)$.

%
\begin{proposition} \label{prop:unicite}The law $\gamma$ of $\mathcal
{C}$ is the unique
fixed point of the mapping $\Phi$ on $\mathscr{M} $, and we have $\Phi
^k(\lambda) \to\gamma$ weakly as $k \to\infty$, for every $\lambda
\in\mathscr{M}$. Furthermore
all moments of $\gamma$ are finite, and $\gamma$ has a continuous
density over $[1,\infty)$.
Finally, the Laplace transform
\[
\varphi(\ell) = \E \bigl[\exp(-\ell \mathcal{C}/2) \bigr] =\int
_1^\infty e^{-\ell
r/2} \gamma(\D r),\qquad\ell
\geq0
\]
solves the differential equation
\begin{equation}
2\ell \varphi''(\ell) + \ell\varphi'(
\ell) + \varphi ^2(\ell) - \varphi(\ell)=0. \label{edphi}
\end{equation}
\end{proposition}

\begin{rem*}
In \cite{LPP95}, the authors discuss the
conductance of an infinite supercritical Galton--Watson tree with
offspring distribution $\theta$. This conductance also satisfies a
recursive distributional equation, which depends on $\theta$. In that
setting, it is conjectured that the distribution of the conductance is
absolutely continuous with respect to Lebesgue measure at least if
$\theta(k)=0$ for all sufficiently large $k$; see \cite{L00,LP10}.
\end{rem*}

\begin{pf*}{Proof of Proposition~\ref{prop:unicite}}
We start with a few preliminary observations. If $\lambda
,\lambda' \in\mathscr{M,}$ we say that a random pair $(X,Y)$ is a
coupling of $\lambda$ and $\lambda'$ if $X$ is distributed according to
$\lambda$ and $Y$ is distributed according to $ \lambda'$. The
stochastic partial order $\preceq$ on $ \mathscr{M}$ is defined by
saying that $ \lambda\preceq\lambda'$ if
and only if there exists a coupling $(X,Y)$ of $\lambda$ and $\lambda
'$ such that $X \leq Y$ a.s. It is then clear that the mapping $\Phi$
is increasing for the stochastic partial order.

We endow the set $\mathscr{M}_{1}$ of all probability measures on $[1,
\infty]$ that have a finite first moment with the
$1$-Wasserstein metric
\[
\mathrm{d}_{1} \bigl( \lambda,\lambda' \bigr):= \inf
\bigl\{ E \bigl[\llvert X-Y\rrvert \bigr]: (X,Y) \mbox{ coupling of }
\bigl(\lambda, \lambda' \bigr) \bigr\}. %
\]
The metric space $( \mathscr{M}_{1}, \mathrm{d}_{1})$ is Polish and its
topology is finer than the weak topology on $ \mathscr{M}_{1}$. From
the easy bound $G(u,x,y)\leq x+y$, we immediately see
that $\Phi$ maps $ \mathscr{M}_{1}$ into $ \mathscr{M}_{1}$.
We then observe that the mapping $\Phi$ is strictly contractant on~$
\mathscr{M}_{1}$. To see this, let $(X_{1},Y_{1})$ and $(X_{2},Y_{2})$
be two
independent copies of a coupling between $\lambda,\lambda' \in\mathscr
{M}_{1}$ and let $U$ be uniformly distributed over $[0,1]$ and
independent of $(X_1,Y_1,X_2,Y_2)$. Then the two variables
$G(U,X_{1},X_{2})$ and $G(U,Y_{1},Y_{2})$ give a coupling of $ \Phi(
\lambda)$ and $ \Phi(\lambda')$. Using the fact that $X_{1},Y_{1},
X_{2},Y_{2} \geq1$, we have
\begin{eqnarray*}
&& \bigl\llvert G(U,X_{1},X_{2}) - G(U,Y_{1},Y_{2})
\bigr\rrvert
\\
&&\qquad= \biggl\llvert \biggl( U+ \frac{1-U}{X_1+X_2} \biggr)^{-1} -
\biggl( U+ \frac
{1-U}{Y_1+Y_2} \biggr)^{-1} \biggr\rrvert
\\
&&\qquad= \biggl\llvert \frac{(X_1+X_2
-Y_1-Y_2)(1-U)}{(U(X_1+X_2)+1-U)(U(Y_1+Y_2)+1-U)} \biggr\rrvert
\\
&&\qquad\leq \bigl(\llvert X_{1}-Y_{1}\rrvert +\llvert
X_{2}-Y_{2}\rrvert \bigr) \frac{1-U}{(1+U)^2}.
\end{eqnarray*}
Taking expected values and minimizing over the choice of the coupling
between $\lambda$ and $\lambda'$, we get $ \mathrm{d}_{1}( \Phi( \lambda
), \Phi(\lambda')) \leq2(1- \log(2))\,\mathrm{d}_{1}( \lambda,\lambda
')$. Since $ 2(1- \log(2)) <1$, the mapping $\Phi$ is contractant on $
\mathscr{M}_{1}$ and by completeness it has a unique fixed point $\gamma_0$
in $ \mathscr{M}_{1}$. Furthermore, for every $ \lambda\in\mathscr
{M}_{1}$, we have $ \Phi^k(\lambda) \to\gamma_0$ for the metric $
\mathrm{d}_{1}$, hence also weakly, as $k \to\infty$.

Since we know from \eqref{recursive-bis} that $\gamma$ is also a fixed point
of $\Phi$, the equality $\gamma=\gamma_0$ will follow if we can verify
that $\gamma_0$
is the unique fixed point of $\Phi$ in $\mathscr{M}$. To this end, it
will be
enough to verify that we have $ \Phi^k(\lambda) \to\gamma_0$ as $k\to
\infty$, for every $\lambda\in\mathscr{M}$.
Let $\lambda\in\mathscr{M}$
and for every $t\in\R$ set $F_\lambda(t)=\lambda([t,\infty])$. Also
set $F^{(2)}_{\lambda}(t) = P( X_{1}+X_{2} \geq t)$
where $X_1$ and $X_2$ are independent and distributed according to
$\lambda$. Then we have, for every
$t>1$,
\begin{eqnarray}
\label{eq:repart} F_{\Phi(\lambda)}(t) &=& P \biggl( U + \frac{1-U}{
{X}_{1} + {X}_{2}} \leq
t^{-1} \biggr)
\nonumber
\\
&=& P \biggl( U < t^{-1} \mbox{ and } \frac{t-Ut}{1-Ut}
\leq{X}_{1} + {X}_{2} \biggr)
\nonumber\\[-8pt]\\[-8pt]\nonumber
&=& \int_{0}^{1/t} \D u F^{(2)}_{\lambda}
\biggl( \frac
{t-ut}{1-ut} \biggr)
\nonumber
\\
&=& \frac{t-1}{t} \int_{t}^\infty
\frac{\D x}{(x-1)^2} F_{\lambda
}^{(2)}(x).\nonumber
\end{eqnarray}
It follows that, for every $t\geq1$,
\begin{equation}
\label{eq:major} F_{ \Phi(\lambda)} (t)\leq\frac
{F_{\lambda}^{(2)}(t)}{t} \leq
\frac{2F_{\lambda}(t/2)}{t}.
\end{equation}
We apply this to $\lambda=\Phi(\delta_\infty)$, where $\delta_\infty$
is the Dirac measure at $\infty$. We have $F_{\Phi(\delta_\infty)}
(t)=t^{-1}$, and it follows that, for every $t\geq1$,
\[
F_{ \Phi^2( \delta_{\infty})} (t) \leq \frac{4}{t^2}. %
\]
This implies that $\Phi^2( \delta_{\infty})\in\mathscr{M}_1$. By
monotonicity, we have also
$\Phi^2( \lambda)\in\mathscr{M}_1$ for every $\lambda\in\mathscr{M}$,
and from the preceding results
we get $ \Phi^k(\lambda) \to\gamma_0$ for every $\lambda\in\mathscr
{M}$. As explained above this
implies that $\gamma=\gamma_0$ is the unique fixed point of $\Phi$ in
$\mathscr{M}$.

Let us now check that all moments of $\gamma$ are finite.
To simplify notation, we write $F= F_{\gamma}$ and $F^{(2)}=
F^{(2)}_{\gamma}$. By \eqref{eq:repart}, we have for every $t > 1$,
\begin{equation}
\label{eq:repart0} F(t) = \frac{t-1}{t} \int_{t}^{\infty}
\frac{ \D x}{(x-1)^2} F^{(2)}(x)
\end{equation}
which implies that $F(t) \leq2 F(t/2)/t$ for every $t \geq1$, by the
same argument as above. Iterating this inequality, we get
that $F(t) \leq c_{1} \exp( - c_{2} (\log t)^2)$, with certain
constants $c_{1}, c_{2}>0$. It follows that
all moments of $\gamma$ are finite.

By construction, we have $F^{(2)}(t)=1$ for every $t\in[1,2]$. It then
immediately follows from
\eqref{eq:repart0} that we have
\begin{equation}
\label{eq:valueon[1,2]} F(t) = \frac{K_{0}}{t} + 1-K_0\qquad\forall t
\in[1,2],
\end{equation}
where
\[
K_0= 2-\int_2^{\infty}
\frac{ \D x}{(x-1)^2} F^{(2)}(x) \in[1,2].
\]
Then we observe that the right-hand side of \eqref{eq:repart0} is
a continuous function of $t\in(1,\infty)$, so that $F$ is continuous on
$[1,\infty)$
[the right-continuity at $1$ is obvious from~\eqref{eq:valueon[1,2]}].
Thus, $\gamma$ has no atoms
and it follows that the function $F^{(2)}$ is
also continuous on $[1,\infty)$. Using \eqref{eq:repart0} again, we
obtain that $F$ is
continuously differentiable on $[1,\infty)$, and consequently $\gamma$
has a continuous density
$f=-F'$ with respect to Lebesgue measure on $[1,\infty)$. By \eqref
{eq:valueon[1,2]}, $f(t)=K_0t^{-2}$
for $t\in[1,2]$ and in particular $f(1)=K_0$.

Let us finally derive the differential equation \eqref{edphi}. To this
end, we
first differentiate \eqref{eq:repart0} with respect to $t$ to get that
the linear differential equation
\begin{eqnarray}
\label{edo} t(t-1) F'(t) - F(t) &=& -F^{(2)}(t),
\end{eqnarray}
holds for $t\in[1,\infty)$. Then
let $ g: [1,\infty) \to\mathbb{R}_{+}$ be a continuously
differentiable function such that $g(x)$ and $g'(x)$ are
both $o(x^\alpha)$ when $x \to\infty$, for some $\alpha\in(0,\infty)$.
From the definition of $F$ and Fubini's theorem, we have
\[
\int_{1}^\infty\D t g'(t) F(t) =
\mathbb{E} \bigl[ g( \mathcal{C}) \bigr] - g(1)
\]
and similarly
\[
\int_{1}^\infty\D t g'(t)
F^{(2)}(t) = \mathbb{E} \bigl[ g( \mathcal {C}_{1}+
\mathcal{C}_{2}) \bigr] - g(1),
\]
where $ \mathcal{C}_{1}$ and $ \mathcal{C}_{2}$ are independent copies
of $ \mathcal{C}$ under the probability $ \mathbb{P}$. We then multiply
both sides of \eqref{edo} by $g'(t)$ and integrate for $t$ running from
$1$ to $\infty$ to get
\begin{equation}
\label{eq:f} \E \bigl[ \mathcal{C}_{1}( \mathcal{C}_{1}-1)
g'( \mathcal{C}_{1}) \bigr] + \E \bigl[ g(
\mathcal{C}_{1}) \bigr] = \E \bigl[g( \mathcal {C}_{1}+
\mathcal{C}_{2}) \bigr].
\end{equation}
When $g(x)=\exp(- x\ell/2 )$ for $\ell>0$, this readily gives \eqref
{edphi}.
\end{pf*}

\begin{rem*}
We may also take $ g(x) =x^m$ for $m \in\{1,2,3,4,\ldots\}$ in \eqref
{eq:f}. This leads to recursive formulas for the moments of $ \mathcal
{C}$ in terms of the first moment $ \mathbb{E}[ \mathcal{C}]$
(simulations give $ \mathbb{E}[ \mathcal{C}] \approx1.72$).
\end{rem*}

\subsection*{Singular behavior of the density of \texorpdfstring{$\gamma$}{gamma}}
By \eqref{eq:valueon[1,2]}, the values of $F$ and $f=-F'$ on the
interval $[1,2]$ are determined by
the constant $K_0=f(1)$. We have not been able to obtain an exact
numerical value for $K_0$, but simulations indicate that $ K_{0}\approx
1.47$ (see Figure~\ref{fig:density}). We may now observe that the
values of $F$ over $[1,2]$ determine the values of $F^{(2)}$ over
$[2,3]$, via the formula
\[
1-F^{(2)}(t) = \int_1^{t-1} \D s f(s)
\bigl(1-F(t-s) \bigr)\qquad\forall t\in[2,3].
\]
We can then use either
\eqref{edo} or \eqref{eq:repart0} to get a complicated explicit
expression for $F$ over $[2,3]$, again
in terms of $K_0$.
By iterating the argument, we can in principle determine $F$ by
solving linear differential equations on
the successive intervals $[n,n+1]$, $n=1,2,\ldots.$ Unfortunately, the
calculations become tedious and we have not been able to find a closed
expression for $F(t)$. However, from the expressions found for the
first two intervals $[1,2]$ and $[2,3]$,
one can verify that, although the function $f$ is continuously
differentiable on $(1,3)$, one has
\[
f''(2{-})= \frac{3K_{0}}{8}\qquad\mbox{whereas }
f''(2{+}) = \frac{3K_{0}-4K_{0}^2}{8},
\]
so that $f$ is not twice differentiable at the point $2$. See the
inflection point
at $2$ on Figure~\ref{fig:density}.

\subsection{The flow property of harmonic measure}

In this section, we establish a property of harmonic measure that plays
an important role in
the proof of Theorem~\ref{thm:main-harmonic}. This property is well
known in the
discrete setting, but perhaps less standard in the continuous setting,
and we sketch
a short proof.

We fix a Yule-type tree $\t\in\T$. In this section only, we slightly
abuse notation
by writing $W=(W_t)_{t\geq0}$ for Brownian motion with drift $1/2$ on
$\t$ started from
the root. As previously, $W_\infty$ is the exit ray of $W$, and the
distribution of $W_\infty$
is the harmonic measure of $\t$. For every $r>0$, if $x$ is the unique point
of $\t_r$ such that $x\prec W_\infty$, we write $W^{(r)}_\infty$
for the ray of $\t[x]$ that is obtained by shifting $W_\infty$ at time $r$.

%
\begin{lemma}
\label{flow-property}
Let $r>0$ and $x\in\t_r$. Conditionally on $\{x\prec W_\infty\}$, the
law of
$W_\infty^{(r)}$ is the harmonic measure of $\t[x]$.
\end{lemma}

\begin{pf}
For simplicity, we suppose that $x$ is not a branching point of $\t$,
and then
we can choose $\ve>0$
sufficiently small so that there is a unique descendant $x_\ve$ of $x$ in
$\t$ at distance $\ve$ from $x$. Clearly, the harmonic measure of $\t
[x]$ can be
obtained by considering the distribution of the exit ray of Brownian
motion started
from $x_\ve$ and conditioned never to hit $x$. On the other hand, by considering
the successive passage times at $x_\ve$, we can also verify that the
conditional law of $W_\infty^{(r)}$ knowing that $x\prec W_\infty$
corresponds to the same distribution. We leave the details to the reader.
\end{pf}

\section{Proof of Theorem~\texorpdfstring{\protect\ref{thm:main-harmonic}}{3}}
\label{sec:stationary+ergodic}

Let us outline the main steps of the proof of Theorem~\ref{thm:main-harmonic}.
Proposition~\ref{prop:tree-selected-law} below uses the spine decomposition
(Proposition \ref{spine-decomposition}) and the Ray-Knight theorem for
local times of
Brownian motion with drift to determine the exact distribution of the subtree
of the Yule tree above level $r$
that is selected by harmonic measure. In Section~\ref{sec:asymptotics},
we use stochastic calculus
to prove that this law converges as $r \to\infty$ to an explicit
distribution, which is absolutely continuous with respect to $\Theta$
(Corollary~\ref{limit-density}).
In the last two subsections, we rely on arguments of ergodic theory,
mainly inspired by
\cite{LPP95}, to
complete the proof of Theorem~\ref{thm:main-harmonic}.

We recall that $ \mathbb{P}$ stands for the probability measure
under which the Yule tree is defined, whereas Brownian motion (with
drift $1/2$) on the Yule tree is defined under the
probability measure $ P$.

\subsection{The subtree above level $r$ selected by harmonic measure}
\label{sec:subtreeabove}

In this subsection, we fix $r>0$. We will implicitly use
the fact that $\Gamma$ has a.s. no branching point at height $r$.

There is a unique
point $x\in\Gamma_r$ such that $x\prec W_\infty$, and we set $\Gamma
^{(r)}= \Gamma[x]$, which is the subtree above level $r$ selected by harmonic
measure.
We are interested in the distribution
of $\Gamma^{(r)}$.
Let $F$ be a nonnegative measurable function on $\T$, and consider the quantity
\begin{equation}
\label{law-selected} I_r:=\E\otimes E \bigl[F \bigl(\Gamma^{(r)}
\bigr) \bigr]=\E\otimes E \biggl[ \sum_{x\in\Gamma
_r} F \bigl(
\Gamma[x] \bigr) {\mathbf1}_{\{x\prec W_\infty\}} \biggr],
\end{equation}
where the notation $\E\otimes E$ means that we consider the expectation
first under the probability
measure $P$ (under which the Brownian motion $W$ is defined) and then
under $\P$.
We will use Proposition \ref{spine-decomposition} to evaluate $I_r$. In
the first part of the argument, until the derivation
of formula \eqref{selected-tech111} below, we argue
under the probability measure~$P$, that is,
conditionally given the tree $\Gamma$.

Let us fix $x\in\Gamma_r$ and $R>r$. We will use the notation
$\widetilde\Gamma[x]:=\{y\in\Gamma: x\prec y\}$. This is just the set
of all descendants of $x$ in $\Gamma$, now viewed as a subset of $\Gamma
$ and not as a Yule-type tree
as in the definition of $\Gamma[x]$. Define
\[
\Gamma^{x,R}= \bigl\{y\in\Gamma\setminus\widetilde\Gamma[x]:H(y)\leq R
\bigr\} \cup\widetilde\Gamma[x].
\]
Let $W^{x,R}$ be Brownian motion (with drift $1/2$) on $\Gamma^{x,R}$
(we assume that
$W^{x,R}$ is reflected both at the root and at the leaves of $\Gamma
^{x,R}$, which are the points $y$
of $\Gamma\setminus\widetilde\Gamma[x]$ such that $H(y)=R$). We look
for an expression of the
probability that $W^{x,R}$ never hits the leaves of $\Gamma^{x,R}$, or
equivalently that
$W^{x,R}$ escapes to infinity in $\widetilde\Gamma[x]$ before hitting
any leaf of $\Gamma^{x,R}$.
Write $(\ell^{x,R}_t)_{t\geq0}$ for the local time process of
$W^{x,R}$ at $x$. Note that we use here
the standard normalization of local time as an occupation time density.
With this normalization,
$\ell^{x,R}_t$ is the a.s. limit as $\ve\to0$ of the quantities $2\ve
N^{x,\ve}_t$, where
$N^{x,\ve}_t$ is the number of ``upcrossings'' of $W^{x,R}$ from $x$ to
the point $x_\ve\in\Gamma$
such that $x\prec x_\ve$ and $d(x,x_\ve)=\ve$ (this point is unique for
$\ve$ small) before time $t$.
We claim that $\ell^{x,R}_\infty$ has an exponential distribution with
parameter $\cc(\Gamma[x])/2$.
This is easy from excursion theory, but an elementary argument can
be given as follows. Each time $W^{x,R}$ does an upcrossing from $x$ to
$x_\ve$, there is a probability of order $\ve  \cc(\Gamma[x])$
that it escapes to infinity before coming back to $x$ (by the very
definition of $\cc(\Gamma[x])$). Hence, the
total number of upcrossings from $x$ to $x_\ve$ before escaping to
infinity is geometric with parameter of order $\ve  \cc(\Gamma[x])$,
and our claim follows from the approximation of local time by
upcrossing numbers.

We then consider, for every $a\in[0,r]$, the local time
process $(L^{a,R}_t)_{t\geq0}$ of $W^{x,R}$ at the unique point of
$\tllbracket0, x\trrbracket$
at distance $a$ from the root. Note in particular that $L^{r,R}_t= \ell
^{x,R}_t$. The distribution of the\vspace*{2pt}
process $(L^{a,R}_\infty)_{0\leq a\leq r}$ can be derived via a time
change argument, which consists in
looking at $W^{x,R}$ only when it visits $\tllbracket0, x\trrbracket$.
More precisely, we set, for every
$s\geq0$,
\[
\tau_s:=\inf \biggl\{ t\geq0: \int_0^t
\mathbf{1}_{\tllbracket0,
x\trrbracket} \bigl(W^{x,R}_r \bigr)\, \D r > s
\biggr\}
\]
with $\inf\varnothing=\infty$ as usual. Setting $Z_s=W^{x,r}_{\tau_s}$
if $\tau_s<\infty$
and $Z_s=x$ otherwise, we obtain that the process $(Z_s)_{s\geq0}$ is
under $P$ a
Brownian motion (with drift $1/2$) on $\tllbracket0, x\trrbracket$
started from the root, reflected at both ends of
the segment $\tllbracket0, x\trrbracket$, and stopped when its local
time at $x$ hits an independent
exponential variable with parameter $\cc(\Gamma[x])/2$. The latter
exponential random variable is of course
the local time $\ell^{x,R}_\infty=L^{r,R}_\infty$, and the independence
property in the last sentence
corresponds to the independence of excursions
of $W^{x,R}$ ``below'' and
``above'' $x$. The preceding assertions can be obtained either by
arguments of excursion theory,
or, via scaling limits, from the (easy) corresponding properties for
random walk on discrete trees.

Observe that, for every $y\in\tllbracket0, x\trrbracket$, the total
local time of $W^{x,R}$ at $y$
coincides with the total local time of $Z$ at $y$. Using the
Ray--Knight theorem for Brownian motion with drift (see, e.g., \cite
{BS02}, page~93)
we get that, conditionally on $\ell^{x,R}_\infty=\ell$, the process
$(L^{r-a,R}_\infty)_{0\leq a\leq r}$ is distributed
as the process $(X_a)_{0\leq a\leq r}$ which solves the stochastic
differential equation
\begin{equation}
\label{EDSRK} \cases{ dX_a = 2\sqrt{X_a} \,\D
\eta_a + (2- X_a)\,\D a,
\cr
X_0=\ell,}
\end{equation}
where $(\eta_a)_{a\geq0}$ is a standard linear Brownian motion. In
what follows, we will write $P_\ell$
for the probability measure under which the process $X$ starts from
$\ell$, and $P_{(c)}$
for the probability measure under which the process $X$ starts with an
exponential distribution
with parameter $c/2$.

Now write $x_j$, $1\leq j\leq k$ for the branching points of $\Gamma
^{x,R}$ (or equivalently
of~$\Gamma$) that belong to
$\tllbracket0,x\trrbracket$, and set $a_j=H(x_j)$ for $1\leq j\leq k$.
Also let
$\Gamma_{x,j,R}$ be the subtree of $\Gamma^{x,R}$ that branches off
$\tllbracket0,x\trrbracket$ at
$x_j$. We consider the event $A_{x,R}$ where $W^{x,R}$ never hits the
leaves of $\Gamma^{x,R}$.
We can compute the conditional probability of $A_{x,R}$ knowing
the local times $(L^{a,R}_\infty)_{0\leq a\leq r}$, using arguments of
excursion theory. For $1\leq j\leq k$,
write $n_{(j)}$ for the excursion measure of Brownian motion with drift
$1/2$ in the tree $\Gamma_{x,j,R}$
(defined as in the beginning of Section~\ref{sec:conductance}) and let
$\cc(\Gamma_{x,j,R})$ be the conductance of $\Gamma_{x,j,R}$ between
its root
$x_j$ and the set of its leaves. This conductance may be defined as the
measure under $n_{(j)}$ of the event
$E_j$ where the excursion hits the leaves before returning to the root.
Conditionally given
$(L^{a,R}_\infty)_{0\leq a\leq r}$, the excursions of $W^{x,R}$ inside
the tree $\Gamma_{x,j,R}$
form a Poisson point process with intensity $\frac{1}{2}
L^{a_j,R}_\infty n_{(j)}(\cdot)$
and these point processes are independent when $j$ varies (once again,
the reader who
is unfamiliar with excursion theory may find it easier to deduce these
statements from
their discrete versions, which are elementary). Consequently,\vspace*{1pt} the
conditional probability for
a fixed $j$ that no excursion in $\Gamma_{x,j,R}$ hits the leaves is
$\exp(-\frac{1}{2} L^{a_j,R}_\infty n_{(j)}(E_j))= \exp(-\frac{1}{2}
\cc(\Gamma_{x,j,R}) L^{a_j,R}_\infty)$.
Finally, thanks to the conditional independence of the point processes
of excursions
in the different trees $\Gamma_{x,j,R}$, we get that the conditional
probability of $A_{x,R}$ knowing $(L^{a,R}_\infty)_{0\leq a\leq r}$ is
\[
\exp \Biggl(- \frac{1}{2}\sum_{j=1}^k
\cc(\Gamma_{x,j,R}) L^{a_j,R}_\infty \Biggr).
\]

Using the distribution of the process $(L^{r-a,R}_\infty)_{0\leq a\leq
r}$, we have thus
\begin{eqnarray}
\label{selected-tech1} P(A_{x,R})&=& E \Biggl[\exp \Biggl(-\frac{1}{2}
\sum_{j=1}^k \cc(\Gamma _{x,j,R})
L^{a_j,R}_\infty \Biggr) \Biggr]
\nonumber\\[-8pt]\\[-8pt]\nonumber
&=& E_{(\cc(\Gamma[x]))} \Biggl[\exp \Biggl(-\frac{1}{2}\sum
_{j=1}^k \cc(\Gamma _{x,j,R})
X_{r-a_j} \Biggr) \Biggr].
\end{eqnarray}
At this point, we let $R$ tend to infinity. It is easy to verify that
$P(A_{x,R})$ increases to $P(A_x)$,
where
$A_x=\{x\prec W_\infty\}$. Furthermore, for every $j\in\{1,\ldots,k\}$,
$\cc(\Gamma_{x,j,R})$ decreases
to $\cc(\Gamma_{x,j})$, where $\Gamma_{x,j}$ is the subtree of $\Gamma$
branching off
$\tllbracket0,x\trrbracket$ at
$x_j$. Consequently, we obtain that
\begin{equation}
\label{selected-tech111} P(x\prec W_\infty) = E_{(\cc(\Gamma[x]))} \Biggl[\exp
\Biggl(-\frac{1}{2}\sum_{j=1}^k \cc(
\Gamma_{x,j}) X_{r-a_j} \Biggr) \Biggr].
\end{equation}

We can now return to the computation of the quantity $I_r$ defined in
\eqref{law-selected}. By integrating
\eqref{selected-tech111} with respect to $\P$, we get
\begin{eqnarray*}
I_r&=&\E \biggl[ \sum_{x\in\Gamma_r} F \bigl(
\Gamma[x] \bigr) P(x\prec W_\infty) \biggr]
\\
& =& \E \Biggl[\sum_{x\in\Gamma_r} F \bigl(\Gamma[x] \bigr)
E_{(\cc(\Gamma[x]))} \Biggl[\exp \Biggl(-\frac{1}{2}\sum
_{j=1}^k \cc(\Gamma _{x,j})
X_{r-a_j} \Biggr) \Biggr] \Biggr].
\end{eqnarray*}
Note that the quantity inside the sum over $x\in\Gamma_r$ is a
function of
$\Gamma[x]$ and of the subtrees of $\Gamma$ branching off the segment
$\tllbracket0,x\trrbracket$. We can
thus apply Proposition \ref{spine-decomposition} and we get
\[
I_r=e^r \int\Theta(\D\t) F(\t) \mathbf{E} \biggl[
E_{(\cc(\t))} \biggl[\exp \biggl(-\frac{1}{2} \int\n_r
\bigl(\D a \,\D\t' \bigr) \cc \bigl(\t' \bigr)
X_{r-a} \biggr) \biggr] \biggr],
\]
where under the probability measure $\mathbf{P}$, $ \n_r(\D a \,\D\t')$
is a Poisson point measure
on $[0,r]\times\T$ with intensity $2 \D a \Theta(\D\t')$. We can
interchange the expectation under
$\mathbf{P}$ and the one under $P_{(\cc(\t))}$, and using the
exponential formula for
Poisson measures, we arrive at
\begin{equation}
\label{firststepeq} I_r=e^r \int\Theta(\D\t) F(\t)
E_{(\cc(\t))} \biggl[\exp-2\int_0^r \D a
\bigl(1-\varphi(X_a) \bigr) \biggr],
\end{equation}
where we recall that for every $s\geq0$,
\[
\varphi(s)=\E \bigl[\exp(-s\mathcal{C}/2) \bigr]= \Theta \bigl(\exp \bigl(-s\cc(
\t)/2 \bigr) \bigr)
\]
is the Laplace transform (evaluated at $s/2$) of the distribution of
the conductance of the Yule tree.
We have thus proved the following proposition.

%
\begin{proposition}
\label{prop:tree-selected-law}
The distribution under $\P\otimes P$ of the subtree $\Gamma^{(r)}$ has a
density with respect to the law $\Theta(\D\t)$ of the Yule tree, which
is given by
$\Phi_r(\cc(\t))$, where, for every $c>0$,
\[
\Phi_r(c)= E_{(c)} \biggl[\exp-\int_0^r
\D a \bigl(1-2\varphi(X_a) \bigr) \biggr].
\]
\end{proposition}

\subsection{Asymptotics}
\label{sec:asymptotics}
In this section, we study the asymptotic behavior of $\Phi_r(c)$
when $r$ tends to $\infty$. We first observe that, in terms of the law
$\gamma(\D s)$ of $\mathcal{C}(\Gamma)$, we have
\[
\varphi(\ell) = \int_{[1,\infty)} e^{-\ell s/2} \gamma(\D s),
\qquad \varphi'(\ell) = - \frac{1}{2} \int
_{[1,\infty)} s e^{-\ell s/2} \gamma(\D s).
\]
It follows that $\varphi(\ell)\leq e^{-\ell/2}$ and $\llvert  \varphi'(\ell
)\rrvert  \leq\frac{1}{2}(\int s\gamma(\D s))e^{-\ell/2}$.
By differentiating~\eqref{edphi}, we have also
\begin{equation}
\label{edphi2} 2\ell \varphi'''(\ell)
+ (2+\ell) \varphi''(\ell) + 2\varphi(\ell )
\varphi'(\ell)=0.
\end{equation}

Our main tool is the next proposition.

%
\begin{proposition}
\label{asympto-density}
For every $\ell\geq0$,
\[
\lim_{r\to\infty} E_\ell \biggl[ \exp-\int
_0^r \D a \bigl(1-2\varphi(X_a)
\bigr) \biggr] = - \frac{\varphi'(\ell) e^{\ell/2}}{\int_0^\infty\D s \varphi
'(s)^2 e^{s/2}}.
\]
Additionally, there exists a constant $K<\infty$ such that, for every
$\ell\geq0$ and $r>0$,
\[
E_\ell \biggl[ \exp-\int_0^r \D a
\bigl(1-2\varphi(X_a) \bigr) \biggr] \leq K.
\]
\end{proposition}

\begin{pf}
Under the probability measure $P_{(\cc(\t))}$ the process $X$ starts
with an initial distribution which is exponential
with parameter $\cc(\t)/2$. Consequently, under $\int\Theta(\D\t)
P_{(\cc(\t))}$, the initial density of $X$ is
\begin{equation}
\label{asymptech1} q(\ell)=\int_{[1,\infty)} \gamma(\D s)
\frac{s}{2} e^{-s\ell/2}= -\varphi'(\ell).
\end{equation}
However, from (\ref{firststepeq}) with $F=1$, we have
\begin{eqnarray*}
1&=& \int\Theta(\D\t) E_{(\cc(\t))} \biggl[\exp-\int_0^r
\D a \bigl(1-2\varphi (X_a) \bigr) \biggr]
\\
&=&-\int\D\ell \varphi'(\ell) E_\ell \biggl[\exp-\int
_0^r \D a \bigl(1-2\varphi(X_a)
\bigr) \biggr].
\end{eqnarray*}

We can generalize the last identity via a minor extension of the calculations
of the preceding section.
We let $L^0_\infty$ be the total local time accumulated by the process $W$
at the root of $\Gamma$.
Let $r>0$ and let $h$ be a bounded nonnegative continuous
function on $(0,\infty)$, and instead of the quantity $I_r$
of the preceding section, set
\[
I_r^h:=\E\otimes E \biggl[ h \bigl(L^0_\infty
\bigr) \sum_{x\in\Gamma_r} F \bigl(\Gamma [x] \bigr) {
\mathbf1}_{\{x\prec W_\infty\}} \biggr],
\]
where $F$ is a given nonnegative measurable function on $\T$. The same
calculations that led
to (\ref{selected-tech1}) give, for every $x\in\Gamma_r$ and $R>r$,
\begin{eqnarray}
\label{selected-tech2} E \bigl[ h \bigl(L^{0,R}_\infty \bigr) {
\mathbf1}_{A_{x,R}} \bigr]&=& E \Biggl[h \bigl(L^{0,R}_\infty
\bigr) \exp \Biggl(-\frac{1}{2}\sum_{j=1}^k
\cc(\Gamma_{x,j,R}) L^{a_j,R}_\infty \Biggr) \Biggr]
\nonumber\\[-8pt]\\[-8pt]\nonumber
&=& E_{(\cc(\Gamma[x]))} \Biggl[h(X_r) \exp \Biggl(-\frac{1}{2}
\sum_{j=1}^k \cc (\Gamma_{x,j,R})
X_{r-a_j} \Biggr) \Biggr].
\end{eqnarray}
When $R\to\infty$, $L^{0,R}_\infty$ converges to $L^0_\infty$, and so
we get
\[
E \bigl[ h \bigl(L^{0}_\infty \bigr) {\mathbf1}_{\{x\prec W_\infty\}}
\bigr]=E_{(\cc(\Gamma
[x]))} \Biggl[h(X_r) \exp \Biggl(-\frac{1}{2}
\sum_{j=1}^k \cc(\Gamma_{x,j})
X_{r-a_j} \Biggr) \Biggr].
\]
We then sum over $x\in\Gamma_r$ and integrate with respect to $\P$. By
the same manipulations as
in the preceding section, we arrive at
\begin{equation}
\label{firststepeq2} I^h_r=e^r \int\Theta(\D\t)
F(\t) E_{(\cc(\t))} \biggl[h(X_r) \exp-2\int_0^r
\D a \bigl(1-\varphi(X_a) \bigr) \biggr].
\end{equation}
Note that if $F=1$,
\[
I^h_r= \E\otimes E \bigl[h \bigl(L^0_\infty
\bigr) \bigr] = - \int_0^\infty
\varphi'( \ell) h(\ell)\, \D\ell
\]
since given $\Gamma=\t$ the local time $L^0_\infty$ follows an
exponential distribution with parameter $\cc(\t)/2$,
and we use the same calculation as in (\ref{asymptech1}). Hence, the
case\vadjust{\goodbreak} $F=1$ of (\ref{firststepeq2}) gives
\begin{equation}
\label{asymptech2} \quad\int_0^\infty\D\ell
\varphi'(\ell) E_\ell \biggl[h(X_r) \exp-\int
_0^r \D a \bigl(1-2\varphi(X_a)
\bigr) \biggr] = \int_0^\infty\D\ell
\varphi'(\ell) h(\ell).
\end{equation}
By an obvious truncation argument, this identity also holds if $h$ is unbounded.

At this point, we need a lemma.\noqed
\end{pf}

%
\begin{lemma}\label{martingalelemma}
The process
\[
M_a:= -\varphi'(X_a) \exp \biggl(
\frac{X_a}{2}-\int_0^a \D s \bigl(1-2
\varphi (X_s) \bigr) \biggr)
\]
is a martingale under $P_\ell$, for every $\ell\geq0$.
\end{lemma}

\begin{pf} 
From the stochastic differential
equation (\ref{EDSRK}),
an application of It\^o's formula shows that the finite variation part
of the
semimartingale~$-M_a$ is
\begin{eqnarray*}
&& \int_0^a \bigl(2X_s
\varphi'''(X_s) +
(2+X_s)\varphi''(X_s) + 2
\varphi (X_s)\varphi'(X_s) \bigr)
\\
&&\qquad {}\times\exp \biggl(\frac{X_s}{2}-\int_0^s
\D u \bigl(1-2\varphi (X_u) \bigr) \biggr)\, \D s
\end{eqnarray*}
and this vanishes thanks to (\ref{edphi2}). Hence, $M$ is a local
martingale. Furthermore, we already noticed
that, for every
$\ell\geq0$, $\llvert  \varphi'(\ell)\rrvert  \leq Ce^{-\ell/2}$, where
$C:=\frac{1}{2}\int s \gamma(\D s)$. It follows that $\llvert  M\rrvert  $ is bounded
by $C e^a$ over the
time interval $[0,a]$, and thus $M$ is a (true) martingale.
\end{pf}

We return to the proof of Proposition~\ref{asympto-density}.
Let $\ell\geq0$ and $t>0$. On the probability space where $X$ is
defined, we introduce a new probability
measure $Q^t_\ell$ by setting
\[
Q^t_\ell= \frac{M_t}{M_0}\cdot P_\ell.
\]
Note that the fact that $Q^t_\ell$ is a probability measure follows
from the martingale property derived
in Lemma~\ref{martingalelemma}. Furthermore, we have $P_\ell$ a.s.
\[
\frac{M_t}{M_0} = \frac{\varphi'(X_t)}{\varphi'(\ell)} \exp \biggl(\frac{X_t-\ell}{2} - \int
_0^t \D s \bigl(1-2\varphi(X_s)
\bigr) \biggr),
\]
so that the martingale part of $\log\frac{M_t}{M_0}$ is
\[
\int_0^t \sqrt{X_s} \,\D
\eta_s + 2\int_0^t
\frac{\varphi''(X_s)}{\varphi
'(X_s)} \sqrt{X_s} \,\D\eta_s,
\]
where $\eta$ is the linear Brownian motion in (\ref{EDSRK}). An
application of Girsanov's theorem
shows that the process
\[
\widetilde\eta_s:= \eta_s -\int_0^s
\sqrt{X_u} \biggl(1+ \frac{2\varphi
''(X_u)}{\varphi'(X_u)} \biggr)\, \D u,\qquad0\leq
s \leq t,
\]
is a linear Brownian motion over the time interval $[0,t]$, under
$Q^t_\ell$. Furthermore, still on the time
interval $[0,t]$, the process $X$ satisfies the stochastic differential equation
\[
dX_s = 2\sqrt{X_s} \,\D\widetilde\eta_s +
2X_s \biggl(1+ \frac{2\varphi
''(X_s)}{\varphi'(X_s)} \biggr)\,\D s +
(2-X_s)\,\D s,
\]
or equivalently, using (\ref{edphi}),
\begin{equation}
\label{EDSbis} dX_s = 2\sqrt{X_s} \,\D\widetilde
\eta_s + \biggl(2-X_s + 2\frac{\varphi
-\varphi^2}{\varphi'}(X_s)
\biggr)\,\D s.
\end{equation}
Notice that the function
\[
\ell\mapsto\frac{\varphi-\varphi^2}{\varphi'}(\ell)
\]
is continuously differentiable over $[0,\infty)$, takes negative values
on $(0,\infty)$ and vanishes at $0$.
Pathwise uniqueness, and therefore also weak uniqueness, holds for~(\ref
{EDSbis}) by an
application of the classical Yamada--Watanabe criterion (see, e.g.,~\cite{RY}, Theorem IX.3.5).
The
preceding considerations show that, under the probability measure
$Q^t_\ell$ and on the time
interval $[0,t]$, the process $X$ is distributed as the diffusion
process on $[0,\infty)$
with generator
\[
\mathcal{L} = 2r \frac{\D^2}{\D r^2} + \biggl(2-r + 2\frac{\varphi
-\varphi^2}{\varphi'}(r)
\biggr) \frac{\D}{\D r}
\]
started from $\ell$. Write $\widetilde X$ for this diffusion process,
and assume that $\widetilde X$ starts from $\ell$
under the probability measure $P_\ell$. Note\vspace*{1pt} that $0$ is an entrance
point for $\widetilde X$, but,
independently of its starting point, $\widetilde X$ does not visit $0$
at a positive time (indeed this follows
from the fact that $X$ does not visit $0$ at a positive time).

Standard comparison theorems for stochastic differential equations
(see, e.g., \cite{RY}, Theorem IX.3.7) can be used to
compare the solutions of (\ref{EDSRK}) and (\ref{EDSbis}), and it
follows that
$\widetilde X$ is recurrent on $(0,\infty)$.

We next observe that the finite measure $\rho$
on $(0,\infty)$ defined by
\[
\rho(\D\ell):= \varphi'(\ell)^2 e^{\ell/2} \,\D
\ell
\]
is invariant for $\widetilde X$. Indeed, we have, for any bounded
continuous function $h$ on~$(0,\infty)$,
\begin{eqnarray*}
&&\int_{(0,\infty)} \D\ell \varphi'(
\ell)^2 e^{\ell/2} E_\ell \bigl[h(\widetilde
X_t) \bigr]
\\
&&\qquad=\int_{(0,\infty)}\D\ell \varphi'(
\ell)^2 e^{\ell/2} Q^t_\ell
\bigl[h(X_t) \bigr]
\\
&&\qquad=\int_{(0,\infty)} \D\ell \varphi'(\ell)
E_\ell \biggl[h(X_t) \varphi'(X_t)
\exp \biggl(\frac{X_t}{2} - \int_0^t \D s
\bigl(1-2\varphi (X_s) \bigr) \biggr) \biggr]
\\
&&\qquad=\int_{(0,\infty)} \D\ell \varphi'(
\ell)^2 e^{\ell/2} h(\ell),
\end{eqnarray*}
where the last equality follows from (\ref{asymptech2}). We normalize
$\rho$ by setting
\[
\widehat\rho= \frac{\rho}{\rho((0,\infty))}.
\]

From the known results about the convergence of positive recurrent
diffusion processes
toward their stationary distribution (see Chapter~23 in Kallenberg \cite{Kal}),
the distribution of $\widetilde X_t$ under $P_\ell$ converges to
$\widehat\rho$ in variation norm
as $t\to\infty$, for any $\ell\geq0$. Consequently, for any bounded
Borel function $g$
on $[0,\infty)$, and every $\ell\geq0$,
\begin{equation}
\label{conv-mesu-invar} E_\ell \bigl[g(\widetilde X_t) \bigr]
\mathop{\la}\limits_{t\to\infty} \int g \,\D\widehat \rho.
\end{equation}
We claim that (\ref{conv-mesu-invar}) still holds if $g$ may be
unbounded but is assumed to be nonnegative,
monotone increasing and such that $\int g \,\D\widehat\rho<\infty$. To
see this, fix $\ell\geq0$ and write $\Pi_t(\ell, \D\ell')$ for the
distribution of $\widetilde X_t$ under $P_\ell$. Using
comparison theorems for stochastic differential equations (see,\vspace*{1pt} e.g.,
\cite{RY}, Theorem IX.3.7), we can, for every choice of $\ell'\geq\ell$,
couple a solution $\widetilde X^1$ of~\eqref{EDSbis} starting from $\ell
$ and a solution $\widetilde X^2$ starting from $\ell'$ so that
$\widetilde X^2_t \geq\widetilde X^1_t$ for all $t \geq0$. It follows that
\begin{eqnarray*}
\int\Pi_{t} \bigl(\ell, \D\ell' \bigr) g \bigl(
\ell' \bigr) &\leq& \frac
{1}{\widehat{\rho}( \ell, \infty)} \int_{\ell}^{\infty}
\widehat\rho (\D u) \int\Pi_{t} \bigl(u, \D\ell' \bigr)
g \bigl(\ell' \bigr)
\\
&\leq& \frac{1}{\widehat
{\rho}( \ell, \infty)} \int_{0}^{\infty} \widehat
\rho(\D u) \int\Pi _{t} \bigl(u, \D\ell' \bigr) g \bigl(
\ell' \bigr)
\\
&=& \frac{1}{\widehat{\rho}( \ell, \infty
)} \int\widehat{\rho} \bigl(\D\ell' \bigr) g
\bigl(\ell' \bigr).
\end{eqnarray*}
Applying the above display to the function $ g \mathbf{1}_{(A,\infty
)}$, where $A>0$, we get that
\[
0 \leq\int\Pi_{t} \bigl( \ell, \D\ell' \bigr) g \bigl(
\ell' \bigr) -\int_{[0,A]} \Pi_{t}
\bigl( \ell, \D\ell' \bigr) g \bigl(\ell' \bigr) \leq
\frac{1}{\widehat{\rho}( \ell, \infty)}\cdot\int_{(A,\infty)} \widehat{\rho} \bigl(\D
\ell' \bigr) g \bigl(\ell' \bigr).
\]
Since $g$ is integrable with respect to $\widehat{\rho}$, the
right-hand side can be made arbitrarily small, by choosing $ A$ large
enough. Our claim now follows by letting $t \to\infty$, using the fact
that $g \mathbf{1}_{[0,A]}$ is a bounded Borel function.

We can thus apply (\ref{conv-mesu-invar}) to the nonnegative
increasing function
\[
g(\ell)= -\frac{1}{\varphi'(\ell)} e^{-\ell/2}
\]
which is such that
$\int g \,\D\rho=-\int\varphi'(\ell)\,\D\ell= 1$. Note that for this
particular function~$g$,
\[
E_\ell \bigl[g(\widetilde X_t) \bigr] =
Q^t_\ell \bigl[g(X_t) \bigr]= -
\frac{e^{-\ell
/2}}{\varphi'(\ell)} E_\ell \biggl[ \exp \biggl(-\int
_0^t \D s \bigl(1-2\varphi(X_s)
\bigr) \biggr) \biggr].
\]
It follows from (\ref{conv-mesu-invar})
that, for every $\ell\geq0$,
\begin{eqnarray*}
\lim_{t\to\infty} -\frac{e^{-\ell/2}}{\varphi'(\ell)} E_\ell \biggl[ \exp
\biggl(-\int_0^t \D s \bigl(1-2
\varphi(X_s) \bigr) \biggr) \biggr] &=&\int g \,\D \widehat\rho=
\frac{1}{\rho((0,\infty))}
\\
&=& \frac{1}{\int_{(0,\infty)}
\D s  \varphi'(s)^2e^{s/2}}.
\end{eqnarray*}
This gives the first assertion of the proposition.

The second assertion is now easy. By the first assertion, there exists
a constant $K$ such that, for every $r\geq0$,
\[
E_0 \biggl[ \exp \biggl(-\int_0^r
\D s \bigl(1-2\varphi(X_s) \bigr) \biggr) \biggr] \leq K.
\]
Since the function $\varphi$ is monotone decreasing, a comparison
argument gives for every
$\ell\geq0$ and $r\geq0$,
\[
E_\ell \biggl[ \exp \biggl(-\int_0^r
\D s \bigl(1-2\varphi(X_s) \bigr) \biggr) \biggr]\leq
E_0 \biggl[ \exp \biggl(-\int_0^r
\D s \bigl(1-2 \varphi(X_s) \bigr) \biggr) \biggr]\leq K.
\]
This completes the proof of the proposition. \hfill$\square$

To simplify notation, we set
\[
C_0:= \int_0^\infty \D s
\varphi'(s)^2 e^{s/2} = \int\!\!\!\int\gamma (\D
\ell)\gamma \bigl(\D\ell' \bigr) \frac{\ell\ell'}{2(\ell+\ell'-1)}.
\]

%
\begin{corollary}
\label{limit-density}
For every $c>0$,
\[
\lim_{r\to\infty} \Phi_r(c)= \Phi_\infty(c),
\]
where
\[
\Phi_\infty(c) = \frac{1}{C_0} \int\gamma(\D s)\frac{cs}{2(c+s-1)}.
\]
\end{corollary}

\begin{pf} By definition, we have
\[
\Phi_r(c)= \frac{c}{2}\int_0^\infty
\D\ell e^{-c\ell/2} E_\ell \biggl[ \exp \biggl(-\int
_0^r \D s \bigl(1-2\varphi(X_s)
\bigr) \biggr) \biggr].
\]
From Proposition~\ref{asympto-density}
and an application of the dominated convergence theorem, we get
\[
\lim_{r\to\infty}\Phi_r(c)= \frac{c}{2}\int
_0^\infty \D\ell e^{-c\ell/2} \times \biggl(-
\frac{\varphi'(\ell) e^{\ell/2}}{C_0} \biggr).
\]
The limit is identified with $\Phi_\infty(c)$ by a straightforward
calculation.
\end{pf}

\subsection{The invariant measure}
\label{sec:ergodic}

For the purposes of this section, it will be useful to introduce the
set of all pairs consisting of
a tree $\t\in\T$ and a distinguished geodesic ray $\mathbf{v}$, which
we can represent by an element
of $\{1,2\}^\N$. We formally set
\[
\T^*= \T\times \{1,2\}^\N.
\]
We can define shifts $(\sigma_r)_{r\geq0}$ on $\T^*$
in the following way. For $r=0$, $\sigma_r$ is just the identity mapping
of $\T^*$. Then let $r>0$ and $(\t,\mathbf{v})\in\T^*$.
Write $\mathbf{v}=(v_1,v_2,\ldots)$ and $\mathbf{v}_n=(v_1,\ldots,v_n)$
for every $n\geq0$. Also let $x_{r,\mathbf{v}}$ be the unique
element of $\t_r$ such that $x_{r,\mathbf{v}}\prec\mathbf{v}$. Then,
if $k=\min\{n\geq0: z_{\mathbf{v}_n}\geq r\}$, we set
\[
\sigma_r(\t,\mathbf{v})= \bigl(\t[x_{r,\mathbf{v}}],
(v_{k+1},v_{k+2},\ldots) \bigr).
\]
Informally, $\sigma_r(\t,\mathbf{v})$ is obtained by taking the subtree
of $\t$ consisting of descendants
of the vertex at height $r$ on the distinguished geodesic ray, and
keeping in this subtree
the ``same'' geodesic ray.
It is straightforward to verify that $\sigma_r\circ\sigma_s=\sigma
_{r+s}$ for every $r,s\geq0$.

Under the probability measure $\P\otimes P$, we can view $(\Gamma
,W_\infty)$ as
a random variable with values in $\T^*$. Write $\Theta^*$
for the distribution of $(\Gamma,W_\infty)$. Then $\Theta^*$
is not invariant under the shifts $\sigma_r$, but Corollary~\ref{limit-density}
will give an invariant measure absolutely continuous with respect
to $\Theta^*$.

%
\begin{proposition}
\label{invariant-meas}
The probability measure
\[
\Lambda^*(\D\t \,\D\mathbf{v}):=\Phi_\infty \bigl(\cc(\t) \bigr)
\Theta^*(\D\t \,\D \mathbf{v})
\]
is invariant under the shifts $\sigma_r$, $r\geq0$.
\end{proposition}

\begin{pf} Let $r>0$. We have
\[
\sigma_r(\Gamma,W_\infty)= \bigl(\Gamma^{(r)},W_\infty^{(r)}
\bigr),
\]
where $\Gamma^{(r)}$ and $W_\infty^{(r)}$ are as in the previous sections.

By Proposition~\ref{prop:tree-selected-law}, we have, for any bounded
measurable function $F$ on $\T$,
\[
\E\otimes E \bigl[F \bigl(\Gamma^{(r)} \bigr) \bigr] = \int\Theta(\D\t)
\Phi_r \bigl(\cc(\t) \bigr) F(\t).
\]
Write $\nu_\t$ for the harmonic measure of a Yule-type tree $\t$.
At this point, we use the flow property of harmonic measure.
By Lemma \ref{flow-property} and the preceding identity, we
have also, for any bounded
measurable function $F$ on $\T^*$,
\begin{eqnarray}
\label{invar-tech1} \E\otimes E \bigl[F \bigl(\Gamma^{(r)},W_\infty^{(r)}
\bigr) \bigr]& =&\E\otimes E \biggl[\int\nu_{\Gamma^{(r)}}(\D\mathbf{v}) F
\bigl( \Gamma ^{(r)},\mathbf{v} \bigr) \biggr]
\nonumber
\\
&=&\int\Theta(\D\t) \Phi_r \bigl(\cc(\t) \bigr)\int
\nu_\t(\D\mathbf{v}) F(\t,\mathbf{v})
\\
&=& \int\Theta^*(\D\t \,\D\mathbf{v}) \Phi_r \bigl(\cc(\t) \bigr) F(
\t, \mathbf{v}),\nonumber
\end{eqnarray}
since $\Theta^*(\D\t \,\D\mathbf{v})=\Theta(\D\t)\nu_\t(\D\mathbf{v})$
by construction.

If we now let $r\to\infty$, Corollary~\ref{limit-density} gives
\[
\lim_{r\to\infty} \E\otimes E \bigl[F \bigl(\Gamma^{(r)},W_\infty^{(r)}
\bigr) \bigr] = \int \Theta^*(\D\t \,\D\mathbf{v}) \Phi_\infty \bigl(
\cc(\t) \bigr) F(\t,\mathbf{v})
\]
noting that the functions $\Phi_r$ are uniformly bounded thanks to the last
assertion of Proposition~\ref{asympto-density}. Let $s>0$. If we replace $F$ by
$F\circ\sigma_s$ in the last
convergence, observing that $F\circ\sigma_s(\Gamma^{(r)},W_\infty
^{(r)})=F\circ\sigma_s\circ\sigma_r(\Gamma,W_\infty)
= F(\Gamma^{(s+r)},W_\infty^{(s+r)})$, we get
\[
\int\Theta^*(\D\t \,\D\mathbf{v}) \Phi_\infty \bigl(\cc(\t) \bigr) F(
\t, \mathbf{v})= \int\Theta^*(\D\t \,\D\mathbf{v}) \Phi_\infty \bigl(\cc(
\t) \bigr) F\circ\sigma _s(\t,\mathbf{v}),
\]
which was the desired result.
\end{pf}

%
\begin{proposition}
\label{ergodicity}
For every $r>0$, the shift $\sigma_r$ acting on the probability space
$(\T^*, \Lambda^*)$ is ergodic.
\end{proposition}

\begin{pf} We take $r=1$ in this proof, and we write $\sigma=\sigma_1$
for simplicity. We essentially rely on ideas of
\cite{LPP95} (see also \cite{LP10}, Chapter~16); however, our setting
is different, because our trees are not discrete,
and also because we consider ordered trees rather than unordered trees
in \cite{LPP95}. For this
reason, we will provide some details. We write
$\pi_1$ for the canonical projection from $\T^*$ onto $\T$, and let
$\Lambda$ be the image of
$\Lambda^*$ under this projection, so that
\[
\Lambda(\D\t)=\Phi_\infty \bigl(\mathcal{C}(\t) \bigr) \Theta(\D\t).
\]
We define a transition kernel $\bp(\t, \D\t')$ on $\T$ by setting
\[
\bp \bigl(\t, \D\t' \bigr) = \sum_{x\in\t_1}
\nu_\t \bigl( \bigl\{{\mathbf v}\in\{1,2\}^\N: x\prec {
\mathbf v} \bigr\} \bigr) \delta_{\t[x]} \bigl(\D\t' \bigr).
\]
Informally, under the probability measure
$\bp(\t, \D\t')$, we choose one of the subtrees of $\t$ above level $1$
with probability equal to
its harmonic measure.
Then it follows from Proposition~\ref{invariant-meas} that $\Lambda$ is a
stationary probability
measure for the Markov chain with transition kernel $\bp$. Indeed, Lemma~\ref{flow-property} shows that we may obtain
this Markov chain under its stationary measure $\Lambda$ by considering
the process
\[
\mathcal{Z}_n(\t,\mathbf{v}):= \pi_1 \bigl(
\sigma_n(\t,\mathbf{v}) \bigr),\qquad n=0,1,2,\ldots
\]
on the probability space $(\T^*,\Lambda^*)$. Note that $\mathcal{Z}_0(\t
,\mathbf{v})=\t$.

Write $\T^\infty$ for the set of all sequences $(\t^0,\t^1,\ldots)$ of
elements of $\T$. By
\cite{LP10}, Proposition 16.2, if a measurable subset $F$ of $\T^\infty
$ is shift-invariant
for the Markov chain $\mathcal{Z}$, in the sense that $\mathbf
{1}_F(\mathcal{Z}_0,\mathcal{Z}_1,\ldots)= \mathbf{1}_F(\mathcal
{Z}_1,\mathcal{Z}_2,\ldots)$
a.s., then there exists a measurable subset $A$ of $\T$ such that
\[
\mathbf{1}_F(\mathcal{Z}_0,\mathcal{Z}_1,
\ldots) = \mathbf{1}_A(\mathcal {Z}_0),\qquad
\mbox{a.s.}
\]
and moreover,
\[
\bp(\t,A)=\mathbf{1}_A(\t),\qquad\Lambda(\D\t)\mbox{ a.s.}
\]

We let $\widehat\T^\infty$ be the set of all sequences $(\t^0,\t
^1,\ldots)$ in $\T^\infty$,
such that, for every integers $0\leq i<j$, $\t^j$ is a subtree of $\t
^i$ above generation $j-i$
(i.e., there exists a point $x\in\t^i_{j-i}$ such that $\t^j=\t^i[x]$).
Note that $\widehat\T^\infty$ is a measurable subset of $\T^\infty$ and that
$(\mathcal{Z}_0(\t,\mathbf{v)},\mathcal{Z}_1(\t,\mathbf{v}),\ldots)\in
\widehat\T^\infty$
for every $(\t,\mathbf{v})\in\T^*$. If
$(\t^0,\t^1,\ldots)\in\widehat\T^\infty$, there exists $\mathbf
{v}\in\{1,2\}^\N$ such that
$\t^j=\mathcal{Z}_j(\t^0,\mathbf{v})$ for every $j\geq0$, and we set
$\Psi(\t^0,\t^1,\ldots)=(\t^0,\mathbf{v})$.
Note that $\mathbf{v}$ is a priori not unique, but for the previous
definition to make sense we
take the smallest possible $\mathbf{v}$ in lexicographical ordering (of
course for the
random trees that we consider later this uniqueness problem does not arise).
In this way, we define a measurable mapping
$\Psi$ from $\widehat\T^\infty$ into $\T^*$, and we have
$\Psi(\mathcal{Z}_0(\t,\mathbf{v)},\mathcal{Z}_1(\t,\mathbf{v}),\ldots
)=(\t,\mathbf{v})$, $\Lambda^*$ a.s.

Let us now prove the statement of the proposition. We let $B$
be a measurable subset of $\T^*$ such that $\sigma^{-1}(B)= B$, and we aim
at proving that $\Lambda^*(B)=0$\vspace*{2pt} or~$1$. To this end, we set
$F=\Psi^{-1}(B)$, which is a measurable subset of $\widehat\T^\infty
\subset\T^\infty$. Furthermore, we
claim that
$F$ is shift-invariant. To see this, we have to verify that
\[
\bigl\{(\mathcal{Z}_0,\mathcal{Z}_1,\ldots)\in F \bigr
\}= \bigl\{(\mathcal{Z}_1,\mathcal {Z}_2,\ldots)\in F
\bigr\},\qquad\mbox{a.s.}
\]
or equivalently
\[
\bigl\{\Psi(\mathcal{Z}_0,\mathcal{Z}_1,\ldots)\in B
\bigr\}= \bigl\{\Psi(\mathcal {Z}_1,\mathcal{Z}_2,\ldots)
\in B \bigr\},\qquad\mbox{a.s.}
\]
But this is immediate since by construction
$\Psi(\mathcal{Z}_1,\mathcal{Z}_2,\ldots)=\sigma\circ\Psi(\mathcal
{Z}_0,\mathcal{Z}_1,\ldots)$ a.s. and $\sigma^{-1}(B)= B$ by
assumption.

From preceding considerations, we then obtain that there exists
a measurable subset $A$ of $\t$, such that $(\mathcal{Z}_0,\mathcal
{Z}_1,\ldots)\in F$ if and only if $\mathcal{Z}_0\in A$, a.s., and
moreover $\bp(\t,A)=\mathbf{1}_A(\t)$, $\Lambda(\D\t)$ a.s. Since
$\Psi(\mathcal{Z}_0(\t,\mathbf{v)},\mathcal{Z}_1(\t,\mathbf{v}),\ldots
)=(\t,\mathbf{v})$, $\Lambda^*$ a.s., it also follows that we have
$(\t,\mathbf{v})\in B$ if and only if $\t\in A$, $\Lambda^*$ a.s.

However, from the property $\bp(\t,A)=\mathbf{1}_A(\t)$, $\Lambda(\D\t
)$ a.s., one can verify
that $\Lambda(A)=0$ or $1$. First, note that this property also implies
that $\bp(\t,A)=\mathbf{1}_A(\t)$, $\Theta(\D\t)$ a.s.
Hence, $\Theta(\D\t)$ a.s., the tree $\t$ belongs to $A$ if and only if each
of its subtrees above level $1$ belong to $A$ [it is clear that that
the measure
$\bp(\t,\cdot)$ assigns a positive mass to each of these subtrees].
Then, if $p_k=\P(\#\Gamma_1=k)$, for every $k\geq1$, the
branching property of the Yule tree shows that
\[
\Theta(A)= \sum_{k=1}^\infty
p_k \Theta(A)^k
\]
which is only possible if $\Theta(A)=0$ or $1$, or equivalently $\Lambda
(A)=0$ or $1$.
Finally, we also get that $\Lambda^*(B)=0$ or $1$, which completes the
proof.
\end{pf}

\subsection{End of the proof}
\label{sec:proofoftheoremmain-harmonic}

Recall that $\nu_\t$ stands for the harmonic measure of a tree $\t\in\T
$. With this notation, we have $\nu=\nu_\Gamma$. For every
$r>0$, we then consider the nonnegative measurable function $F_r$
defined on $\T^*$
by the formula
\[
F_r(\t,\mathbf{v})= \nu_\t \bigl(\mathcal{B}_\t(
\mathbf{v},r) \bigr),
\]
where $\mathcal{B}_\t(\mathbf{v},r)$ denotes the set of all geodesic
rays of $\t$ that
coincide with the ray $\mathbf{v}$ over the interval $[0,r]$. We claim
that, for every
$r,s>0$, we have
\[
F_{r+s} = F_r \times F_s\circ
\sigma_r.
\]
Indeed, if we write $\sigma_r(\t,\mathbf{v})=(\t^{(r)},\mathbf
{v}^{(r)})$, this is equivalent to saying that
\[
\frac{\nu_\t(\mathcal{B}_\t(\mathbf{v},s+r))}{\nu_\t(\mathcal{B}_\t
(\mathbf{v},r))} = \nu_{\t^{(r)}} \bigl(\mathcal{B}_{\t^{(r)}}
\bigl( \mathbf{v}^{(r)},s \bigr) \bigr),
\]
and the latter equality is an immediate consequence of Lemma \ref
{flow-property}.

If we set $G_r=-\log F_r\geq0$, we have for every $r,s>0$,
\[
G_{s+r}= G_r + G_s\circ\sigma_r
\]
and the ergodic theorem (with Proposition~\ref{ergodicity}) implies that
\[
\frac{1}{s} G_s \mathop{\la}\limits_{s\to\infty}^{\Lambda^*\,\mathrm{a.s.}}
\Lambda^*(G_1).
\]
Since $\Lambda^*$ has a strictly positive density with respect to
$\Theta^*$, the latter convergence
also holds $\Theta^*$ a.s. Recalling that $\Theta^*$ is
the distribution of $(\Gamma,W_\infty)$, this exactly gives the
convergence~\eqref{equivalent-asymp},
with $\beta= \Lambda^*(G_1)$. This completes the proof of Theorem~\ref
{thm:main-harmonic}, except
that we have not checked that $\beta<1$. We will do this in the next
proposition, and
then we will complete the proof of Proposition~\ref{prop:valuebeta} by
deriving the explicit formulas~\eqref{value-beta} for $\beta$
in terms of the law $\gamma$ of the conductance $\cc(\Gamma)$.

%
\begin{proposition}
\label{prop:beta<1}
We have $\beta<1$.
\end{proposition}

\begin{pf} Here again, we strongly rely on ideas from \cite{LPP95} (see also
\cite{LP10}, Chapter~16). We start with some notation. If $\t\in\T$
and $x\in\t_1$, we set
\[
\nu^*_\t(x)= \nu_{\t} \bigl( \bigl\{\mathbf{v}\in\{1,2
\}^\N:x\prec\mathbf{v} \bigr\} \bigr).
\]
Clearly, $(\nu^*_\t(x))_{x\in\t_1}$ is a probability distribution on $\t
_1$. We also
set, for every $\t\in\T$,
\[
U(\t)=\liminf_{r\to\infty} e^{-r} \#\t_r
\in[0,\infty].
\]
It is well known that the preceding liminf is a limit, $\Theta(\D\t)$
a.s., and that the
distribution of $U(\t)$ under $\Theta(\D\t)$ is exponential.
It follows that, for $\Theta$-almost every $\t\in\T$, we can also
define, for every $x\in\t_1$,
\[
U_\t(x)=\lim_{r\to\infty} e^{-r} \#
\t_{r-1}[x]= \frac{1}{e} U \bigl(\t[x] \bigr),
\]
and, if we set
\[
u_\t(x) = \frac{U_\t(x)}{U(\t)}
\]
the collection $(u_\t(x))_{x\in\t_1}$ is a probability distribution on
$\t_1$.

By a concavity argument, we have
\begin{equation}
\label{eq:Shannon} \sum_{x\in\t_1} \nu^*_\t(x)
\log \biggl(\frac{u_\t(x)}{\nu^*_\t(x)} \biggr) \leq0
\end{equation}
and the inequality is even strict if
$(\nu^*_\t(x))_{x\in\t_1}\neq(u_\t(x))_{x\in\t_1}$. It is easy to
verify that the latter property
holds with positive
probability under $\Theta$. To give a precise argument, recall the
notation used in Section~\ref{sec:yuletree}
to define the Yule tree $\Gamma$, and consider the event
\[
E:= \bigl\{\mathcal{Y}_\varnothing<\tfrac{1}{4},
\mathcal{Y}_1>1, \mathcal {Y}_2 >\tfrac{3}{4}
\bigr\} \cap\{\# \Gamma_1 = 9\},
\]
which clearly has positive probability. On this event, write $x_1,\ldots,x_9$
for the elements of $\Gamma_1$ listed in the lexicographical order of
$\mathcal{V}$. On the event $E$, we have
$\nu^*_\Gamma(x_1)>1/8$, because Brownian motion (with drift $1/2$) has
probability
$1/2$ to hit the ancestor of $x_1$ at height $3/4$ before\vspace*{1pt} the other
point of $\Gamma_{3/4}$, and then
probability at least $1-e^{-1/2}>1/4$ to escape to infinity before
returning to the first branching point of $\Gamma$.
On the other hand, the branching property of the Yule tree shows that,
conditionally on the event $E$ (which only
involves the part of the tree below height~$1$), the random variables
$u_\Gamma(x_1),\ldots,u_\Gamma(x_9)$ have the same distribution. It
follows that
we have $u_\Gamma(x_1)<\nu^*_\Gamma(x_1)$ with positive probability on $E$.

Next, we have
\begin{eqnarray*}
\beta&=& \Lambda^*(G_1)=\int\log \biggl(\frac{1}{\nu_\t(\mathcal{B}_\t
(\mathbf{v},1))} \biggr)
\Lambda^*(\D\t \,\D\mathbf{v})
\\
&=&\int\sum_{x\in\t_1} \nu^*_\t(x) \log
\biggl(\frac{1}{\nu^*_\t(x)} \biggr) \Lambda(\D\t)
< \int\sum_{x\in\t_1} \nu^*_\t(x) \log
\biggl(\frac{1}{u_\t(x)} \biggr) \Lambda(\D\t),
\end{eqnarray*}
where the strict inequality follows from~\eqref{eq:Shannon} and the
fact that $(\nu^*_\t(x))_{x\in\t_1}\neq(u_\t(x))_{x\in\t_1}$ with positive
probability under $\Theta$, hence also under $\Lambda$. Next, recalling the
Markov chain $(\mathcal{Z}_n)$ introduced in the proof of Proposition~\ref{ergodicity}, we have
\begin{eqnarray*}
\int\sum_{x\in\t_1} \nu^*_\t(x) \log
\biggl(\frac{1}{u_\t(x)} \biggr) \Lambda(\D\t)&=& \int\sum
_{x\in\t_1} \nu^*_\t(x) \log \biggl(\frac{e U(\t)}{U(\t[x])}
\biggr) \Lambda(\D\t)
\\
&=&1 + \int\log \biggl(\frac{U(\mathcal{Z}_0)}{U(\mathcal{Z}_1)} \biggr) \Lambda^*(\D\t \,\D\mathbf{v})
\\
&=&1
\end{eqnarray*}
because $\mathcal{Z}_0$ and $\mathcal{Z}_1$ have the same distribution
under $\Lambda^*$,
and we also use the fact that $\log U(\t)$ is integrable under $\Theta
(\D\t)$
hence under $\Lambda(\D\t)$. Together with the preceding display, this
completes the proof.
\end{pf}

\begin{pf*}{Proof of Proposition \ref{prop:valuebeta}}
The first
assertion of Proposition \ref{prop:valuebeta} follows from
Proposition~\ref{prop:unicite}. To complete the proof of
Proposition~\ref{prop:valuebeta}, we start by establishing the first
half of formula~\eqref{value-beta}, that is,
\begin{equation}
\label{formubeta} \beta= \frac{2\int\!\!\int\!\!\int\gamma(\D r)\gamma(\D s)\gamma(\D
t)\sklafrac{rs}{r+s+t-1}\log\sklvfrac{r+t}{r}}{
\int\!\!\int\gamma(\D r)\gamma(\D s) \sklafrac{rs}{r+s-1}}.
\end{equation}

We use the notation of the beginning of this section, and we first fix
$\ve>0$ and define a function $H_\ve$ on $\T^*$ by setting
\[
H_\ve(\t,\mathbf{v})= \cases{ 0, &\quad if $z_\varnothing\geq
\ve$,
\cr
-\log\nu_\t \bigl( \bigl\{\mathbf{v}'\in\{1,2
\}^\N: \mathbf{v}_1\prec\mathbf {v}'
\bigr\} \bigr), &\quad if $z_\varnothing< \ve$,} %
\]
where we write $\t=(z_v)_{v\in\v}$ as previously, and we recall the
notation $\mathbf{v}_n$
from the beginning of Section~\ref{sec:ergodic}. Clearly, $H_\ve(\t
,\mathbf{v})\leq G_\ve(\t,\mathbf{v})$, and $H_\ve(\t,\mathbf{v})=G_\ve
(\t,\mathbf{v})$
if $z_{\mathbf{v}_1}\geq\ve$. More generally, $H_\ve\circ\sigma_r(\t
,\mathbf{v})=G_\ve\circ\sigma_r(\t,\mathbf{v})$
if there is at most one index $i\geq0$ such that $r\leq z_{\mathbf
{v}_i} < r+\ve$. It follows from these remarks
that, for every integer $n\geq1$,
\begin{equation}
\label{ergodic-tech1} G_1\geq\sum_{k=0}^{n-1}
H_{1/n} \circ\sigma_{k/n}
\end{equation}
and, for every $(\t,\mathbf{v})\in\T^*$,
\begin{equation}
\label{ergodic-tech2} G_1(\t,\mathbf{v})=\lim_{n\to\infty} \sum
_{k=0}^{n-1} H_{1/n} \circ
\sigma_{k/n}(\t,\mathbf{v}).
\end{equation}

Let us then investigate the behavior of $\Lambda^*(H_\ve)$ when $\ve\to
0$. It will be convenient to write
$\t_{(1)}$ and $\t_{(2)}$ for the two ``subtrees'' of $\t$ obtained at
the first branching point [formally
$\t_{(i)}$ corresponds to the collection $(z_{iv}-z_\varnothing)_{v\in\v
}$, for $i=1$ or~$2$]. We observe that,
if $i=1$ or $i=2$, the exit ray of Brownian motion
on $\t$ will belong to $\{(i,v_2,v_3,\ldots): (v_2,v_3,\ldots)\in\{1,2\}
^\N\}$ with probability
\[
\frac{\cc(\t_{(i)})}{\cc(\t_{(1)})+\cc(\t_{(2)})}. %
\]
Thanks to this observation, we can write
\begin{eqnarray*}
&&\Lambda^*(H_\ve)
\\
&&\qquad = -\int\Theta(\D\t) \Phi_\infty \bigl(\cc(\t) \bigr) {
\mathbf1}_{\{z_\varnothing<\ve\}
} \biggl(\frac{\cc(\t_{(1)})}{\cc(\t_{(1)})+\cc(\t_{(2)})} \log\frac{\cc(\t_{(1)})}{\cc(\t_{(1)})+\cc(\t_{(2)})}
\\
&&\quad\qquad{} + \frac{\cc(\t_{(2)})}{\cc
(\t_{(1)})+\cc(\t_{(2)})} \log \frac{\cc(\t_{(2)})}{\cc(\t_{(1)})+\cc(\t_{(2)})} \biggr)
\\
&&\qquad =-2\int\Theta(\D\t) \Phi_\infty \bigl(\cc(\t) \bigr) {
\mathbf1}_{\{z_\varnothing
<\ve\}}\frac{\cc(\t_{(1)})}{\cc(\t_{(1)})+\cc(\t_{(2)})} \log\frac{\cc(\t_{(1)})}{\cc(\t_{(1)})+\cc(\t_{(2)})},
\end{eqnarray*}
by a symmetry argument. An easy calculation gives
\[
\cc(\t)= \frac{\cc(\t_{(1)})+\cc(\t_{(2)})}{e^{-z_\varnothing} +
(1-e^{-z_\varnothing})(\cc(\t_{(1)})+\cc(\t_{(2)}))}.
\]
Since, under $\Theta(\D\t)$, $\t_{(1)}$ and $\t_{(2)}$ are independent
and distributed according to
$\Theta$, and are also independent of $z_\varnothing$, we get
\begin{eqnarray*}
\Lambda^*(H_\ve) &=&-2\int\!\!\!\int\Theta(\D\t)\Theta \bigl(\D
\t' \bigr) \frac{\cc(\t)}{\cc(\t)+\cc
(\t')} \log\frac{\cc(\t)}{\cc(\t)+\cc(\t')}
\\
&&{} \times\int_0^\ve \D z e^{-z}
\Phi_\infty \biggl(\frac{\cc(\t)+\cc(\t')}{e^{-z} + (1-e^{-z})(\cc(\t)+\cc
(\t'))} \biggr).
\end{eqnarray*}
Note that the function $(\t,\t')\mapsto \frac{\cc(\t)}{\cc(\t)+\cc(\t')}
\log\frac{\cc(\t)}{\cc(\t)+\cc(\t')}$ is integrable with respect to
the measure $\Theta(\D\t)\Theta(\D\t')$, and that
$\Phi_\infty$ is bounded and continuous. We can thus let $\ve\to0$ in
the preceding expression and get
\begin{eqnarray}
\label{ergodic-tech3} \lim_{\ve\to0}\frac{1}{\ve}
\Lambda^*(H_\ve) &=& -2\int\!\!\!\int\Theta(\D\t)\Theta \bigl(\D
\t' \bigr) \Phi_\infty \bigl(\cc(\t)+\cc \bigl(\t
' \bigr) \bigr)
\nonumber\\[-8pt]\\[-8pt]\nonumber
&&{}\times \frac{\cc(\t)}{\cc(\t)+\cc(\t')} \log\frac{\cc(\t)}{\cc(\t)+\cc(\t')}.
\nonumber
\end{eqnarray}
Since the limit in the preceding display is finite, we can use (\ref
{ergodic-tech2}) and Fatou's lemma to get that
$\Lambda^*(G_1)<\infty$, and then (\ref{ergodic-tech1}) (to justify
dominated convergence) and (\ref{ergodic-tech2}) again
to obtain that
\[
\Lambda^*(G_1) =\lim_{n\to\infty} n
\Lambda^*(H_{1/n})
\]
coincides with the right-hand side of (\ref{ergodic-tech3}).
Finally, we use the expression of $\Phi_\infty$ to obtain
formula~\eqref{formubeta}.

We will now establish the second half of formula~\eqref{value-beta},
which will
complete the proof of Proposition~\ref{prop:valuebeta}. We let $\cc
_0,\cc_1,\cc_2$
be independent and distributed according to $\gamma$
under the probability measure $\P$. Then the denominator of
the right-hand side of~\eqref{formubeta} can be written as
\[
\E \biggl[ \frac{ \mathcal{C}_{0} \mathcal{C}_{1}}{ \mathcal{C}_{0}+
\mathcal{C}_{1}-1} \biggr].
\]
On the other hand, the numerator is equal to
\begin{eqnarray*}
&& 2 \E \biggl[\frac{ { \mathcal{C}_{0} \mathcal{C}_{1}}}{ \mathcal
{C}_{0}+ \mathcal{C}_{1}+ \mathcal{C}_{2}-1} \log \biggl(\frac{ \mathcal
{C}_{1}+ \mathcal{C}_{2}}{ \mathcal{C}_{1}} \biggr) \biggr]
\\
&& \qquad= \E \biggl[\frac{ { \mathcal{C}_{0}( \mathcal{C}_{1}+ \mathcal
{C}_{2})}\log( \mathcal{C}_{1}+ \mathcal{C}_{2})}{ \mathcal{C}_{0}+
\mathcal{C}_{1}+ \mathcal{C}_{2}-1} \biggr] - \E \biggl[\frac{ {( \mathcal
{C}_{0}+ \mathcal{C}_{2}) \mathcal{C}_{1}}\log( \mathcal{C}_{1})}{
\mathcal{C}_{0}+ \mathcal{C}_{1}+ \mathcal{C}_{2}-1}
\biggr]
\\
&& \qquad = \E \bigl[ f( \mathcal{C}_{1} + \mathcal{C}_{2})
\bigr] - \E \bigl[g( \mathcal{C}_{1}+ \mathcal{C}_{2})
\bigr],
\end{eqnarray*}
where we have set, for every $x\geq1$,
\[
f(x) = \E \biggl[ \frac{ \mathcal{C}_0x}{ \mathcal{C}_0+x-1}\log x \biggr]\quad\mbox{and}\quad g(x) = \E
\biggl[ \frac{ \mathcal{C}_0x}{ \mathcal
{C}_0+x-1}\log\mathcal{C}_0 \biggr]. %
\]
Using~\eqref{eq:f}, we replace $\E[f( \mathcal{C}_{1}+ \mathcal
{C}_{2})]$ by $\E[f( \mathcal{C}_{1})] + \E[ \mathcal{C}_{1}( \mathcal
{C}_{1}-1)f'( \mathcal{C}_{1})] $, and similarly for $g$, to obtain
\begin{eqnarray*}
&& \E \bigl[ f( \mathcal{C}_{1} + \mathcal{C}_{2}) \bigr] -
\E \bigl[g( \mathcal{C}_{1}+ \mathcal{C}_{2}) \bigr]
\\
&&\qquad =
\E \biggl[ \frac{ \mathcal{C}_{0}
\mathcal{C}_{1}}{ \mathcal{C}_{0}+ \mathcal{C}_{1}-1} \log\mathcal {C}_{1} \biggr]
\\
&&\quad\qquad{}+ \E \biggl[ \frac{ \mathcal{C}_{0}( \mathcal
{C}_{0}-1)\mathcal{C}_{1}( \mathcal{C}_{1}-1)}{( \mathcal{C}_{0}+
\mathcal{C}_{1}-1)^2} \log\mathcal{C}_{1} \biggr] +\E
\biggl[ \frac{
\mathcal{C}_{0} \mathcal{C}_{1}(\mathcal{C}_{1}-1)}{ \mathcal{C}_{0}+
\mathcal{C}_{1}-1} \biggr]
\\
&&\quad\qquad{} - \E \biggl[ \frac{ \mathcal{C}_{0} \mathcal{C}_{1}}{
\mathcal{C}_{0}+ \mathcal{C}_{1}-1}\log\mathcal{C}_{0} \biggr] - \E
\biggl[ \frac{ \mathcal{C}_{0}( \mathcal{C}_{0}-1)\mathcal{C}_{1}(
\mathcal{C}_{1}-1)}{( \mathcal{C}_{0}+ \mathcal{C}_{1}-1)^2}\log \mathcal{C}_{0} \biggr]
\\
&&\qquad = \E \biggl[ \frac{ \mathcal{C}_{0} \mathcal{C}_{1}( \mathcal
{C}_{0}-1)}{ \mathcal{C}_{0}+ \mathcal{C}_{1}-1} \biggr]
\\
&&\qquad =\frac{1}{2} \E \biggl[ \frac{ \mathcal{C}_{0} \mathcal{C}_{1}(
\mathcal{C}_{0}+ \mathcal{C}_{1}-1) - \mathcal{C}_{0} \mathcal{C}_{1}}{
\mathcal{C}_{0}+ \mathcal{C}_{1}-1} \biggr]
\\
&&\qquad =\frac{1}{2} \biggl( \E[ \mathcal{C}_{0}]^2 - \E
\biggl[ \frac{
\mathcal{C}_{0} \mathcal{C}_{1}}{ \mathcal{C}_{0}+ \mathcal{C}_{1}-1} \biggr] \biggr).
\end{eqnarray*}
If we substitute this in~\eqref{formubeta}, we arrive at
\[
2 \beta= \frac{\E[ \mathcal{C}_{0}]^2}{\E [ \afrac{ \mathcal
{C}_{0} \mathcal{C}_{1}}{ \mathcal{C}_{0}+ \mathcal{C}_{1}-1}  ]} -1,
\]
which gives the second half of~\eqref{value-beta} and completes the
proof of Proposition~\ref{prop:valuebeta}.
\end{pf*}

\begin{rem*}
Despite all that is known about the distribution $\gamma$ (see
Section~\ref{sec:conductance}), it requires
some work to derive the fact that $\beta<1$ (Proposition~\ref
{prop:beta<1}) from the explicit formulas of Proposition~\ref{prop:valuebeta}.
The approximate numerical value $\beta=0.78\ldots$ is obtained by first
estimating $\gamma$ using
Proposition \ref{prop:unicite} (or more precisely the convergence of
$\Phi^k(\lambda)$ to $\gamma$, for any
probability measure $\lambda$ on $[1,\infty)$), and then applying a
Monte--Carlo method to evaluate the
integrals in the right-hand side of~(\ref{value-beta}).
\end{rem*}

\section{Discrete random trees}
\renewcommand{\t}{\mathsf{T}}

\label{sec:continuous->discrete}
In this section, we prove Theorem~\ref{thm:maindiscrete} and
Corollary~\ref{cor:planetree}. We first explain why discrete reduced trees
converge modulo a suitable rescaling toward the continuous reduced tree
$\Delta$. This leads to a first connection between the discrete
harmonic measures and the
continuous one (Proposition~\ref{convergence-harmonic}). Combining this result with
Theorem~\ref{thm:main-harmonic}, one gets a first estimate in the
direction of Theorem~\ref{thm:maindiscrete} (Corollary~\ref{cor:comparaison}).
The recursive
properties of Galton--Watson trees are then used to complete the proof
of Theorem~\ref{thm:maindiscrete}. Corollary~\ref{cor:planetree} is proved at
the end of the section.

\subsection{Notation for trees}
\label{sec:tree-discrete}

We consider discrete rooted ordered trees, which are
also called plane trees in combinatorics. A plane tree $ \tau$ is a
finite subset of
\[
\mathcal{U}=\bigcup_{n=0}^\infty
\N^n,
\]
where $\N^0=\{\varnothing\}$, such that the following holds:
\begin{longlist}[(iii)]
\item[(i)] $\varnothing\in\tau$.

\item[(ii)] If $u=(u_1,\ldots,u_n)\in\tau\setminus\{\varnothing\}$ then
$\widehat u:=(u_1,\ldots,u_{n-1})\in\tau$.

\item[(iii)] For every $u=(u_1,\ldots,u_n)\in\tau$, there exists an
integer $k_u(\tau)\geq0$
such that, for every $j\in\N$, $(u_1,\ldots,u_n,j)\in\tau$ if and only
if $1\leq j\leq
k_u(\tau)$.
\end{longlist}

In this section, we say tree instead of plane tree. We often view a
tree $\tau$ as a graph whose vertices are the elements of $\tau$ and
whose edges are
the pairs $\{\widehat u,u\}$ for all $u\in \tau\setminus\{\varnothing
\}$.

We will use the notation and terminology introduced at the beginning of
Section~\ref{sec:treedelta}
in a slightly different setting. In particular,
$\llvert  u\rrvert  $ is the generation of $u$, $uv$ denotes the concatenation of $u$
and $v$, $\prec$
stands for the genealogical order and $u\wedge v$ is the maximal
element of
$\{w\in\mathcal{U}: w\prec u$ and $w\prec v\}$.

The height of
a tree $\tau$ is
\[
h(\tau)=\max \bigl\{\llvert v\rrvert:v\in\tau \bigr\}.
\]
We write $\mathscr{T}$
for the set of all trees, and $\mathscr{T}_n$ for the set of all trees
with height $n$.

Let $\tau$ be a tree.
The set $\tau$ is equipped with the distance
\[
d(v,w)=\tfrac{1}{2} \bigl(\llvert v\rrvert +\llvert w\rrvert -2\llvert v
\wedge w\rrvert \bigr).
\]
Notice that this is half the usual graph distance. We will write $B_\tau(v,r)$,
or simply $B(v,r)$ if there is no ambiguity,
for the closed ball of radius $r$ centered at $v$, with respect to the
distance $d$,
in the tree $\tau$.

The set of all vertices of $\tau$ at generation $n$ is denoted by
\[
\tau_n:= \bigl\{v\in\tau: \llvert v\rrvert =n \bigr\}.
\]
If $v\in\tau$, the subtree of descendants of $v$ is
\[
\widetilde\tau[v]:= \bigl\{v'\in\tau: v\prec v'
\bigr\}.
\]
Note that $\widetilde\tau[v]$ is not a tree with our definitions, but
we turn it into
a tree by relabelling its vertices, setting
\[
\tau[v]:= \{w\in\mathcal{U}: vw\in\tau\}.
\]

If $v\in\tau$, then for every $i\in\{0,1,\ldots,\llvert  v\rrvert  \}$ we write
$\langle v\rangle_i$ for the ancestor of $v$
at generation $i$. Suppose that $\llvert  v\rrvert  =n$.
Then $B_\tau(v,i)\cap\tau_n= \widetilde\tau[\langle v\rangle
_{n-i}]\cap\tau_n$, for every
$i\in\{0,1,\ldots,n\}$. This simple observation will be used several
times below.

\subsubsection*{Galton--Watson trees} Let $\theta$ be a probability measure
on $\Z_+$, and assume
that $\theta$ has mean one and finite variance $\sigma^2>0$. There
exists a unique probability measure
${\rm GW}_\theta(\D\tau)$ on $\mathscr{T}$ such that the following two
properties hold:
\begin{longlist}[(ii)]
\item[(i)] The law of $k_\varnothing(\tau)$ under ${\rm GW}_\theta(\D
\tau)$ is $\theta$.
\item[(ii)] Let $k\geq1$ such that $\theta(k)>0$. Then under ${\rm
GW}_\theta(\D\tau \mid  k_\varnothing(\tau)=k)$, the subtrees
$\tau[1],\ldots,\tau[k]$ are independent and distributed according to
${\rm GW}_\theta$.
\end{longlist}
A random tree distributed according to ${\rm GW}_\theta$ will be called
a
Galton--Watson tree with offspring distribution $\theta$ (see, e.g.,
\cite{probasur} for a discussion of Galton--Watson trees).

For\vspace*{1pt} every integer $n\geq0$, we let $\t^{(n)}$ be a
Galton--Watson tree with offspring distribution $\theta$, conditioned
on nonextinction at generation $n$.
In particular, $\t^{(0)}$ is just a
Galton--Watson tree with offspring distribution $\theta$. We suppose
that the random trees
$\t^{(n)}$ are defined under the probability measure $\P$.

We let $\t^{*n}$ be the reduced tree associated with $\t^{(n)}$, which consists
of all vertices of $\t^{(n)}$ that have (at least) one descendant at
generation $n$.
A priori $\t^{*n}$ is not a tree in the sense of the preceding
definition. However
we can relabel the vertices of~$\t^{*n}$, preserving both the
lexicographical order
and the genealogical order, so that $\t^{*n}$ becomes a tree
in the sense of our definitions. We will always assume that this
relabelling has
been done.

Note that $\llvert  u\rrvert  \leq n$ for every $u\in\t^{*n}$.
It will be convenient to introduce truncations of $\t^{*n}$. For every
$s\in[0,n]$, we set
\[
R_s \bigl(\t^{*n} \bigr)= \bigl\{v\in\t^{*n}:
\llvert v\rrvert \leq n-\lfloor s\rfloor \bigr\}.
\]

We then consider
simple random walk on $\t^{*n}$, starting from the root $\varnothing$,
which we denote by
$Z^n=(Z^n_k)_{k\geq0}$. This random walk is defined under the
probability measure $P$ (as previously,
it is important to distinguish the probability measures governing the
trees on one hand, the random walks on the
other hand).

We let
\[
H_n=\inf \bigl\{k\geq0: \bigl\llvert Z^n_k
\bigr\rrvert =n \bigr\}
\]
be the first hitting time of generation $n$ by $Z^n$, and we set
\[
\Sigma_n= Z^n_{H_n}.
\]
The discrete harmonic measure $\mu_n$,
is the law of $\Sigma_n$ under $P$.
Notice that $\mu_n$ is a probability measure
on the set $\t^{*n}_n$ of all vertices of $\t^{*n}$
at generation $n$.

We start with a lemma that gives bounds on the size of level sets in $
\t^{*n}$.

%
\begin{lemma}
\label{lem:estimate-reduced-tree}
There exists a constant $C$ depending only on
$\theta$ such that, for every integer $n\geq2$ and every integer $p$ such
that $1\leq p\leq n/2$,
\[
\E \bigl[ \bigl(\log\#\t^{*n}_{n-p} \bigr)^4
\bigr]^{1/4} \leq C \log\frac{n}{p} \quad\mbox{and}\quad \E \bigl[
\bigl(\log\#\t^{*n}_{n} \bigr)^4
\bigr]^{1/4} \leq C \log n.
\]
\end{lemma}

\begin{pf}
Set $q_n=\P(h(\t^{(0)})\geq n)$. By a standard result (Theorem 9.1 of
\cite{AN}, Chapter~1), we have
\begin{equation}
\label{survivalpro} q_n \sim\frac{2}{n\sigma^2}\qquad\mbox{as } n \to
\infty.
\end{equation}
Then, for every $p\in\{0,1,\ldots,n\}$,
\begin{eqnarray*}
\E \bigl[\#\t^{*n}_{n-p} \bigr]&=&\E \bigl[\# \bigl\{v\in
\t^{(n)}_{n-p}: h \bigl(\t^{(n)}[v] \bigr)\geq p \bigr
\} \bigr]
\\
&=&(q_n)^{-1} \E \bigl[\# \bigl\{v\in\t^{(0)}_{n-p}
: h \bigl(\t^{(0)}[v] \bigr)\geq p \bigr\} \bigr].
\end{eqnarray*}
By the branching property of Galton--Watson trees, the conditional distribution
of $\#\{v\in\t^{(0)}_{n-p}: h(\t^{(0)}[v])\geq p\}$ knowing that
$\#\t^{(0)}_{n-p}=k$ is the binomial distribution $\mathcal{B}(k,q_p)$. Hence,
\[
\E \bigl[\#\t^{*n}_{n-p} \bigr] = \frac{q_p \E[\#\t^{(0)}_{n-p}]}{q_n}=
\frac{q_p}{q_n}.
\]
We can find $a>0$ such that the function $x\la(\log(a+x))^4$
is concave over $[1,\infty)$. Then
\begin{eqnarray*}
\E \bigl[ \bigl(\log\# \t^{*n}_{n-p} \bigr)^4
\bigr]^{1/4} &\leq& \E \bigl[ \bigl(\log \bigl(a+ \#\t ^{*n}_{n-p}
\bigr) \bigr)^4 \bigr]^{1/4} \leq \log \bigl(a+ \E \bigl[
\# \t^{*n}_{n-p} \bigr] \bigr)
\\
&=& \log \biggl(a+\frac{q_p}{q_n} \biggr),
\end{eqnarray*}
and the bounds of the lemma easily follow from (\ref{survivalpro}).
\end{pf}

\subsection{Discrete and continuous reduced trees}
\label{sec:convergences}
\subsubsection{Convergence of discrete reduced trees}
Recall from Section~\ref{sec:treedelta} the definition of the
continuous reduced tree ${\Delta}$. For every $\ve\in(0,1)$, we have
set $\Delta_\ve=\{x\in\Delta: H(x)\leq1-\ve\}$. We will implicitly use
the fact that, for every fixed $\ve$, there is a.s. no branching point
of $\Delta$ at height $1-\ve$. The skeleton of $\Delta_\ve$ is defined as
\begin{eqnarray*}
\operatorname{Sk}(\Delta_\ve) &=& \{\varnothing\}\cup \bigl\{v\in
\mathcal{V} \setminus\{\varnothing\}: Y_{\widehat v}\leq1-\ve \bigr\}
\\
&=& \{\varnothing\}\cup \bigl\{v\in\mathcal{V}\setminus\{\varnothing\}: (
\widehat v,Y_{\widehat v})\in\Delta_\ve \bigr\}.
\end{eqnarray*}

Consider then a tree $\tau\in\mathscr{T}$ such that
every vertex of $\tau$ has either $0,1$ or $2$ children.
It will be convenient to write $\mathscr{T}_{\rm bin}$ for the
collection of all such trees. With $\tau$ we associate
another tree denoted by $[\tau]$, which is obtained by ``removing'' all
vertices that have exactly one
child. More precisely, write $\mathcal{S}(\tau)$ for the set of all vertices
$v$ of $\tau$ having $0$ or $2$ children. Then we can find a unique
tree $[\tau]$ such that there exists a bijection $u\la w_u$ from $[\tau
]$ onto $\mathcal{S}(\tau)$ that preserves
both the genealogical
order and the lexicographical order of vertices. We call this bijection
the canonical bijection
from $[\tau]$ onto $\mathcal{S}(\tau)$.

%
\begin{proposition}
\label{conv-reduced-tree}
We can construct the reduced trees $\t^{*n}$ and the (continuous) tree
$\Delta$
on the same probability space $(\Omega,\mathcal{F},\P)$ so that the
following properties hold
for every fixed $\ve\in(0,1)$ with $\P$-probability one.
\begin{longlist}[(ii)]
\item[(i)] For every
sufficiently large integer $n$, we have $R_{\ve n}(\t^{*n})\in\mathscr
{T}_{\rm bin}$ and $[R_{\ve n}(\t^{*n})] =\operatorname{Sk}(\Delta_\ve)$.
\item[(ii)] For every sufficiently large $n$, such that the
properties stated in {\rm(i)} hold, and for every $u\in\operatorname{Sk}(\Delta_\ve)$, let $w^{n,\ve}_u$ denote the vertex
of $\mathcal{S}(R_{\ve n}(\t^{*n}))$ corresponding to $u$ via the
canonical bijection from $[R_{\ve n}(\t^{*n})]$
onto $\mathcal{S}(R_{\ve n}(\t^{*n}))$. Then we have
\[
\lim_{n\to\infty} \frac{1}{n} \bigl\llvert
w^{n,\ve}_u \bigr\rrvert = Y_{u} \wedge(1-\ve).
\]
\end{longlist}
\end{proposition}

%
\begin{figure}[b]

\includegraphics{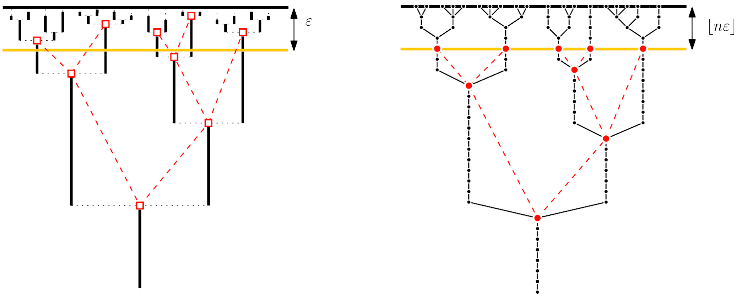}

\caption{Setting of Proposition~\protect\ref{conv-reduced-tree}. On the left, the
tree $\Delta$, its truncation $\Delta_{ \varepsilon}$ and the skeleton
$ \operatorname{Sk}( \Delta_{ \varepsilon})$. On the right, a large reduced
tree $\t^{*n}$ of height $n$, its truncation $R_{ \varepsilon n}( \t
^{*n})$ and the associated binary tree $[R_{ \varepsilon n}( \t^{*n})]$.}\label{fig6}
\end{figure}

See Figure~\ref{fig6} for an illustration of
Proposition~\ref{conv-reduced-tree}. This proposition is essentially a consequence
of classical results on the convergence in distribution of
reduced critical Galton--Watson trees, see in particular
\cite{Zu75} and \cite{FSS}. A simple way of proving Proposition~\ref{conv-reduced-tree} is to use\vspace*{2pt}
the convergence
in distribution of the rescaled contour functions associated with the
trees $\t^{(n)}$
toward a Brownian excursion with height greater than $1$ (see
\cite{probasur}, Corollary 1.13). By using the Skorokhod
representation theorem, one may assume that
the trees $\t^{(n)}$ and the Brownian excursion are constructed so
that the latter convergence holds
almost surely. We then use the relation between the
Brownian excursion with height greater than $1$ and the continuous
reduced tree $\Delta$, which can be found in
\cite{LG89}, Section~5 (this is a particular case of a more general
result connecting reduced L\'evy trees
with the so-called height process, see \cite{DLG}, Section~2.7). Let
us briefly explain this relation.

%
\begin{figure}[t]

\includegraphics{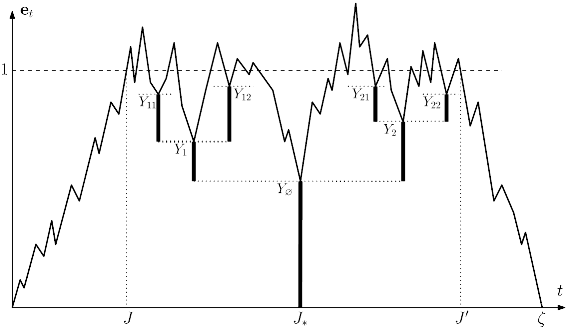}

\caption{The relation between the Brownian excursion with height
greater than $1$ and the continuous reduced tree $\Delta$. For each
$0<a<1$, the number
of vertices of $\Delta$ at height $a$ corresponds to the number of
``subexcursions'' above height $a$ that hit level $1$.}
\label{redu-excursion}
\end{figure}

We write $(\mathbf{e}_t)_{0\leq t\leq\zeta}$
for a Brownian excursion
conditioned to hit level $1$. We associate with this process a
collection $(Y_v)_{v\in\mathcal{V}}$ of nonnegative random variables
defined as follows. If $J=\inf\{t\in[0,\zeta]:\mathbf{e}_t=1\}$,
$J'=\sup\{t\in[0,\zeta]:\mathbf{e}_t=1\}$, we set $Y_\varnothing:=\min\{
\mathbf{e}_t:J\leq t\leq J'\}$ and we also
let $J_*$ be the a.s. unique time in $[J,J']$ such that $\mathbf
{e}_{J_*}=Y_\varnothing$. Then we let $Y_1$ be the minimum of $\mathbf
{e}$ between $J$ and $\sup\{t\leq J_*: \mathbf{e}_t=1\}$,
and $Y_2$ be the minimum of $\mathbf{e}$ between $\inf\{t\geq
J_*:\mathbf{e}_t=1\}$ and $J'$. The construction is continued by
induction (compare Figures~\ref{redu-excursion} and~\ref{fig:Delta}). According to \cite{LG89}, Section~5, the
collection $(Y_v)_{v\in\mathcal{V}}$ has the distribution described in
Section~\ref{sec:treedelta}: this shows that the
tree $\Delta$ can be embedded in the graph of $\mathbf{e}$ in the way
suggested by Figure~\ref{redu-excursion}. Moreover, for every $a\in
[0,1)$, the number of vertices of the tree $\Delta$ at height $a$
corresponds to the number of excursions of $\mathbf{e}$ above level $a$ that
hit height $1$ (to be precise, a vertex which is a branching point of
the tree should be counted twice).

Once we know that the rescaled contour functions associated with the
trees $\t^{(n)}$ converge a.s. to a Brownian excursion with height
greater than $1$, the various assertions of Proposition \ref
{conv-reduced-tree} follow, with the
continuous reduced tree $\Delta$ constructed as explained above from
the limiting Brownian excursion.
We leave the details to the reader.

Let us comment on the properties stated in Proposition~\ref{conv-reduced-tree}. In
property (ii), we have $Y_u>1-\ve$ if and only if $u$ is a leaf (i.e.,
a vertex with no child) of $\operatorname{Sk}(\Delta_\ve)$. Furthermore,
if $u$ is a vertex of $\operatorname{Sk}(\Delta_\ve)$ which is not a leaf, the
vertex $w^{n,\ve}_u$, which
is well defined for $n$ large enough,
does not depend on $\ve$.
More precisely, suppose that $0<\delta<\ve$, and suppose that
$n$ is sufficiently large so that the properties stated in (i) hold
as well as the same properties with $\ve$ replaced by $\delta$. Then,
if $u\in\operatorname{Sk}(\Delta_\ve)$
is not a leaf of $\operatorname{Sk}(\Delta_\ve)$, we must have $w^{n,\ve
}_u=w^{n,\delta}_u$.
On the other hand, if $u$ is a leaf of $\operatorname{Sk}(\Delta_\ve)$, then we
must have
$\llvert  w^{n,\ve}_u\rrvert  =n-\lfloor\ve n\rfloor$, and $w^{n,\ve}_u$
is an ancestor of $w^{n,\delta}_u$. We leave the verification of these
properties
to the reader.

\subsubsection{Convergence of conductances}

Let $i$ be a positive integer and let $\tau\in\mathscr{T}$ be a tree
such that $h(\tau)\geq i$. Consider the new graph $\tau'$
obtained by adding to the graph $\tau$ an edge between the root
$\varnothing$ and
an extra vertex $\partial$. We let $\mathcal{C}_i(\tau)$ be the
probability that simple
random walk on $\tau'$ starting from $\varnothing$ hits generation $i$
of $\tau$ before
hitting the vertex $\partial$.
The notation is justified by the fact that
${\mathcal C}_i(\tau)$ can be interpreted as the effective conductance
between $\partial$ and generation $i$ of
$\tau$ in the graph $\tau'$; see \cite{LP10}, Chapter~2.

%
\begin{proposition} Suppose that
the reduced trees $\t^{*n}$ and the (continuous) tree $\Delta$ are
constructed so that the properties stated in Proposition~\ref{conv-reduced-tree}
hold, and that the Yule tree $\Gamma$ is obtained from $\Delta$
as explained in Section~\ref{sec:yuletree}. Then
\begin{eqnarray*}
n \mathcal{C}_{n} \bigl( \t^{*n} \bigr) &
\mathop {\longrightarrow}\limits_{n\to\infty}^{\mathrm{a.s.}} & \mathcal
{C}(\Gamma).
\end{eqnarray*}
\end{proposition}

We omit the easy proof, as this result is not needed for the proof of
Theorem \ref{thm:maindiscrete}.

\label{sec:convergenceconductance}

\subsubsection{Convergence of harmonic measures}
\label{sec:conv-harmonic}
Our goal is now to verify that the discrete harmonic measures $\mu_n$
converge in some sense to the continuous harmonic measure $\mu$ defined
in Section~\ref{sec:treedelta}.

For every
$x\in\partial\Delta_\ve=\{z\in\Delta: H(z)=1-\ve\}$, we set
\[
\mu^\ve(x)=\mu \bigl(\{y\in\partial\Delta: x\prec y\} \bigr).
\]
Similarly, we define a probability measure $\mu^\ve_n$ on $\t
^{*n}_{n-\lfloor\ve n\rfloor}$
by setting
\[
\mu^\ve_n(u)= \mu_n \bigl(\{ v \in
\t_{n}: u \prec v\} \bigr),
\]
%
for every $u\in\t^{*n}_{n-\lfloor\ve n\rfloor}$. Clearly, $\mu_n^\ve$
is also the distribution
of $\langle\Sigma_n \rangle_{n-\lfloor\ve n\rfloor}$.

%
\begin{proposition}
\label{convergence-harmonic}
Suppose that
the reduced trees $\t^{*n}$ and the (continuous) tree $\Delta$ have been
constructed so that the properties of Proposition~\ref{conv-reduced-tree}
hold, and recall the notation $(w^{n,\ve}_u)_{u\in\operatorname{Sk}(\Delta_\ve
)}$ introduced in this proposition.
Then $\P$ a.s. for every $x=(v,1-\ve)\in\partial\Delta_\ve$,
\[
\lim_{n\to\infty} \mu^\ve_n
\bigl(w^{n,\ve}_v \bigr) = \mu^\ve(x).
\]
\end{proposition}

\begin{pf}
Let $\delta\in(0,\ve)$ and set $T_\delta=\inf\{t\geq0: H(B_t)=1-\delta
\} <T$. Define a
probability measure $\mu^{\ve,(\delta)}$ on $\partial\Delta_\ve$
by setting for every $x\in\partial\Delta_\ve$,
\[
\mu^{\ve,(\delta)}(x)= P(x\prec B_{T_\delta}).
\]
Similarly, we write $\mu^{(\delta)}_n$ for the distribution of the
hitting point of
generation $n-\lfloor\delta n\rfloor$ by random walk on $\t^{*n}$
started from $\varnothing$, and we define a probability measure $\mu
^{\ve,(\delta)}_n$
on $\t^{*n}_{n-\lfloor\ve n\rfloor}$ by setting
\[
\mu^{\ve,(\delta)}_n (v)=\mu_n^{(\delta)} \bigl(
\bigl\{w\in\t^{*n}_{n-\lfloor
\delta n\rfloor}:v\prec w \bigr\} \bigr),
\]
for every $v\in\t^{*n}_{n-\lfloor\ve n\rfloor}$.\vadjust{\goodbreak}

It is easy to verify that
\[
\lim_{\delta\to0} \mu^{\ve,(\delta)}(x)= \mu^\ve(x)
\]
for every $x\in\partial\Delta_\ve$, $\P$-a.s. Indeed we have the bound
$\llvert  \mu^{\ve,(\delta)}(x)- \mu^\ve(x)\rrvert  \leq\delta/\ve$, which follows
from the fact that
there is probability at least $1-\delta/\ve$ that after time $T_\delta$
Brownian motion will hit the
boundary $\partial\Delta$ before returning to height $1-\ve$ (and if
this event occurs then
for $x\in\partial\Delta_\ve$, we have $x\prec B_T$ if and only if
$x\prec B_{T_\delta}$). By similar arguments, one has $\P$-a.s.
\[
\lim_{\delta\to0} \Bigl(\limsup_{n\to\infty} \Bigl( \sup
_{v\in\t^{*n}_{n-\lfloor\ve n\rfloor}} \bigl\llvert \mu^{\ve,(\delta)}_n(v)-
\mu^\ve_n(v) \bigr\rrvert \Bigr) \Bigr)=0.
\]

In view of the preceding remarks, the convergence of the proposition
will follow if
we can verify that for every fixed $\delta\in(0,\ve)$, we have a.s. for every
$x=(u,1-\ve)\in\partial\Delta_\ve$,
\begin{equation}
\label{convharmo1} \lim_{n\to\infty} \mu^{\ve,(\delta)}_n
\bigl(w^{n,\ve}_u \bigr) = \mu^{\ve,(\delta)}(x).
\end{equation}

By considering the successive
passage times of Brownian motion stopped at time $T_\delta$ in the set
$\{(v,Y_v\wedge(1-\delta)): v\in\operatorname{Sk}(\Delta_\delta)\}$,
we get a Markov chain $X^{(\delta)}$, which is absorbed in the set $\{
(v,1-\delta):v$ is a leaf of $\operatorname{Sk}(\Delta_\delta)\}$, and
whose transition kernels are explicitly described in terms of
the quantities $Y_v,  v\in\operatorname{Sk}(\Delta_\delta)$.

Let $n$ be sufficiently large so that assertions (i) and (ii) of Proposition~\ref{conv-reduced-tree} hold with $\ve$ replaced by $\delta$, and
consider random walk on $\t^{*n}$ started from $\varnothing$ and
stopped at the first hitting time
of generation $n-\lfloor\delta n\rfloor$. By considering the
successive passage times of this random walk in the set
$\{ w^{n,\delta}_v: v\in \operatorname{Sk}(\Delta_\delta)\}$, we again get a
Markov chain $X^{(\delta),n}$,
which is absorbed in the set $\{w^{n,\delta}_v:v$ is a leaf of
$\operatorname{Sk}(\Delta_\delta)\}$
and whose transition kernels are explicit in terms
of the quantities $\llvert  w^n_v\rrvert  $, $v\in\operatorname{Sk}(\Delta_\delta)$.

Identifying both sets $\{(v,Y_v\wedge(1-\delta)): v\in\operatorname{Sk}(\Delta
_\delta)\}$
and $\{ w^{n,\delta}_v: v\in \operatorname{Sk}(\Delta_\delta)\}$ with
$\operatorname{Sk}(\Delta_\delta)$, we can view $X^{(\delta)}$ and
$X^{(\delta),n}$ as Markov chains with values in the set $\operatorname{Sk}(\Delta_\delta)$,
and then assertion (ii) of Proposition~\ref{conv-reduced-tree} implies that the
transition kernels of $X^{(\delta),n}$ converge to those of $X^{(\delta)}$.
Write $X^{(\delta)}_\infty$ for the absorption point of $X^{(\delta)}$,
and similarly write
$X^{(\delta),n}_\infty$ for the absorption\vspace*{1pt} point of $X^{(\delta),n}$.
We thus obtain that
the distribution of $X^{(\delta),n}_\infty$ converges to that of
$X^{(\delta)}_\infty$.
Consequently, for every $u\in\mathcal{V}$ such that $x=(u,1-\ve)\in
\partial\Delta_\ve$, we have
\[
\lim_{n\to\infty} P \bigl(u\prec X^{(\delta),n}_\infty
\bigr) = P \bigl(u\prec X^{(\delta
)}_\infty \bigr).
\]
However, from our definitions, we have
\[
P \bigl(u\prec X^{(\delta)}_\infty \bigr) = \mu^{\ve,(\delta)}(x),
\]
and, for $n$ sufficiently
large, noting that $w^{n,\ve}_u$ coincides
with the ancestor of $w^{n,\delta}_u$ at generation $n-\lfloor\ve
n\rfloor$
(see the remarks after Proposition~\ref{conv-reduced-tree}),
\[
P \bigl(u\prec X^{(\delta),n}_\infty \bigr)= \mu_n^{\ve,(\delta)}
\bigl(w^{n,\ve}_u \bigr).
\]
This completes the proof of (\ref{convharmo1}) and of the proposition.\vadjust{\goodbreak}
\end{pf}

Recall that, if $v\in\mathcal{U}$, $\langle v\rangle_{i}$ is the
ancestor of $v$ at generation $i \leq\llvert  v\rrvert  $.

%
\begin{corollary} \label{cor:comparaison} Let $ \xi\in(0,1)$. We can
find $ \varepsilon_{0}\in(0,1/2)$ such that
the following holds. For every $ \varepsilon\in(0, \varepsilon_{0})$,
there exists $n_{0} \geq0$ such that for every $n \geq n_{0}$ we have
\begin{eqnarray*}
\mathbb{E} \otimes E \bigl[ \bigl\llvert \log \mu_{n}^
\varepsilon
\bigl( \langle\Sigma_{n}\rangle_{n- \lfloor\varepsilon n
\rfloor} \bigr) - \beta\log
\varepsilon \bigr\rrvert ^2 \bigr] & \leq& \xi \llvert \log\varepsilon
\rrvert ^2.
\end{eqnarray*}
\end{corollary}

\begin{pf} Recall our notation $\mathscr{B}_\mathbf{d}(x,r)$ for the
closed ball
of radius $r$ centered at $x\in\Delta$.
Fix $\eta\in(0,1)$. Since $B_T$ is distributed according to $\mu$, it
follows from Theorem~\ref{thm:main-harmonic} that
there exists $ \varepsilon_{0}\in(0,1/2)$ such that for every $
\varepsilon\in(0, \varepsilon_{0})$ we have
\begin{equation}
\label{compatech1} \mathbb{P} \otimes P \bigl( \bigl\llvert \log\mu \bigl(
\mathscr{B}_\mathbf {d}(B_T, 2\varepsilon) \bigr) - \beta
\log \varepsilon \bigr\rrvert > (\eta/2)\llvert \log\varepsilon\rrvert \bigr) <
\eta/2.
\end{equation}

Let us fix $\ve\in(0,\ve_0)$. We now claim that, under $\P\otimes P$,
\begin{equation}
\label{convharmocont} \mu_{n}^ \varepsilon \bigl( \langle
\Sigma_{n}\rangle_{n- \lfloor
\varepsilon n \rfloor} \bigr)\mathop{\longrightarrow}_{n\to\infty}^{(\mathrm{d})}
\mu \bigl(\mathscr{B}_\bd(B_T, 2\varepsilon) \bigr).
\end{equation}
To see this, let $f$ be a continuous function on $[0,1]$. Since the distribution
of $\langle\Sigma_n \rangle_{n-\lfloor\ve n\rfloor}$ under $P$ is $\mu
_n^\ve$, we have
\[
\E\otimes E \bigl[ f \bigl(\mu_{n}^ \varepsilon \bigl( \langle
\Sigma_{n}\rangle_{n-
\lfloor\varepsilon n \rfloor} \bigr) \bigr) \bigr] = \E \biggl[
\sum_{u\in\t^{*n}_{n-\lfloor\ve n\rfloor}} \mu_n^\ve(u) f
\bigl( \mu_n^\ve(u) \bigr) \biggr].
\]
By Proposition~\ref{conv-reduced-tree}, we know that $\P$ a.s. for $n$
sufficiently large,
\[
\sum_{u\in\t^{*n}_{n-\lfloor\ve n\rfloor}} \mu_n^\ve(u) f
\bigl(\mu_n^\ve(u) \bigr) =\sum
_{x=(v,1-\ve)\in\partial\Delta_\ve} \mu^\ve_n \bigl(w^{n,\ve}_v
\bigr) f \bigl(\mu ^\ve_n \bigl(w^{n,\ve}_v
\bigr) \bigr)
\]
and, by Proposition~\ref{convergence-harmonic}, the latter quantities converge as
$n\to\infty$ toward
\[
\sum_{x\in\partial\Delta_\ve} \mu^\ve(x) f \bigl(
\mu^\ve(x) \bigr) = E \bigl[f \bigl( \mu \bigl(\mathscr{B}_\bd(B_T,
2\varepsilon) \bigr) \bigr) \bigr].
\]
Our claim~\eqref{convharmocont} now follows.

By~\eqref{compatech1} and~\eqref{convharmocont}, we can
find $ n_{0}=n_{0}( \varepsilon) \geq\varepsilon^{-1}$ such that for
$n\geq n_{0}$ we have
\[
\mathbb{P} \otimes P \bigl( \bigl\llvert \log \mu_{n}^
\varepsilon
\bigl( \langle\Sigma_{n}\rangle_{n - \lfloor \varepsilon n
\rfloor} \bigr) - \beta\log
\varepsilon \bigr\rrvert > \eta \llvert {\log \varepsilon}\rrvert \bigr) < \eta.
\]
It follows that
\begin{eqnarray}
\label{compatech2} &&\mathbb{E} \otimes E \bigl[ \bigl\llvert \log
\mu_{n}^ \varepsilon \bigl( \langle\Sigma_{n}
\rangle_{n - \lfloor \varepsilon n \rfloor} \bigr) - \beta\log\varepsilon \bigr\rrvert ^2
\bigr]
\nonumber
\\
&&\qquad\leq \eta^2 \llvert \log\varepsilon\rrvert ^2 +
\eta^{1/2} \mathbb{E} \otimes E \bigl[ \bigl\llvert \log
\mu_{n}^ \varepsilon \bigl( \langle \Sigma_{n}
\rangle_{n - \lfloor \varepsilon n \rfloor} \bigr) - \beta \log\varepsilon \bigr\rrvert
^4 \bigr]^{1/2}
\\
&&\qquad\leq \bigl(\eta^2+ 2\eta^{1/2}\beta^2
\bigr)\llvert \log\ve\rrvert ^2 + 2\eta^{1/2} \mathbb{E}
\otimes E \bigl[ \bigl\llvert \log \mu_{n}^ \varepsilon \bigl(
\langle\Sigma_{n}\rangle_{n - \lfloor \varepsilon n \rfloor} \bigr) \bigr\rrvert
^4 \bigr]^{1/2}.\nonumber
\end{eqnarray}
Let us bound the last term in the right-hand side. It is elementary to
verify that the function
$g(r)=(r \wedge e^{-4}) \llvert  \log(r \wedge e^{-4})\rrvert  ^4$ is nondecreasing
and concave over $[0,1]$. It follows that
\begin{eqnarray*}
E \bigl[ \bigl\llvert \log \mu_{n}^ \varepsilon \bigl( \langle
\Sigma_{n}\rangle _{n - \lfloor \varepsilon n \rfloor} \bigr) \bigr\rrvert ^4
\bigr]& =& \sum_{u\in\t^{*n}_{n-\lfloor\ve n\rfloor}} \mu_n^\ve(u)
\bigl\llvert \log\mu _n^\ve(u) \bigr\rrvert
^4
\\
&\leq&\sum_{u\in\t^{*n}_{n-\lfloor\ve n\rfloor}} \bigl(\mu_n^\ve(u)
\wedge e^{-4} \bigr) \bigl\llvert \log \bigl( \mu_n^\ve(u)
\wedge e^{-4} \bigr) \bigr\rrvert ^4 + 4^4
\\
&=& \sum_{u\in\t^{*n}_{n-\lfloor\ve n\rfloor}} g \bigl(\mu_n^\ve(u)
\bigr)+ 4^4
\\
&\leq&\#\t^{*n}_{n-\lfloor\ve n\rfloor}\times g \bigl( \bigl( \#\t
^{*n}_{n-\lfloor\ve n\rfloor} \bigr)^{-1} \bigr) + 4^4
\\
&\leq& \bigl\llvert \log \#\t^{*n}_{n-\lfloor\ve n\rfloor} \bigr\rrvert
^4 + 2\times4^4.
\end{eqnarray*}
We now use Lemma~\ref{lem:estimate-reduced-tree} to get
\begin{eqnarray*}
\E\otimes E \bigl[ \bigl\llvert \log \mu_{n}^ \varepsilon \bigl(
\langle\Sigma _{n}\rangle_{n - \lfloor \varepsilon n \rfloor} \bigr) \bigr\rrvert
^4 \bigr] &\leq&2\times4^4 + \E \bigl[ \bigl\llvert
\log \#\t^{*n}_{n-\lfloor\ve n\rfloor
} \bigr\rrvert ^4 \bigr]
\\
&\leq& 2\times4^4 + C^4 \biggl(\log\frac{n}{\lfloor\ve n\rfloor}
\biggr)^4.
\end{eqnarray*}
By combining the last estimate with (\ref{compatech2}), we get that,
for every $n\geq n_0(\ve)$,
\begin{eqnarray*}
&&\mathbb{E} \otimes E \bigl[ \bigl\llvert \log \mu_{n}^
\varepsilon
\bigl( \langle\Sigma_{n}\rangle_{n - \lfloor \varepsilon n
\rfloor} \bigr) - \beta\log
\varepsilon \bigr\rrvert ^2 \bigr]
\\
&&\qquad\leq \bigl(\eta^2+ 2\eta^{1/2}\beta^2
\bigr)\llvert \log\ve\rrvert ^2 + 2\eta^{1/2} \bigl(
2^{9/2} + C^2 \llvert \log\ve\rrvert ^2 \bigr).
\end{eqnarray*}
The statement of the corollary follows since $\eta$ was arbitrary.
\end{pf}

\subsection{Proof of the main result}
\label{sec:proofthmaindiscrete}
We need a few preliminary lemmas before we can proceed to the proof of
Theorem~\ref{thm:maindiscrete}.
\subsubsection{Preliminary lemmas}
\label{prelilemma}
Our first lemma is a discrete version of Lem\-ma \ref{flow-property}.
This result is well known and corresponds
to the ``flow rule'' for harmonic measure in \cite{LPP95}. We provide a
detailed statement and
a brief proof because this result plays a key role in what follows.

We consider a plane tree $\tau\in\mathscr{T}_n$, and we write $Z^{(\tau
)}=(Z^{(\tau)}_k)_{k\geq0}$ for simple
random walk on $\tau$ starting from $\varnothing$ (we may assume that
this process is defined under the probability measure $P$). We
set
\[
H_n^{(\tau)}=\inf \bigl\{k\geq0: \bigl\llvert
Z^{(\tau)}_k \bigr\rrvert =n \bigr\},
\]
and $\Sigma_{n}^{(\tau)}= Z^{(\tau)}_{H^{(\tau)}_n}$. We let $\mu^{(\tau
)}_n$ be the distribution of $\Sigma_{n}^{(\tau)}$. We view $\mu^{(\tau)}_n$
as a measure on $\tau$, which is supported on $\tau_n$.

For $0\leq p\leq n$,
we set
\[
L^{(\tau)}_p=\sup \bigl\{k\leq H_n^{(\tau)}:
\bigl\llvert Z^{(\tau)}_k \bigr\rrvert =p \bigr\}.
\]
Clearly, $\Sigma_{n}^{(\tau)} \in\widetilde\tau[Z^{(\tau)}_{L^{(\tau
)}_p}]$ and,
therefore, $Z^{(\tau)}_{L^{(\tau)}_p}=\langle\Sigma_{n}^{(\tau)}\rangle_p$.

%
\begin{lemma}
\label{conditioning-subtree}
Let $p\in\{0,1,\ldots,n-1\}$ and $z\in\tau_p$. Then, conditionally
on $\langle\Sigma_{n}^{(\tau)}\rangle_p=z$, the process
\[
\bigl( Z^{(\tau)}_{(L^{(\tau)}_p+k)\wedge H^{(\tau)}_n} \bigr)_{k\geq0}
\]
is distributed as simple random walk on $\widetilde\tau[z]$
starting from $z$ and conditioned to hit $\widetilde\tau[z]\cap\tau_n$
before returning to $z$, and stopped at this hitting time.
Consequently, for every integer $q\in\{0,1,\ldots,n-p\}$, the
conditional distribution of
\[
\frac{\mu_n^{(\tau)}(B_\tau(\Sigma_{n}^{(\tau)},q))}{\mu_n^{(\tau
)}(B_\tau(\Sigma_{n}^{(\tau)},n-p))}
\]
knowing that $\langle\Sigma_{n}^{(\tau)}\rangle_p=z$ is equal to the
distribution of
\[
\mu^{(\tau[z])}_{n-p} \bigl(B_{\tau[z]} \bigl(
\Sigma_{n-p}^{(\tau[z])},q \bigr) \bigr).
\]
\end{lemma}

\begin{pf}
The first assertion is easy from the fact that the successive
(nontrivial) excursions of $Z^{(\tau)}$
in the subtree $\widetilde\tau[z]$ are independent (and independent of
the behavior of
$Z^{(\tau)}$ outside $\widetilde\tau[z]$) and have the same
distribution as the excursion
of random walk in $\widetilde\tau[z]$ away from $z$. We leave the
details to the reader.

Let us explain why the second assertion of the lemma follows from the
first one.
Clearly, the distribution of the
hitting point of $\widetilde\tau[z]\cap\tau_n$ by simple random walk on
$\widetilde\tau[z]$
starting from $z$ and conditioned to hit $\widetilde\tau[z]\cap\tau_n$
before returning to $z$ is the same as the distribution of the
hitting point of $\widetilde\tau[z]\cap\tau_n$ by simple random walk on
$\widetilde\tau[z]$
starting from $z$.
Let $\mu^{(\tau),z}_n$ be the conditional distribution of
$\Sigma_{n}^{(\tau)}$ knowing that $\langle\Sigma^{(\tau)}_n\rangle_p=z$.
We get from the first assertion of the lemma that
$\mu^{(\tau),z}_n$ is equal to the hitting distribution of $\widetilde
\tau[z] \cap\tau_{n}$ for simple random walk on $\widetilde\tau[z]$
started from $z$ (note that we are here interested in the subgraph
$\widetilde\tau[z]$ of $\tau$ and not in the ``relabelled'' tree $ \tau[z]$).
It also follows that, for every integer $q\in\{0,1,\ldots,n-p\}$,
the conditional distribution of
\[
\mu^{(\tau),z}_n \bigl(B_\tau \bigl(
\Sigma^{(\tau)}_n,q \bigr) \bigr)
\]
knowing that $\langle\Sigma^{(\tau)}_n\rangle_p=z$ coincides
with the distribution of
\[
\mu^{(\tau[z])}_{n-p} \bigl(B_{\tau[z]} \bigl(
\Sigma_{n-p}^{(\tau[z])},q \bigr) \bigr).
\]
Now notice that, on the event $\{\langle\Sigma^{(\tau)}_n\rangle_p=z\}$,
$\mu^{(\tau),z}_n(B_\tau(\Sigma^{(\tau)}_n,q))$ is equal to
\[
\frac{\mu_n^{(\tau)}(B_\tau(\Sigma^{(\tau)}_n,q))}{\mu_n^{(\tau)}(B_\tau
(\Sigma^{(\tau)}_n,n-p))}.
\]
This gives the second assertion of the lemma.
\end{pf}

Let us come back to the (random) reduced tree $\t^{*n}$.
If $1\leq i\leq n$,\break $\widetilde\t^{*n}[\langle\Sigma_n\rangle_{n-i}]$
is the subtree of $\t^{*n}$ above generation $n-i$ that is ``selected'' by
harmonic measure, and
$\t^{*n}[\langle\Sigma_n\rangle_{n-i}]$ is the tree obtained by
relabelling the vertices of
$\widetilde\t^{*n}[\langle\Sigma_n\rangle_{n-i}]$ as explained above.
It is not true that the
distribution of $\t^{*n}[\langle\Sigma_n\rangle_{n-i}]$ under $\P
\otimes P$ coincides with the distribution of
$\t^{*i}$ under $\P$, because harmonic
measure induces a distributional bias. Still the next lemma gives a
useful bound for the
distribution of $\t^{*n}[\langle\Sigma_n\rangle_{n-i}]$ in terms of
that of $\t^{*i}$.
We recall the notation $ \mathcal{C}_{i}( \tau)$ from Section~\ref{sec:convergenceconductance}.

%
\begin{lemma}
\label{lem:tree-selected}
For every $i\in\{1,\ldots,n-1\}$ and every nonnegative function $F$ on
$\mathscr{T}$,
\[
\E\otimes E \bigl[ F \bigl(\t^{*n} \bigl[\langle\Sigma_n
\rangle_{n-i} \bigr] \bigr) \bigr] \leq(i+1) \E \bigl[
\mathcal{C}_i \bigl(\t^{*i} \bigr) F \bigl(
\t^{*i} \bigr) \bigr].
\]
\end{lemma}

\begin{pf} Fix $i\in\{1,\ldots,n-1\}$ in this proof. Recall our
notation $R_{i}(\t^{*n})$
for the
tree $\t^{*n}$ truncated at level $n-i$.
From the branching property of Galton--Watson
trees, one easily verifies the following fact: under $\P$,
conditionally on $R_{i}(\t^{*n})$,
the (relabelled) subtrees $\t^{*n}[v]$, $v\in\t^{*n}_{n-i}$ are
independent and
distributed as $\t^{*i}$ (to make this statement precise we can order
the subtrees
according to the lexicographical order on $\t^{*n}_{n-i}$).

Consider the stopping times of the random walk $Z^n$ which are defined
inductively
as follows,
\begin{eqnarray*}
U^n_0&=&\inf \bigl\{k\geq0: \bigl\llvert
Z^n_k \bigr\rrvert =n-i \bigr\},
\\
V^n_0&=&\inf \bigl\{k\geq U^n_0:
\bigl\llvert Z^n_k \bigr\rrvert =n-i-1 \bigr\},
\end{eqnarray*}
and, for every $j\geq0$,
\begin{eqnarray*}
U^n_{j+1}&=&\inf \bigl\{k\geq V^n_j:
\bigl\llvert Z^n_k \bigr\rrvert =n-i \bigr\},
\\
V^n_{j+1}&=&\inf \bigl\{k\geq U^n_{j+1}:
\bigl\llvert Z^n_k \bigr\rrvert =n-i-1 \bigr\}.
\end{eqnarray*}
Set $W^n_j= Z^n_{U^n_j}$ for every $j\geq0$. Then, under the
probability measure $P$,
$(W^n_j)_{j\geq0}$ is a Markov chain on $\t^{*n}_{n-i}$, whose initial
distribution and transition kernel only
depend on $R_{i}(\t^{*n})$.

Now observe that
\[
\langle\Sigma_n\rangle_{n-i}= W^n_{j_0},
\]
where $j_0$ is the first index $j$ such that
\begin{equation}
\label{max-excursion} \sup_{U^n_j\leq k\leq V^n_j} \bigl\llvert Z^n_k
\bigr\rrvert = n.
\end{equation}
If $j\geq0$ is fixed, then, conditionally on the Markov chain
$W^n$, the probability that~(\ref{max-excursion}) holds
is $\mathcal{C}_i(\t^{*n}[W^n_j])$.

Thanks to these observations, we have
\begin{eqnarray*}
&&E \bigl[ F \bigl(\t^{*n} \bigl[\langle\Sigma_n
\rangle_{n-i} \bigr] \bigr) \bigr]
\\
&&\qquad=\sum_{j=0}^\infty E \Biggl[ F
\bigl( \t^{*n} \bigl[W^n_j \bigr] \bigr)
\mathcal{C}_i \bigl(\t ^{*n} \bigl[W^n_j
\bigr] \bigr) \prod_{\ell=0}^{j-1} \bigl(1-
\mathcal{C}_i \bigl(\t^{*n} \bigl[W^n_\ell
\bigr] \bigr) \bigr) \Biggr].
\end{eqnarray*}
We then use the simple bound $\mathcal{C}_i(\t)\geq\frac{1}{i+1}$, which
holds for any tree $\t$ with height greater than or equal to $i$. It
follows that
\[
E \bigl[ F \bigl(\t^{*n} \bigl[\langle\Sigma_n
\rangle_{n-i} \bigr] \bigr) \bigr] \leq\sum
_{j=0}^\infty \biggl(1-\frac{1}{i+1}
\biggr)^j E \bigl[ F \bigl(\t ^{*n} \bigl[W^n_j
\bigr] \bigr) \mathcal{C}_i \bigl(\t^{*n}
\bigl[W^n_j \bigr] \bigr) \bigr].
\]
For\vspace*{2pt} every $u\in\mathcal{U}$ with $\llvert  u\rrvert  =n-i$, let $\pi^n_j(u)=
P(W^n_j=u)$, and recall that
$\pi^n_j(u)$ only depends on the truncated tree $R_{i}(\t^{*n})$. Then,
for every $j\geq0$,
\begin{eqnarray*}
\E\otimes E \bigl[ F \bigl(\t^{*n} \bigl[W^n_j
\bigr] \bigr) \mathcal{C}_i \bigl(\t^{*n}
\bigl[W^n_j \bigr] \bigr) \bigr] &=&\E \biggl[\sum
_{u\in\t^{*n}_{n-i}} \pi^n_j(u) F \bigl(
\t^{*n}[u] \bigr) \mathcal {C}_i \bigl(\t^{*n}[u]
\bigr) \biggr]
\\
&=& \E \bigl[F \bigl(\t^{*i} \bigr) \mathcal{C}_i \bigl(
\t^{*i} \bigr) \bigr],
\end{eqnarray*}
by the observation of the beginning of the proof. We conclude that
\begin{eqnarray*}
E \bigl[ F \bigl(\t^{*n} \bigl[\langle\Sigma_n
\rangle_{n-i} \bigr] \bigr) \bigr] &\leq& \sum
_{j=0}^\infty \biggl(1-\frac{1}{i+1}
\biggr)^j \E \bigl[F \bigl(\t^{*i} \bigr)
\mathcal{C}_i \bigl(\t^{*i} \bigr) \bigr]
\\
&=& (i+1) \E \bigl[F \bigl(\t^{*i} \bigr) \mathcal{C}_i
\bigl( \t^{*i} \bigr) \bigr],
\end{eqnarray*}
as desired.
\end{pf}

Our last lemma gives an estimate for the conductance $\mathcal{C}_i(\t^{*i})$.

%
\begin{lemma}
\label{moment-conductance}
There exists a constant $K\geq1$ such that, for every integer $n\geq1$,
\[
\E \bigl[\mathcal{C}_n \bigl(\t^{*n} \bigr)^2
\bigr] \leq\frac{K}{(n+1)^2}.
\]
\end{lemma}

\begin{pf}
Obviously, we can assume that $n\geq2$, and we set $j=\lfloor
n/2\rfloor\geq1$. An immediate
application of the Nash--Williams inequality (\cite{LP10}, Chapter~2) gives
\[
\mathcal{C}_n \bigl(\t^{*n} \bigr) \leq\frac{ \# \t^{*n}_j}{j}
\]
(just consider the cutsets obtained by looking for every integer $\ell
\in\{1,\ldots,j\}$
at the collection of edges of $\t^{*n}$ between generation $\ell-1$ and
generation $\ell$).
Then
\begin{eqnarray*}
\E \bigl[ \bigl(\# \t^{*n}_j \bigr)^2
\bigr]&=& \E \bigl[ \bigl(\# \bigl\{v\in\t^{(0)}_j: h \bigl(
\t^{(0)}[v] \bigr)\geq n-j \bigr\} \bigr)^2\mid h \bigl(
\t^{(0)} \bigr)\geq n \bigr]
\\
&=& q_n^{-1} \E \bigl[ \bigl(\# \bigl\{v\in
\t^{(0)}_j: h \bigl(\t^{(0)}[v] \bigr)\geq n-j
\bigr\} \bigr)^2 \bigr].
\end{eqnarray*}
As we already observed in the proof of Lemma~\ref{lem:estimate-reduced-tree},
the conditional distribution
of $\#\{v\in\t^{(0)}_{j}: h(\t^{(0)}[v])\geq n-j\}$ knowing that
$\#\t^{(0)}_{j}=k$ is the binomial distribution $\mathcal{B}(k,q_{n-j})$.
It follows that
\begin{eqnarray*}
&&\E \bigl[ \bigl(\# \bigl\{v\in\t^{(0)}_j: h \bigl(
\t^{(0)}[v] \bigr)\geq n-j \bigr\} \bigr)^2 \bigr]
\\
&&\qquad=q_{n-j}^2 \E \bigl[ \bigl(\#\t^{(0)}_j
\bigr)^2 \bigr] + \bigl(q_{n-j}-q_{n-j}^2
\bigr) \E \bigl[\#\t ^{(0)}_j \bigr]
\\
&&\qquad= q_{n-j}^2 \sigma^2 j +
q_{n-j}.
\end{eqnarray*}
We conclude that
\[
\E \bigl[\mathcal{C}_n \bigl(\t^{*n} \bigr)^2
\bigr] \leq \bigl(j^2q_n \bigr)^{-1}
\bigl(q_{n-j}^2 \sigma ^2 j + q_{n-j}
\bigr),
\]
and the statement of the lemma follows from (\ref{survivalpro}).
\end{pf}

\subsubsection{Proof of Theorem~\texorpdfstring{\protect\ref{thm:maindiscrete}}{1}}

We will prove that
\begin{equation}
\label{eq:goal} \mathbb{E} \otimes E \bigl[ \bigl\llvert \log\mu_{n}(
\Sigma_{n}) + \beta \log n \bigr\rrvert \bigr] = o(\log n)\qquad
\mbox{as }n \to\infty.
\end{equation}
Theorem~\ref{thm:maindiscrete} follows, since~\eqref{eq:goal} and the
Markov inequality give, for any $\delta>0$,
\[
\mathbb{P} \otimes P \bigl( \bigl\llvert \log\mu_{n}(
\Sigma_{n}) + \beta\log n \bigr\rrvert \geq\delta\log n \bigr)
\mathop{ \longrightarrow}\limits_{n\to\infty} 0,
\]
and, therefore,
\[
\mathbb{E} \bigl[ P \bigl( \mu_{n}( \Sigma_{n}) \leq
n^{ - \beta- \delta
}\mbox{ or } \mu_{n}( \Sigma_{n})
\geq n^{ - \beta+ \delta} \bigr) \bigr] \mathop{\longrightarrow}\limits_{n\to\infty}
0. %
\]
Since by definition $\mu_n$ is the distribution of $\Sigma_n$ under
$P$, the last convergence is equivalent to the first
assertion of Theorem~\ref{thm:maindiscrete}.

Fix $\xi>0$ and let $ \varepsilon>0$ and $n_{0} \geq0$ be such that
the conclusion of
Corollary~\ref{cor:comparaison} holds for every $n\geq n_0$. Without loss of
generality, we may and will assume that $\ve=1/N$, for some
integer $N\geq4$, which is fixed throughout the proof.
We also fix a constant $\alpha>0$, such that $\alpha\log N <1/2$.

Let $n> N$ be sufficiently large so that $N^{\lfloor\alpha\log
n\rfloor}\geq n_0$.
We then let $\ell\geq1$ be the unique integer
such that
\[
N^\ell< n \leq N^{\ell+1}.
\]
Notice that
\begin{equation}
\label{trivial1} \frac{\log n}{\log N}-1\leq\ell\leq\frac{\log n}{\log N}.
\end{equation}
Our starting point is the equality
\begin{eqnarray}
\label{decomp-harmo} \log\mu_n(\Sigma_n)&=& \log
\frac{\mu_n(\Sigma_n)}{\mu_n(B(\Sigma
_n,N))}
\nonumber\\[-8pt]\\[-8pt]\nonumber
&&{} +\sum_{j=2}^\ell\log
\frac{\mu_n(B(\Sigma_n,N^{j-1}))}{\mu_n(B(\Sigma_n,N^j))} + \log\mu_n \bigl(B \bigl(\Sigma_n,
N^\ell \bigr) \bigr).
\nonumber
\end{eqnarray}
To simplify notation, we set
\begin{eqnarray*}
A_{1}^n &:=& \log\frac{\mu_n(\Sigma_n)}{\mu_n(B(\Sigma_n,N))} +\beta \log N,
\\
A^n_j&:=& \log\frac{\mu_n(B(\Sigma_n,N^{j-1}))}{\mu_n(B(\Sigma_n,N^j))} + \beta\log N \qquad
\mbox{for every }j\in\{2,\ldots,\ell\},
\\
A_{\ell
+1}^n &:=& \log\mu_n \bigl(B \bigl(
\Sigma_n, N^\ell \bigr) \bigr) + \beta\log
\bigl(n/N^\ell \bigr).
\end{eqnarray*}
From~\eqref{decomp-harmo}, we see that
\begin{equation}
\label{pfdiscrete1}
\qquad\mathbb{E} \otimes E \bigl[ \bigl\llvert \log\mu_{n}(
\Sigma_{n}) + \beta\log n \bigr\rrvert \bigr] = \mathbb{E} \otimes E
\Biggl[ \Biggl\llvert \sum_{j=1}^{\ell+1}
A_{j}^n \Biggr\rrvert \Biggr]\leq \sum
_{i=1}^{\ell+1} \mathbb{E} \otimes E \bigl[ \bigl\llvert
A^n_j \bigr\rrvert \bigr].
\end{equation}
We will now bound the different terms in the sum of the right-hand side.

\begin{pf*}{First step: A priori bounds}
We verify that, for $j\in\{
1,\ldots, \ell+1\}$, we have
\begin{equation}
\label{aprioribound} \E\otimes E \bigl[ \bigl\llvert A^n_j
\bigr\rrvert \bigr] \leq(C\sqrt{K}+ \beta) \log N,
\end{equation}
where $C$ is the constant in Lemma~\ref{lem:estimate-reduced-tree}, and
$K$ is the
constant in Lemma~\ref{moment-conductance}. Suppose first that $2\leq j \leq
\ell$.
Applying the second assertion of Lemma~\ref{conditioning-subtree} (with
$p=n-N^j$ and $q=N^{j-1}$) to the tree
$\t^{*n}$, we obtain that, for every
$z\in\t^{*n}_{n-N^j}$, the conditional distribution of $A^n_j$ under $P$,
knowing that $\langle\Sigma_n\rangle_{n-N^j}=z$, is the same
as the distribution of
\[
\log\mu_{N^j}^{(\t^{*n}[z])} \bigl(B \bigl(\Sigma_{N^j}^{(\t^{*n}[z])},N^{j-1}
\bigr) \bigr) + \beta\log N.
\]
Recalling that $\mu_{N^j}^{(\t^{*n}[z])}$ is the distribution of $\Sigma
_{N^j}^{(\t^{*n}[z])}$ under $P$,
we get
\begin{eqnarray}
\label{aprioritech1}
&& E \bigl[ \bigl\llvert A^n_j \bigr\rrvert
\mid\langle\Sigma_n\rangle_{n-N^j}=z \bigr]\nonumber
\\
&&\qquad   \leq E \bigl[
\bigl\llvert \log\mu_{N^j}^{(\t^{*n}[z])} \bigl(B \bigl(\Sigma
_{N^j}^{( \t^{*n}[z])},N^{j-1} \bigr) \bigr) \bigr\rrvert
\bigr] + \beta\log N
\\
&&\qquad  = G_j \bigl(\t^{*n}[z] \bigr) + \beta\log N,\nonumber
\end{eqnarray}
where for any tree $\tau\in\mathscr{T}_{N^j}$,
\begin{eqnarray*}
G_j(\tau)&=&\int\mu^{(\tau)}_{N^j}(\D y) \bigl
\llvert \log\mu^{(\tau
)}_{N^j} \bigl(B_\tau \bigl(y,
N^{j-1} \bigr) \bigr) \bigr\rrvert
\\
&=& \sum_{z\in\tau_{N^j-N^{j-1}}} \mu^{(\tau)}_{N^j}
\bigl(\widetilde\tau[z] \bigr) \bigl\llvert \log\mu^{(\tau
)}_{N^j}
\bigl(\widetilde\tau[z] \bigr) \bigr\rrvert.
\end{eqnarray*}
In the latter form, $G_j(\tau)$ is just the entropy of the probability
measure that assigns mass $\mu^{(\tau)}_{N^j}(\widetilde\tau[z])$
to every point $z\in\tau_{N^j-N^{j-1}}$. By a standard bound for the
entropy of probability measures on
finite sets, we have $G_j(\tau)\leq\log\# \tau_{N^j-N^{j-1}}$ for any
tree $\tau\in\mathscr{T}_{N^j}$.
Recalling~\eqref{aprioritech1}, we get
\begin{eqnarray*}
\E \otimes E \bigl[ \bigl\llvert A^n_j \bigr\rrvert
\bigr] &\leq&\E\otimes E \bigl[\log\#\t _{N^j-N^{j-1}}^{*n} \bigl[
\langle \Sigma_n\rangle_{n-N^j} \bigr] \bigr]+ \beta\log N
\\
&\leq& \bigl(N^j +1 \bigr) \E \bigl[\mathcal{C}_{N^j}
\bigl( \t^{*N^j} \bigr) \log\#\t ^{*N^j}_{N^j-N^{j-1}} \bigr] +
\beta \log N
\\
&\leq& \bigl(N^j+1 \bigr) \E \bigl[ \bigl(\mathcal{C}_{N^j}
\bigl(\t^{*N^j} \bigr) \bigr)^2 \bigr]^{1/2} \E
\bigl[ \bigl(\log\#\t^{*N^j}_{N^j-N^{j-1}} \bigr)^2
\bigr]^{1/2} + \beta\log N
\\
&\leq&\sqrt{K} \E \bigl[ \bigl(\log\#\t^{*N^j}_{N^j-N^{j-1}}
\bigr)^2 \bigr]^{1/2} + \beta\log N,
\end{eqnarray*}
using successively Lemma~\ref{lem:tree-selected}, the Cauchy--Schwarz
inequality and Lemma~\ref{moment-conductance}.
Finally, Lemma~\ref{lem:estimate-reduced-tree} gives
\[
\E \bigl[ \bigl(\log\#\t^{*N^j}_{N^j-N^{j-1}} \bigr)^2
\bigr]^{1/2}\leq C \log N,
\]
and this completes the proof of~\eqref{aprioribound} when $2\leq j \leq
\ell$.

The cases $j=1$ and $j=\ell+1$ are treated on a similar manner. For
$j=\ell+1$, we observe that
the same entropy bound gives
\[
E \bigl[ \bigl\llvert \log\mu_n \bigl(B \bigl(\Sigma_n,N^\ell
\bigr) \bigr) \bigr\rrvert \bigr] =\sum_{y\in\t^{*n}_{n-N^\ell}}
\mu_n \bigl(\widetilde\t^{*n}[y] \bigr) \bigl\llvert \log\mu
_n \bigl(\widetilde\t^{*n}[y] \bigr) \bigr\rrvert \leq\log
\# \t^{*n}_{n-N^\ell}.
\]
It follows that
\[
\E\otimes E \bigl[ \bigl\llvert \log\mu_n \bigl(B \bigl(
\Sigma_n,N^\ell \bigr) \bigr) \bigr\rrvert \bigr]\leq\E
\bigl[ \log\#\t ^{*n}_{n-N^\ell} \bigr] \leq C \log N,
\]
by Lemma~\ref{lem:estimate-reduced-tree} and using the fact that $N^\ell
< n\leq N^{\ell+1}$.

Finally, for the case $j=1$, we use
exactly the same argument as in the case $2\leq j\leq\ell$, to get
\[
E \biggl[ \biggl\llvert \log\frac{\mu_n(\Sigma_n)}{\mu_n(B(\Sigma_n,N))} \biggr\rrvert \biggr] \leq E
\bigl[\log\#\t^{*n}_{n} \bigl[\langle\Sigma_n
\rangle_{n-N} \bigr] \bigr],
\]
and we obtain similarly, using Lemmas~\ref{lem:tree-selected}, \ref
{moment-conductance} and \ref{lem:estimate-reduced-tree},
\begin{eqnarray*}
\E\otimes E \bigl[\log\#\t^{*n}_{n} \bigl[\langle
\Sigma_n\rangle_{n-N} \bigr] \bigr] &\leq&(N+1) \E \bigl[
\mathcal{C}_N \bigl(\t^{*N} \bigr) \log\#
\t^{*N}_N \bigr]
\\
&\leq&\sqrt{K} \E \bigl[ \bigl(\log\#\t^{*N}_N
\bigr)^2 \bigr]^{1/2}
\\
&\leq& C\sqrt{K}\log N.
\end{eqnarray*}
This completes the proof of~\eqref{aprioribound}.\vadjust{\goodbreak}\noqed
\end{pf*}

\begin{pf*}{Second step: Refined bounds}
We will get a better bound than
\eqref{aprioribound}
for certain values of $j$. Precisely we prove that, if $\lfloor\alpha
\log n\rfloor\leq j \leq\ell$, we have
\begin{equation}
\label{refinedbound} \E\otimes E \bigl[ \bigl\llvert A^n_j
\bigr\rrvert \bigr] \leq\sqrt{\xi K} \log N.
\end{equation}
Let us fix $j\in\{\lfloor\alpha\log n\rfloor,\ldots,\ell\}$.
Recall that we have then $N^j\geq n_0$.
From~\eqref{aprioritech1}, we have
\begin{equation}
\label{estimtech1} E \bigl[ \bigl\llvert A^n_j \bigr\rrvert
\bigr] = E \bigl[F_j \bigl(\t^{*n} \bigl[\langle
\Sigma_n\rangle_{n-N^j} \bigr] \bigr) \bigr],
\end{equation}
where, if $\tau\in\mathscr{T}_{N^j}$,
\[
F_j(\tau)= \bigl\llvert \beta\log N - G_j(\tau) \bigr
\rrvert = \biggl\llvert \int\mu^{(\tau)}_{N^j}(\D y) \bigl(\log
\mu^{(\tau)}_{N^j} \bigl(B_\tau \bigl(y,
N^{j-1} \bigr) \bigr) + \beta\log N \bigr) \biggr\rrvert.
\]
Using Lemma~\ref{lem:tree-selected} as in the first step, we have
\[
\E\otimes E \bigl[ \bigl\llvert A^n_j \bigr\rrvert
\bigr]= \mathbb{E} \otimes E \bigl[F_j \bigl(\t^{*n} \bigl[
\langle \Sigma_n\rangle_{n-N^j} \bigr] \bigr) \bigr] \leq
\bigl(N^j+1 \bigr) \mathbb{E} \bigl[ \mathcal{C}_{N^j}
\bigl( \t^{*N^j} \bigr) F_{j} \bigl(\t^{*N^j} \bigr)
\bigr]. %
\]
We then apply the Cauchy--Schwarz inequality together with the bound of
Lemma~\ref{moment-conductance} to get
\begin{eqnarray*}
\mathbb{E} \otimes E \bigl[ \bigl\llvert A_{j}^n \bigr
\rrvert \bigr] &\leq&\sqrt{K} \mathbb{E} \bigl[F_{j} \bigl(
\t^{*N^j} \bigr)^2 \bigr]^{1/2}
\\
&=& \sqrt{K} \E \biggl[ \biggl(\int\mu_{N^j}(\D y) \bigl\llvert \log
\mu _{N^j} \bigl(B \bigl(y, N^{j-1} \bigr) \bigr) + \beta\log N
\bigr\rrvert \biggr)^2 \biggr]^{1/2}
\\
&\leq&\sqrt{K} \E \biggl[\int\mu_{N^j}(\D y) \bigl\llvert \log\mu
_{N^j} \bigl(B \bigl(y, N^{j-1} \bigr) \bigr) + \beta\log N
\bigr\rrvert ^2 \biggr]^{1/2}
\\
&=& \sqrt{K} \E\otimes E \bigl[ \bigl\llvert \log\mu_{N^j} \bigl(B
\bigl( \Sigma_{N^j}, N^{j-1} \bigr) \bigr) + \beta\log N \bigr
\rrvert ^2 \bigr] ^{1/2}
\\
&=& \sqrt{K} \cdot\mathbb{E} \otimes E \bigl[ \bigl\llvert \log \mu
_{N^j}^{1/N} \bigl( \langle\Sigma_{N^j}
\rangle_{N^j-N^{j-1}} \bigr) + \beta\log N \bigr\rrvert ^2
\bigr]^{1/2},
\end{eqnarray*}
where the last equality follows from the definition of the measures $\mu
^\ve_n$ at the beginning of
Section~\ref{sec:conv-harmonic}. Now recall that $1/N=\ve$ and note
that $N^j-N^{j-1}=N^j - \ve N^{j}$. Since
we have $N^j\geq n_0$, we can apply the bound of Corollary~\ref{cor:comparaison}
and we get that the right-hand side
of the preceding display is bounded above by
$ \sqrt{\xi K} \log N$, which completes the proof of~\eqref{refinedbound}.

By combining~\eqref{aprioribound} and~\eqref{refinedbound}, and using
\eqref{pfdiscrete1}, we arrive at the bound
\begin{eqnarray*}
\E\otimes E \bigl[ \bigl\llvert \log\mu_n(\Sigma_n) +
\beta\log n \bigr\rrvert \bigr] &\leq&\lfloor\alpha\log n\rfloor(C\sqrt{K} +\beta)
\log N + \ell \sqrt {\xi K} \log N
\\
&\leq& \bigl(\alpha(C\sqrt{K} +\beta)\log N + \sqrt{\xi K} \bigr) \log n,
\end{eqnarray*}
which holds for every sufficiently large $n$. Now note that $\xi>0$ can
be chosen arbitrarily small.
The choice of $\xi$ determines the choice of $N$, but afterward we can
also choose $\alpha$ arbitrarily small
given this choice. We thus see that our claim~\eqref{eq:goal} follows
from the last bound, and this completes the
proof of Theorem~\ref{thm:maindiscrete}.
\end{pf*}

\subsection{Proof of Corollary \texorpdfstring{\protect\ref{cor:planetree}}{2}}
\label{sec:absolutecontinuity}
In what follows, we always \textit{implicitly restrict our attention} to
integers $N\geq1$ such that
$\mathbb{P}(\#\t^{(0)}=N+1)>0$. For such values of $N$, $\mathbf{T}(N)$
is distributed as
$\t^{(0)}$ conditioned on the event $\{\#\t^{(0)}=N+1\}$.
We write $(C^{(N)}_t)_{0\leq t\leq2N}$ for the contour function of the tree
$\mathbf{T}(N)$ (see, e.g., \cite{probasur} or \cite{Pit06}, Figure~6.2,
where the
contour function is called the Harris walk of the tree). By
a famous theorem of Aldous \cite{Al3}, we have the convergence in distribution
\begin{equation}
\label{conv-contour} \biggl(\frac{\sigma}{2\sqrt{N}} C^{(N)}_{2Nt}, 0
\leq t\leq1 \biggr) \mathop{\la}_{N\to\infty}^{(d)} (
\mathbf{e}_t,0\leq t\leq1),
\end{equation}
where $(\mathbf{e}_t,0\leq t\leq1)$ stands for a Brownian excursion
with duration $1$.
Since $h(\mathbf{T}(N))$ is just the maximum of the contour function,
it follows that
\begin{equation}
\label{asymp-maxi} \frac{1}{\sqrt{N}} h \bigl(\mathbf{T}(N) \bigr) \mathop{
\la}_{N\to\infty}^{(d)} \frac{2}{\sigma} \max_{0\leq t\leq1}
\mathbf{e}_t.
\end{equation}
Consequently, for every $\eta>0$, we can choose a constant $A>0$ such
that for all sufficiently large $N$, the probability
$P(h(\mathbf{T}(N))> A\sqrt{N})$ is bounded above by~$\eta$. Thanks to
this remark, it is enough to
prove that the convergence of Corollary~\ref{cor:planetree} holds when
$n$ and $N$ tend to infinity in
such a way that $n\leq B\sqrt{N}$, for some fixed constant $B$. For
future reference, we note that~\eqref{asymp-maxi} implies that, for every
sufficiently large $N$ and every nonnegative integer $n$ such that
$n\leq B\sqrt{N}$,
\begin{equation}
\label{eq:reasonableballs} \mathbb{P} \bigl(h \bigl(\mathbf{T}(N) \bigr)\geq n \bigr) \geq
c,
\end{equation}
for some constant $c>0$.

If $\tau\in\mathscr{T}$ is a tree, we write $\tau_{\leq n}$ for the
tree that consists of all
vertices of $\tau$ at generation
less than or equal to $n$.

%
\begin{lemma}
\label{techni-unitree}
Let $\ve>0$. We can find $\delta\in(0,\frac{1}{2})$ such
that, for every sufficiently large $N$, and every nonnegative integer
$n$ with $n\leq B\sqrt{N}$, we have
\[
P \bigl( \# \mathbf{T}(N)_{\leq n} \leq(1-\delta)N \mid h \bigl(\mathbf
{T}(N) \bigr)\geq n \bigr) \geq1-\ve.
\]
\end{lemma}

\begin{pf} As a simple consequence of~\eqref{conv-contour}, we can find
$\eta>0$
sufficiently small so that, for every sufficiently large $N$ and for
every integer $n$ with $0\leq n\leq\eta\sqrt{N}$,
\[
\mathbb{P} \biggl( \bigl\{h \bigl(\mathbf{T}(N) \bigr)\geq n \bigr\}\cap \biggl
\{\# \mathbf{T}(N)_{\leq n} < \frac{N}{2} \biggr\} \biggr) > 1-\ve.
\]
So we may concentrate on values of $n$ such that
$\eta\sqrt{N} \leq n\leq B\sqrt{N}$.

We then observe that there exists $\delta\in(0,\frac{1}{2})$
such that, for every $a\in[\frac{1}{2}\sigma\eta,\frac{1}{2}\sigma B]$,
\begin{equation}
\label{unitree1} \mathbb{P} \biggl(\int_0^1 \D
t \mathbf{1}_{\{\mathbf{e}_t\geq a\}}\leq \delta \Big| \sup_{0\leq t\leq1}
\mathbf{e}_t \geq a \biggr) <\ve.
\end{equation}
This bound follows from standard properties of linear
Brownian motion. We omit the details.

We now claim that the result of the lemma holds with the preceding value
of $\delta$. To verify the claim, observe that
from the properties of the contour function,
\[
N+1 -\#\mathbf{T}(N)_{\leq n} = \frac{1}{2} \int
_0^{2N} \D t \mathbf{1}_{\{C^{(N)}_t > n\}}.
\]
It readily follows that
\begin{eqnarray*}
&&\P \bigl( \# \mathbf{T}(N)_{\leq n} >(1-\delta)N \mid h \bigl(\mathbf
{T}(N) \bigr)\geq n \bigr)
\\
&&\qquad = \P \biggl(\frac{1}{2} \int_0^{2N}
\D t \mathbf{1}_{\{C^{(N)}_t > n\}}< \delta N+1 \Big| \sup_{0\leq t\leq
2N}
C^{(N)}_t \geq n \biggr)
\\
&&\qquad = \P \biggl(\int_0^{1}\D t
\mathbf{1}_{\{(\sigma/2\sqrt{N}) C^{(N)}_{2Nt}>(\sigma/2\sqrt{N}) n\}
}<\delta+ \frac{1}{N} \Big| \sup
_{0\leq t\leq1} \frac{\sigma}{2\sqrt
{N}}C^{(N)}_{2N t}
\geq\frac{n}{\sqrt{2N}} \biggr).
\end{eqnarray*}
If the conclusion of the lemma does not hold, we can find
a sequence $N_k$ converging to $+\infty$, and, for every $k$,
an integer $n_k$ with $\eta\sqrt{N_k} \leq n_k\leq B\sqrt{N_k}$, such that
the probability in the last display, evaluated with $N=N_k$ and $n=n_k$
is bounded below by $\ve$. But then, by extracting a convergent
subsequence from the sequence $(n_k/\sqrt{N_k})$ and using the
convergence~\eqref{conv-contour},
we get a contradiction with~\eqref{unitree1}. This contradiction completes
the proof.
\end{pf}

As previously, we let $ \mathsf{T}^{(n)}$ stand for a Galton--Watson tree
with offspring distribution $\theta$, conditioned on nonextinction at
generation $n$.
Corollary \ref{cor:planetree} is a simple consequence of Theorem \ref
{thm:maindiscrete} and the following comparison lemma applied, for every
fixed $\delta>0$ and $\ve>0$, with
\[
A_{n}= \bigl\{ {\tau} \in\mathscr{T}_n:
\mu_n^{(\tau)} \bigl( \bigl\{v\in\tau_{n}:
n^{-\beta-\delta}\leq\mu_n^{( \tau)}(v) \leq n^{-\beta+\delta}
\bigr\} \bigr)\leq 1- \varepsilon \bigr\}.
\]

%
\begin{lemma} \label{lem:comptree}
For every $n\geq0$, let $A_n$ be a subset of $\mathscr{T}_n$.
Assume that $ \mathbb{P}( \t^{(n)}_{\leq n} \in A_{n}) \to0$ as $n \to
\infty$. Then we have
\[
\mathbb{P} \bigl( \mathbf{T}(N)_{\leq n}\in A_{n} \mid h
\bigl( \mathbf{T}(N) \bigr) \geq n \bigr)
\mathop{\mathop{\longrightarrow}_{n,N\to\infty}}_
{n \leq B\sqrt {N} } 0. %
\]
\end{lemma}

\begin{pf} Throughout the proof, we consider positive integers $n$ and
$N$ such that
$n\leq B \sqrt{N}$. Let $\tau\in\mathscr{T}_n$ and set $m=\#\tau-1$
($m$ is the number of edges of $\tau$) and $p=\# \tau_n$. From~\eqref
{survivalpro}, we see that there exists a constant $c_0>0$ such that,
for every $n$,
\begin{equation}
\label{abs1} \mathbb{P} \bigl(\t^{(n)}_{\leq n}=\tau \bigr)
\geq c_0 n \mathbb{P} \bigl(\t ^{(0)}_{\leq n}=\tau
\bigr).
\end{equation}

We then evaluate
\[
\mathbb{P} \bigl( \mathbf{T}(N)_{\leq n} =\tau \bigr)=
\frac{\mathbb{P}(\{\t
^{(0)}_{\leq n}=\tau\}\cap\{\#\t^{(0)}=N+1\})}{
\mathbb{P}(\#\t^{(0)}=N+1)}.
\]
Let $\widetilde\theta$ be the probability measure on $\Z$ defined by
$\widetilde\theta(k)=\theta(k+1)$ for every $k\geq-1$, and
let $Z$ be a random walk on $\Z$ with jump distribution $\widetilde
\theta$ started from $0$. A~standard result
(see, e.g., \cite{probasur}, Section~1) states that $\#\t^{(0)}$ has
the same distribution as the hitting time
of $-1$ by $Z$, and by Kemperman's formula (see, e.g., Pitman \cite{Pit06}, page~122), we get that $\mathbb{P}(\#\t^{(0)}=N+1)=(N+1)^{-1}
\mathbb{P}(Z_{N+1}=-1)$. From a classical local limit theorem, we
obtain the existence of a constant $c_1>0$ such that
\begin{equation}
\label{abs2} \mathbb{P} \bigl(\#\t^{(0)}=N+1 \bigr) \geq
c_1 N^{-3/2}
\end{equation}
[recall that we consider only values of $N$ such that $\mathbb{P}(\#\t
^{(0)}=N+1)>0$].
Then, using the branching property of Galton--Watson trees, we have, if
$m\leq N$,
\begin{equation}
\label{abs20} \qquad\mathbb{P} \bigl( \bigl\{\t^{(0)}_{\leq n}=\tau
\bigr\}\cap \bigl\{\#\t^{(0)}=N+1 \bigr\} \bigr)= \mathbb{P} \bigl(
\t^{(0)}_{\leq n}=\tau \bigr)\times F(p,N-m+p),
\end{equation}
where, for every integer $\ell\geq p$, $F(p,\ell)$ is the probability
that a forest of $p$ independent
Galton--Watson trees with offspring distribution $\theta$ has exactly
$\ell$ vertices. By the same
arguments as in the derivation of~\eqref{abs2},
\begin{equation}
\label{abs3} F(p,\ell)=\frac{p}{\ell} \mathbb{P}(Z_\ell=-p)
\leq c_2 p \ell^{-3/2},
\end{equation}
with some constant $c_2$.

Next, let $\delta\in(0,1)$ and suppose that $m\leq(1-\delta)N$ and
$p\leq\delta^{-1}n$, so that in particular $N-m+p\geq\delta N$. Under
these conditions~\eqref{abs2},~\eqref{abs20}
and~\eqref{abs3} give
\begin{equation}
\label{abs4} \mathbb{P} \bigl( \mathbf{T}(N)_{\leq n} =\tau \bigr)\leq
c_1^{-1}c_2 \delta ^{-5/2} n
\mathbb{P} \bigl(\t^{(0)}_{\leq n}=\tau \bigr).
\end{equation}

Let $G'_{N,n,\delta}$ be the set of all trees $\tau\in\mathscr{T}_n$
such that $\#\tau\leq(1- \delta) N$, let $G''_{n,\delta}$ be the set
of all trees $\tau\in\mathscr{T}_n$
such\vspace*{1pt} that $\#\tau_n\leq\delta^{-1}n$, and set $G_{N,n,\delta}
=G'_{N,n,\delta} \cap G''_{n,\delta}$.
Comparing~\eqref{abs1} and~\eqref{abs4}, we obtain that the density of
the law of $ \mathbf{T}(N)_{\leq n}$ with respect to that of $\t
^{(n)}_{\leq n}$ is bounded above,
on the set $G_{N,n,\delta}$, by a positive constant $ C_{\delta}$
independent of $n$ and $N$ (but depending on $\delta$). If $\ve>0$ is
given, we can use Lemma \ref{techni-unitree} to find $\delta>0$ such
that for every
sufficiently large $N$ and every integer $n$ with $1\leq n \leq B\sqrt
{N}$ we have
\[
\mathbb{P} \bigl( \mathbf{T}(N)_{\leq n}\in G'_{N,n,\delta}
\mid h \bigl( \mathbf{T}(N) \bigr) \geq n \bigr) \geq1-\frac{\varepsilon}{2}.
\]
On the other hand, Theorem 1.13 in Janson \cite{Jan} gives the
existence of a constant $K$ independent of $N$ such that, for
every integer $n\geq1$, $\E[\#\mathbf{T}(N)_n] \leq K n$. Choosing
$\delta$ smaller if necessary, and using~\eqref{eq:reasonableballs},
we see that we have also,
for every integer $n$ with $1\leq n \leq B\sqrt{N}$,
\[
\mathbb{P} \bigl( \mathbf{T}(N)_{\leq n}\in G''_{n,\delta}
\mid h \bigl( \mathbf{T}(N) \bigr) \geq n \bigr) \geq1-\frac{\varepsilon}{2}.
\]

Finally, if $A_{n}$ is a subset of $\mathscr{T}_n$, with $1\leq n \leq
B\sqrt{N}$, we have
\begin{eqnarray*}
&&\mathbb{P} \bigl( \mathbf{T}(N)_{\leq n} \in A_{n} \mid h
\bigl( \mathbf {T}(N) \bigr) \geq n \bigr)
\\
&&\qquad \leq\mathbb{P} \bigl( \mathbf{T}(N)_{\leq n} \notin
G_{N,n,\delta} \mid h \bigl( \mathbf{T}(N) \bigr) \geq n \bigr)
\\
&&\quad\qquad{}+ \mathbb{P} \bigl( \mathbf{T}(N)_{\leq n} \in
A_{n} \cap G_{N,n,\delta} \mid h \bigl( \mathbf{T}(N) \bigr) \geq
n \bigr)
\\
&&\qquad\leq \varepsilon+ \frac{C_{\delta}}{ \mathbb{P}( h( \mathbf
{T}(N)) \geq n)} \mathbb{P} \bigl(
\t^{(n)}_{\leq n}\in A_{n} \bigr).
\end{eqnarray*}
Letting $n,N \to\infty$ with the constraint $n \leq B\sqrt{N}$, and
using the assumption of the lemma together with~\eqref
{eq:reasonableballs}, we see that the last display eventually becomes
less than~$2 \varepsilon$. This proves the lemma.
\end{pf}

\section{Complements}
\label{sec:comple}
\renewcommand{\t}{\mathcal{T}}

\subsection{A different approach to the continuous results}
In this section, we briefly outline another approach to
Theorem~\ref{thm:main-harmonic}, which is based on a different shift
transformation
on the space $\T^*$. Informally, if $(\t,\mathbf{v})\in\T^*$, we let
$S(\t,\mathbf{v})$ be obtained by shifting $(\t,\mathbf{v})$ at the
first node
of $\t$. More precisely, if $\t$ corresponds to the collection
$(z_v)_{v\in\v}$,
and $\mathbf{v}=(v_1,v_2,\ldots)$, we set
\[
S(\t,\mathbf{v})= (\t_{(v_1)}, \widetilde{\mathbf{v}}),
\]
where $\widetilde{\mathbf{v}}=(v_2,v_3,\ldots)$ and, for $i=1$ or
$i=2$, $\t_{(i)}$ is the tree corresponding to the collection
$(z_{iv}-z_\varnothing)_{v\in\v}$, in agreement with the notation of
Section~\ref{sec:proofoftheoremmain-harmonic}.

%
\begin{proposition}
\label{invariantbis}
For every $r\geq1$, set
\[
\kappa(r) = \int\!\!\!\int\gamma(\D s) \gamma(\D t) \frac{r s}{r+s+t-1}.
\]
The finite measure $\kappa(\cc(\t))\cdot\Theta^*(\D\t \,\D\mathbf{v})$
is invariant under $S$.
\end{proposition}

\begin{pf} Let $F$ be a bounded measurable function on $\T^*$. We have
to prove that
\begin{equation}
\label{invarbistech1} \int F\circ S(\t,\mathbf{v}) \kappa \bigl(\cc(\t) \bigr)
\Theta^*(\D\t \,\D\mathbf{v})= \int F(\t,\mathbf{v}) \kappa \bigl(\cc(\t) \bigr)
\Theta^*(\D\t \,\D\mathbf{v}).
\end{equation}
Recall that $\Theta^*(\D\t \,\D\mathbf{v})= \Theta(\D\t)\nu_\t(\D\mathbf
{v})$ by construction. If
we fix $\t\in\T$, the distribution of the pair $(v_1,\widetilde{\mathbf
{v}})$ under $\nu_\t$
is given by
\[
\int\nu_\t(\D\mathbf{v}) \mathbf{1}_{\{v_1=i\}} g(\widetilde{
\mathbf{v}}) = \frac{\cc(\t_{(i)})}{\cc(\t_{(1)}) + \cc(\t_{(2)})}\int\nu_{\t
_{(i)}}(\D\mathbf{u}) g(
\mathbf{u}),
\]
where $i\in\{1,2\}$ and $g$ is any bounded measurable function
on $\{1,2\}^\N$. It follows that the left-hand side of~\eqref
{invarbistech1} may be
written as
\begin{equation}
\label{invarbistech2} \sum_{i=1}^2 \int F(
\t_{(i)}, \mathbf{u}) \kappa \bigl(\cc(\t) \bigr) \frac{\cc(\t
_{(i)})}{\cc(\t_{(1)}) + \cc(\t_{(2)})}
\Theta(\D\t) \nu_{\t_{(i)}}(\D\mathbf{u}).
\end{equation}
We then observe that under $\Theta(\D\t)$ the subtrees $\t_{(1)}$ and
$\t_{(2)}$ are independent and
distributed according to $\Theta$, and moreover we have
\[
\cc(\t)= \biggl( U + \frac{1-U}{\cc(\t_{(1)}) + \cc(\t_{(2)})} \biggr)^{-1},
\]
where $U$ is uniformly distributed over $[0,1]$ and independent of $(\t
_{(1)},\t_{(2)})$.
Using these observations, and a simple symmetry argument, we get that
the quantity
\eqref{invarbistech2} is also equal to
\begin{eqnarray*}
&&2 \int_0^1 \D x\int \Theta(\D\t) \Theta
\bigl(\D\t' \bigr) \nu_{\t}(\D\mathbf {u}) F(\t,
\mathbf{u})
\\
&&\quad{}\times\frac{\cc(\t)}{\cc(\t) + \cc(\t')} \kappa \biggl( \biggl( x +
\frac{1-x}{\cc(\t) + \cc(\t')} \biggr)^{-1} \biggr)
\\
&&\qquad= \int\Theta^*(\D\t \,\D\mathbf{u}) F(\t,\mathbf{u})
\\
&&\qquad\quad{}\times \biggl( 2 \int_0^1 \D x
\int\Theta \bigl(\D\t' \bigr) \frac{\cc
(\t)}{\cc(\t) + \cc(\t')} \kappa \biggl(
\biggl( x + \frac{1-x}{\cc(\t) +
\cc(\t')} \biggr)^{-1} \biggr) \biggr).
\end{eqnarray*}
Hence, the proof of~\eqref{invarbistech1} reduces to checking that, for
every $r\geq1$,
\begin{equation}
\label{invarbistech3} \kappa(r)=2 \int_0^1 \D x\int
\Theta \bigl(\D\t' \bigr) \frac{r}{r+ \cc(\t')} \kappa \biggl(
\biggl( x + \frac{1-x}{r + \cc(\t')} \biggr)^{-1} \biggr).
\end{equation}
To verify~\eqref{invarbistech3}, let $\cc_0,\cc_1,\cc_2$ be independent
and distributed
according to $\gamma$, and let $U$ be uniformly distributed over $[0,1]$
and independent of $(\cc_0,\cc_1,\cc_2)$ under the probability measure
$\P$. Note that by definition, for every $x\geq1$,
\[
\kappa(x)= \E \biggl[ \frac{x\cc_1}{x+\cc_1+ \cc_2 -1} \biggr].
\]
It follows that the right-hand side of~\eqref{invarbistech3} can be
written as
\begin{eqnarray*}
&&2 \E \biggl[ \frac{r}{r+\cc_0} \frac{\cc_1 (U+\vafrac{1-U}{r+\cc
_0} )^{-1}}{
\cc_1+\cc_2+  (U+\vafrac{1-U}{r+\cc_0} )^{-1} -1} \biggr]
\\
&&\qquad=2r \E \biggl[ \frac{\cc_1}{(\cc_1+\cc_2-1)(U(\cc_0+r) + 1- U)+\cc
_0 +r} \biggr]
\\
&&\qquad=r \E \biggl[ \frac{\cc_1+\cc_2}{(\cc_1+\cc_2-1)(U(\cc_0+r) + 1-
U)+\cc_0 +r} \biggr]
\\
&&\qquad= r \E \biggl[ \frac{\cc_1+\cc_2}{(\cc_1+\cc_2)(U(\cc_0+r-1) +
1) + (\cc_0 + r-1)(1-U)} \biggr]
\\
&&\qquad= r \E \biggl[\frac{1}{ (\cc_0+r-1) (U + \vafrac{1-U}{\cc_1 +
\cc_2} ) + 1} \biggr]
\\
&&\qquad= r \E \biggl[ \frac{\widetilde\cc}{r+\cc_0 + \widetilde\cc
-1} \biggr],
\end{eqnarray*}
where $\widetilde\cc= (U + \frac{1-U}{\cc_1 + \cc_2})^{-1}$. By~\eqref
{eq:rde}, $\widetilde\cc$ is distributed according to $\gamma$.
Since $\widetilde\cc$ is also independent of $\cc_0$, we immediately
see that the right-hand side of
the last display is equal to $\kappa(r)$, which completes the proof of
\eqref{invarbistech3}
and of the proposition.
\end{pf}

One can verify that the shift $S$ is ergodic with respect to the
invariant probability measure obtained by
normalizing $\kappa(\cc(\t))\cdot\Theta^*(\D\t \,\D\mathbf{v})$ (we omit
the proof). One then applies
the ergodic theorem to the two functionals defined as follows. First,
we let
$Z_n(\t,\mathbf{v})$ denote the height of the $n$th branching point on
the geodesic ray~$\mathbf{v}$.
One immediately verifies that, for every $n\geq1$,
\[
Z_n= \sum_{i=0}^{n-1}
Z_1\circ S^i. %
\]
If $A= \int\kappa(\cc(\t)) \Theta^*(\D\t \,\D\mathbf{v})$, it follows that
\begin{equation}
\label{altertech1} \frac{1}{n} Z_n \mathop{\longrightarrow}_{n\to\infty}^{\Theta^*\,{\rm a.s.}}
A^{-1} \int Z_1(\t,\mathbf{v}) \kappa \bigl(\cc(\t) \bigr)
\Theta^*(\D\t \,\D\mathbf{v}).
\end{equation}
Note that the limit can also be written as
\[
A^{-1} \E \biggl[ \bigl\llvert \log(1-U) \bigr\rrvert \kappa \biggl(
\biggl(U + \frac{1-U}{\cc_1 +
\cc_2} \biggr)^{-1} \biggr) \biggr]
\]
with the notation of the preceding proof. Second, if $\mathbf
{x}_{n,\mathbf{v}}$ stands for
the $(n+1)$st branching point on the geodesic ray $\mathbf{v}$ [with the notation
of Section~\ref{sec:yuletree}, $\mathbf{x}_{n,\mathbf{v}}=((v_1,\ldots
,v_n),Z_{n+1}(\t,\mathbf{v}))$
if $\mathbf{v}=(v_1,v_2,\ldots)$], we set for every $n\geq1$,
\[
H_n(\t,\mathbf{v}) = \log\nu_\t \bigl( \bigl\{
\mathbf{u} \in\{1,2\}^\N: \mathbf {x}_{n,\mathbf{v}}\prec\mathbf{u}
\bigr\} \bigr).
\]
It is then also easy to verify that
\[
H_n=\sum_{i=0}^{n-1}
H_1\circ S^i
\]
and we have thus
\begin{equation}
\label{altertech2} \frac{1}{n} H_n \mathop{\longrightarrow}_{n\to\infty}^{\Theta^*\,{\rm a.s.}}
A^{-1} \int H_1(\t,\mathbf{v}) \kappa \bigl(\cc(\t) \bigr)
\Theta^*(\D\t \,\D\mathbf{v}).
\end{equation}
The limit can be written as
\[
2 A^{-1} \E \biggl[ \frac{\cc_1}{\cc_1+\cc_2} \log \biggl(
\frac{\cc
_1}{\cc_1+\cc_2} \biggr) \kappa \biggl( \biggl(U + \frac{1-U}{\cc_1 + \cc
_2}
\biggr)^{-1} \biggr) \biggr].
\]
By combining~\eqref{altertech1} and~\eqref{altertech2}, we now obtain
that the convergence~\eqref{equivalent-asymp}
holds with limit
\[
-\beta= \frac{2 \E [ \afrac{\cc_1}{\cc_1+\cc_2} \log ( \afrac{\cc
_1}{\cc_1+\cc_2} ) \kappa ( (U + \vafrac{1-U}{\cc_1 + \cc
_2} )^{-1} ) ]}{
\E [ \llvert  \log(1-U)\rrvert   \kappa ( (U + \vafrac{1-U}{\cc_1 + \cc_2}
)^{-1} ) ]}.
\]
We leave it as an exercise for the reader to check that this is
consistent with the other formulas for $\beta$
in Proposition~\ref{prop:valuebeta}.

\subsection{Supercritical Galton--Watson trees}

One may compare our results about
Brownian motion on the Yule tree to the recent paper of A\"\i d\'ekon
\cite{Aid11}, which deals
with biased random walk on supercritical Galton--Watson trees.
To this end, consider the supercritical offspring distribution $\theta
^{(n)}$ given by
$ \theta^{(n)}(1) = 1 - \frac{1}{n}$ and $ \theta^{(n)}(2)= \frac{1}{n}$.
If $ \t^{(n)}$ is the (infinite) Galton--Watson tree with offspring
distribution $\theta^{(n)}$, then
$\t^{(n)}$, viewed as a metric space for the graph distance rescaled
by the fact $n^{-1}$,
converges in distribution in an appropriate sense (e.g., for the local
Gromov--Hausdorff
topology) to the Yule tree $\Gamma$.

Consider then the biased random walk $(Z^{(n)}_{k})_{k \geq0}$ on $\t
^{(n)}$ with bias parameter $\lambda^{(n)} = 1- \frac{1}{n}$ (see,
e.g., \cite{LPP96} or \cite{Aid11} for a definition of this process).
Since the ``mean drift'' of $Z^{(n)}$ away from the root is $\frac
{1}{2n} + o(n^{-1})$, it should be
clear that the rescaled process $( Z^{(n)}_{ \lfloor n^2 t\rfloor
})_{t\geq0}$ is asymptotically close to Brownian motion with drift
$1/2$ on the Yule tree, in a sense that can easily be made precise.

An explicit form of an invariant measure for the ``environment seen
from the particle'' has been derived
by A\"\i d\'ekon \cite{Aid11}, Theorem 4.1, for biased random walk on a
supercritical Galton--Watson tree (see also \cite{GMPV12} for a related
result in a different setting).
In the unbiased case such an explicit formula already appeared in the
work of Lyons, Pemantle and
Peres \cite{LPP95}, but in the subsequent work of the same authors \cite
{LPP96} dealing with the
biased case, only the existence of the invariance measure was derived
by general arguments.
It is tempting to use A\"\i d\'ekon's formula and the connection between
the $\lambda^{(n)}$-biased random walk on $ \t^{(n)}$ and Brownian
motion with drift
$1/2$ on the Yule tree to recover our formulas for invariant
measures in Propositions
\ref{invariant-meas} and \ref{invariantbis}. Note, however, that the
continuous analog
of A\"\i d\'ekon's formula would be an invariant measure for the
environment seen from
Brownian motion on the Yule tree at a {fixed} time, whereas we have
obtained invariant measures for the environment at the {last} visit of
a fixed height
(Proposition~\ref{invariant-meas}) or the {last} visit of a node of the $n$th generation
(Proposition~\ref{invariantbis}). Still the reader should note the similarity
between the limiting
distribution in \cite{Aid11}, Theorem 4.1, and the formula for the
invariant measure
in Proposition~\ref{invariantbis}. Indeed, we were able to guess the formula for
$\kappa$
in Proposition~\ref{invariantbis} from a (nonrigorous) passage to the limit from the
corresponding formula in \cite{Aid11}.

\begin{appendix}\label{appen}
\section*{Appendix}

In this appendix, we sketch a proof of Proposition \ref
{spine-decomposition}, which is based on the relation
between the continuous reduced tree $\Delta$ of Section~\ref{sec:treedelta} and the Brownian excursion
conditioned to hit level $1$. This relation was described after
Proposition~\ref{conv-reduced-tree} (see Figure~\ref{redu-excursion})
and we retain the notation introduced after this proposition. In
particular, $(\mathbf{e}_t)_{0\leq t\leq\zeta}$
is a Brownian excursion
conditioned to hit level $1$ defined under the probability measure $\P
$, and $\Delta$ is the associated continuous reduced tree.

We fix $\ve\in(0,1)$ and let $N_\ve\geq1$ be the number of excursions
of $\mathbf{e}$ from
$1-\ve$ to~$1$. We let $(R^\ve_1,S^\ve_1),(R^\ve_2,S^\ve_2),\ldots,
(R^\ve_{N_\ve},S^\ve_{N_\ve})$
be the time intervals corresponding to these excursions listed in
chronological order. For convenience, we also set $R^\ve_i=S^\ve
_i=\infty$
if $i>N_\ve$. The key ingredient of our proof is the following lemma.
We write $(B_t)_{t\geq0}$
for a linear Brownian motion that starts from $x$ under the probability
measure $P_x$, and
$T_0=\inf\{t\geq0:B_t=0\}$. We also let $n_\ve$ be the law of a
Brownian excursion above level $1-\ve$
conditioned to hit level $1$. Agreeing that the excursion stays
constant after returning to $1-\ve$, we can view $n_\ve$
as a probability measure on the space
$C(\R_+,\R_+)$ of all continuous functions from $\R_+$ into $\R_+$.

%
\begin{lemma}
\label{appen-lemma}
Let $F,G,H$ be three nonnegative measurable functions on
$C(\R_+,\R_+)$. Then
\begin{eqnarray*}
&&\E \Biggl[\sum_{i=1}^{N_\ve} F \bigl((
\mathbf{e}_{(R^\ve_i-t)^+})_{t\geq0} \bigr) G \bigl((\mathbf
{e}_{(R^\ve_i+t)\wedge S^\ve_i})_{t\geq0} \bigr) H \bigl((\mathbf{e}_{(S^\ve_i+t)\wedge\zeta})_{t\geq0}
\bigr) \Biggr]
\\
&&\qquad =\frac{1}{\ve} E_{1-\ve} \bigl[F \bigl((B_{t\wedge T_0})_{t\geq0}
\bigr) \bigr] n_\ve(G) E_{1-\ve} \bigl[H
\bigl((B_{t\wedge T_0})_{t\geq0} \bigr) \bigr].
\end{eqnarray*}
\end{lemma}

The proof of this lemma is straightforward. First note that, for every
$i\geq1$, the
law of $(\mathbf{e}_{(R^\ve_i+t)\wedge S^\ve_i})_{t\geq0}$ under $\P
(\cdot\mid N_\ve\geq i)$ is $n_\ve$. Then,\vspace*{1pt} since $S^\ve_i$ is a
stopping time for every integer $i\geq1$, we get by applying the
strong Markov property at time $S^\ve_i$,
\begin{eqnarray*}
&&\E \Biggl[\sum_{i=1}^{N_\ve} F \bigl((
\mathbf{e}_{(R^\ve_i-t)^+})_{t\geq0} \bigr) G \bigl((\mathbf
{e}_{(R^\ve_i+t)\wedge S^\ve_i})_{t\geq0} \bigr) H \bigl((\mathbf{e}_{(S^\ve_i+t)\wedge\zeta})_{t\geq0}
\bigr) \Biggr]
\\
&&\qquad= \sum_{i=1}^{\infty} \E \bigl[
\mathbf{1}_{\{S^\ve_i<\infty\}} F \bigl((\mathbf{e}_{(R^\ve
_i-t)^+})_{t\geq0}
\bigr) G \bigl((\mathbf{e}_{(R^\ve_i+t)\wedge S^\ve
_i})_{t\geq0} \bigr) \bigr]
\\
&&\qquad\quad{}\times E_{1-\ve} \bigl[H \bigl((B_{t\wedge T_0})_{t\geq0}
\bigr) \bigr].
\end{eqnarray*}
On the other hand, using the fact that the law of $(\mathbf
{e}_t)_{0\leq t\leq\zeta}$ is invariant under time reversal
[$(\mathbf{e}_t)_{0\leq t\leq\zeta}$ and $(\mathbf{e}_{\zeta
-t})_{0\leq t\leq\zeta}$ have the same law], we also
obtain that the sum in the second line of the last display is equal to
\begin{eqnarray*}
&&\E \Biggl[\sum_{i=1}^{N_\ve} G \bigl((
\mathbf{e}_{(S^\ve_i-t)\vee R^\ve_i})_{t\geq0} \bigr) F \bigl((\mathbf{e}_{(S^\ve_i+t)\wedge\zeta})_{t\geq0}
\bigr) \Biggr]
\\
&&\qquad= \E \Biggl[\sum_{i=1}^{N_\ve} G
\bigl((\mathbf{e}_{(S^\ve_i-t)\vee R^\ve_i})_{t\geq0} \bigr) \Biggr] \times
E_{1-\ve} \bigl[F \bigl((B_{t\wedge T_0})_{t\geq0} \bigr) \bigr]
\\
&&\qquad= \E[N_\ve] n_\ve(G) E_{1-\ve} \bigl[F
\bigl((B_{t\wedge
T_0})_{t\geq0} \bigr) \bigr]
\end{eqnarray*}
giving the desired result since $\E[N_\ve]=\frac{1}{\ve}$.

Let us informally explain why Proposition \ref{spine-decomposition} (or
the equivalent statement in terms
of the tree $\Delta$) follows from the lemma. To make the connection
with Proposition \ref{spine-decomposition},
we take $\ve=e^{-r}$, so that the factor $\frac{1}{\ve}$ becomes the
multiplicative factor
$e^r$ in the formula of Proposition \ref{spine-decomposition}. We first
recall that every vertex $v$ of $\Delta$ at height $1-\ve$
corresponds to one excursion of $\mathbf{e}$ above height $1-\ve$ that
hits level $1$, and we observe that the tree
of descendants of $v$ will be coded by this excursion in the same way
as $\Delta$ is coded by $\mathbf{e}$. Hence,
this tree of descendants is distributed as a scaled copy of $\Delta$
(and the scaling factor will disappear when we
do the logarithmic scale transformation to return to the Yule tree).
Then we need to consider the subtrees branching off the ancestral line
of $v$, and we can first look at those
subtrees branching on the right of the ancestral line. Supposing that
$v$ corresponds to the
excursion during the time interval $(R^\ve_i,S^\ve_i)$, the latter
subtrees exactly correspond to all excursions of
the process $(\mathbf{e}_{(S^\ve_i+t)\wedge\zeta})_{t\geq0}$ above its
past minimum process that hit level $1$, and the
level at which a subtree branches is the starting level of the
corresponding excursion. The formula
of Lemma \ref{appen-lemma} then leads us to consider the excursions of
$(B_{t\wedge T_0})_{t\geq0}$ above its
past minimum process, under the measure $P_{1-\ve}$. If $f_1,\ldots
,f_N$ stand for these excursions, and
if $h_i$ denotes the starting level of the excursion $f_i$, It\^o's
excursion theory shows that
the point measure $\sum_{i=1}^N \delta_{(h_i,f_i)}$ is Poisson with intensity
\[
\mathbf{1}_{[0,1-\ve]}(h) \frac{\mathrm{d} h}{1-h} n_{1-h}(
\mathrm{d}f).
\]
Recalling that $\ve=e^{-r}$, the image of the measure $\mathbf
{1}_{[0,1-\ve]}(h) \frac{\mathrm{d} h}{1-h}$ under the
logarithmic scale transformation $h=1-e^{-s}$ is the measure $\mathbf
{1}_{[0,r]}(s) \,\mathrm{d}s$.
This explains the form of the intensity of the Poisson measure $\n$ in
the statement of Proposition~\ref{spine-decomposition},
noting that the factor $2$ comes from the fact that we also need to
consider the subtrees that
branch on the left of the ancestral line [these are treated in a
similar manner, considering now the excursions of
$(\mathbf{e}_{(R^\ve_i-t)^+})_{t\geq0}$ above its past minimum process
that hit $1$].

Although we avoided introducing the notation that would be needed to
make the previous arguments
precise, the reader will easily turn these arguments into a rigorous proof
of Proposition \ref{spine-decomposition} based on Lemma \ref{appen-lemma}.
\end{appendix}

\section*{Acknowledgment}
We would like to thank Thordur Jonsson for suggesting the study
of the harmonic measure on critical Galton--Watson trees during Spring 2012.


%

\printaddresses

\begin{thebibliography}{33}
\bibitem{Aid11}
%
\begin{barticle}[mr]
\bauthor{\bsnm{A{\"{\i}}d{\'e}kon},~\bfnm{Elie}\binits{E.}}
(\byear{2014}).
\btitle{Speed of the biased random walk on a {G}alton--{W}atson tree}.
\bjournal{Probab. Theory Related Fields}
\bvolume{159}
\bpages{597--617}.
\bid{doi={10.1007/s00440-013-0515-y}, issn={0178-8051}, mr={3230003}}
\end{barticle}
%

\bptok{imsref}%
\endbibitem

\bibitem{Ald91}
%
\begin{barticle}[mr]
\bauthor{\bsnm{Aldous},~\bfnm{David}\binits{D.}}
(\byear{1991}).
\btitle{The continuum random tree. {I}}.
\bjournal{Ann. Probab.}
\bvolume{19}
\bpages{1--28}.
\bid{issn={0091-1798}, mr={1085326}}
\end{barticle}
%

\bptok{imsref}%
\endbibitem

\bibitem{Al3}
%
\begin{barticle}[mr]
\bauthor{\bsnm{Aldous},~\bfnm{David}\binits{D.}}
(\byear{1993}).
\btitle{The continuum random tree. {III}}.
\bjournal{Ann. Probab.}
\bvolume{21}
\bpages{248--289}.
\bid{issn={0091-1798}, mr={1207226}}
\end{barticle}
%

\bptok{imsref}%
\endbibitem

\bibitem{AN}
%
\begin{bbook}[mr]
\bauthor{\bsnm{Athreya},~\bfnm{Krishna~B.}\binits{K.~B.}} \AND
\bauthor{\bsnm{Ney},~\bfnm{Peter~E.}\binits{P.~E.}}
(\byear{1972}).
\btitle{Branching Processes}.
\bpublisher{Springer},
\blocation{New York-Heidelberg}.
\bid{mr={0373040}}
\end{bbook}
%

\bptok{imsref}%
\endbibitem

\bibitem{AEW}
%
\begin{barticle}[mr]
\bauthor{\bsnm{Athreya},~\bfnm{Siva}\binits{S.}},
\bauthor{\bsnm{Eckhoff},~\bfnm{Michael}\binits{M.}} \AND
\bauthor{\bsnm{Winter},~\bfnm{Anita}\binits{A.}}
(\byear{2013}).
\btitle{Brownian motion on {$\Bbb{R}$}-trees}.
\bjournal{Trans. Amer. Math. Soc.}
\bvolume{365}
\bpages{3115--3150}.
\bid{doi={10.1090/S0002-9947-2012-05752-7}, issn={0002-9947}, mr={3034461}}
\end{barticle}
%

\bptok{imsref}%
\endbibitem

\bibitem{Bl}
%
\begin{bbook}[mr]
\bauthor{\bsnm{Blumenthal},~\bfnm{Robert~M.}\binits{R.~M.}}
(\byear{1992}).
\btitle{Excursions of {M}arkov Processes}.
\bpublisher{Birkh\"auser},
\blocation{Boston, MA}.
\bid{doi={10.1007/978-1-4684-9412-9}, mr={1138461}}
\end{bbook}
%

\bptok{imsref}%
\endbibitem

\bibitem{BS02}
%
\begin{bbook}[mr]
\bauthor{\bsnm{Borodin},~\bfnm{Andrei~N.}\binits{A.~N.}} \AND
\bauthor{\bsnm{Salminen},~\bfnm{Paavo}\binits{P.}}
(\byear{2002}).
\btitle{Handbook of {B}rownian Motion---Facts and Formulae},
\bedition{2nd} ed.
\bpublisher{Birkh\"auser},
\blocation{Basel}.
\bid{doi={10.1007/978-3-0348-8163-0}, mr={1912205}}
\end{bbook}
%

\bptok{imsref}%
\endbibitem

\bibitem{Bou87}
%
\begin{barticle}[mr]
\bauthor{\bsnm{Bourgain},~\bfnm{J.}\binits{J.}}
(\byear{1987}).
\btitle{On the {H}ausdorff dimension of harmonic measure in higher dimension}.
\bjournal{Invent. Math.}
\bvolume{87}
\bpages{477--483}.
\bid{doi={10.1007/BF01389238}, issn={0020-9910}, mr={0874032}}
\end{barticle}
%

\bptok{imsref}%
\endbibitem

\bibitem{CRW}
%
\begin{barticle}[mr]
\bauthor{\bsnm{Chauvin},~\bfnm{Brigitte}\binits{B.}},
\bauthor{\bsnm{Rouault},~\bfnm{Alain}\binits{A.}} \AND
\bauthor{\bsnm{Wakolbinger},~\bfnm{Anton}\binits{A.}}
(\byear{1991}).
\btitle{Growing conditioned trees}.
\bjournal{Stochastic Process. Appl.}
\bvolume{39}
\bpages{117--130}.
\bid{doi={10.1016/0304-4149(91)90036-C}, issn={0304-4149}, mr={1135089}}
\end{barticle}
%

\bptok{imsref}%
\endbibitem

\bibitem{Cro08}
%
\begin{barticle}[mr]
\bauthor{\bsnm{Croydon},~\bfnm{David~A.}\binits{D.~A.}}
(\byear{2008}).
\btitle{Volume growth and heat kernel estimates for the continuum
random tree}.
\bjournal{Probab. Theory Related Fields}
\bvolume{140}
\bpages{207--238}.
\bid{doi={10.1007/s00440-007-0063-4}, issn={0178-8051}, mr={2357676}}
\end{barticle}
%

\bptok{imsref}%
\endbibitem

\bibitem{DLG}
%
\begin{barticle}[mr]
\bauthor{\bsnm{Duquesne},~\bfnm{Thomas}\binits{T.}} \AND
\bauthor{\bsnm{Le Gall},~\bfnm{Jean-Fran{\c{c}}ois}\binits{J.-F.}}
(\byear{2002}).
\btitle{Random trees, L\'evy processes and spatial branching processes}.
\bjournal{Ast\'erisque}
\bvolume{281}
\bpages{vi+147}.
\bid{issn={0303-1179}, mr={1954248}}
\bptnote{check pages}%
\end{barticle}
%

\bptok{imsref}%
\endbibitem

\bibitem{DLG06}
%
\begin{barticle}[mr]
\bauthor{\bsnm{Duquesne},~\bfnm{Thomas}\binits{T.}} \AND
\bauthor{\bsnm{Le Gall},~\bfnm{Jean-Fran{\c{c}}ois}\binits{J.-F.}}
(\byear{2006}).
\btitle{The {H}ausdorff measure of stable trees}.
\bjournal{ALEA Lat. Am. J. Probab. Math. Stat.}
\bvolume{1}
\bpages{393--415}.
\bid{issn={1980-0436}, mr={2291942}}
\end{barticle}
%

\bptok{imsref}%
\endbibitem

\bibitem{EK01}
%
\begin{barticle}[mr]
\bauthor{\bsnm{Enriquez},~\bfnm{Nathana{\"e}l}\binits{N.}} \AND
\bauthor{\bsnm{Kifer},~\bfnm{Yuri}\binits{Y.}}
(\byear{2001}).
\btitle{Markov chains on graphs and {B}rownian motion}.
\bjournal{J.~Theoret. Probab.}
\bvolume{14}
\bpages{495--510}.
\bid{doi={10.1023/A:1011119932045}, issn={0894-9840}, mr={1838739}}
\end{barticle}
%

\bptok{imsref}%
\endbibitem

\bibitem{FSS}
%
\begin{barticle}[mr]
\bauthor{\bsnm{Fleischmann},~\bfnm{Klaus}\binits{K.}} \AND
\bauthor{\bsnm{Siegmund-Schultze},~\bfnm{Rainer}\binits{R.}}
(\byear{1977}).
\btitle{The structure of reduced critical {G}alton--{W}atson processes}.
\bjournal{Math. Nachr.}
\bvolume{79}
\bpages{233--241}.
\bid{issn={0025-584X}, mr={0461689}}
\end{barticle}
%

\bptok{imsref}%
\endbibitem

\bibitem{FS00}
%
\begin{barticle}[mr]
\bauthor{\bsnm{Freidlin},~\bfnm{Mark}\binits{M.}} \AND
\bauthor{\bsnm{Sheu},~\bfnm{Shuenn-Jyi}\binits{S.-J.}}
(\byear{2000}).
\btitle{Diffusion processes on graphs: Stochastic differential
equations, large deviation principle}.
\bjournal{Probab. Theory Related Fields}
\bvolume{116}
\bpages{181--220}.
\bid{doi={10.1007/PL00008726}, issn={0178-8051}, mr={1743769}}
\end{barticle}
%

\bptok{imsref}%
\endbibitem

\bibitem{GMPV12}
%
\begin{barticle}[mr]
\bauthor{\bsnm{Gantert},~\bfnm{Nina}\binits{N.}},
\bauthor{\bsnm{M{\"u}ller},~\bfnm{Sebastian}\binits{S.}},
\bauthor{\bsnm{Popov},~\bfnm{Serguei}\binits{S.}} \AND
\bauthor{\bsnm{Vachkovskaia},~\bfnm{Marina}\binits{M.}}
(\byear{2012}).
\btitle{Random walks on {G}alton--{W}atson trees with random conductances}.
\bjournal{Stochastic Process. Appl.}
\bvolume{122}
\bpages{1652--1671}.
\bid{doi={10.1016/j.spa.2012.01.004}, issn={0304-4149}, mr={2914767}}
\end{barticle}
%

\bptok{imsref}%
\endbibitem

\bibitem{Jan}
%
\begin{barticle}[mr]
\bauthor{\bsnm{Janson},~\bfnm{Svante}\binits{S.}}
(\byear{2006}).
\btitle{Random cutting and records in deterministic and random trees}.
\bjournal{Random Structures Algorithms}
\bvolume{29}
\bpages{139--179}.
\bid{doi={10.1002/rsa.20086}, issn={1042-9832}, mr={2245498}}
\end{barticle}
%

\bptok{imsref}%
\endbibitem

\bibitem{Jon12}
%
\begin{bmisc}[auto:parserefs-M02]
\bauthor{\bsnm{Jonsson},~\bfnm{T.}\binits{T.}}
(\byear{2012}).
\bhowpublished{Private communication}.
\end{bmisc}
%

\bptok{imsref}%
\endbibitem

\bibitem{Kal}
%
\begin{bbook}[mr]
\bauthor{\bsnm{Kallenberg},~\bfnm{Olav}\binits{O.}}
(\byear{2002}).
\btitle{Foundations of Modern Probability},
\bedition{2nd} ed.
\bpublisher{Springer},
\blocation{New York}.
\bid{doi={10.1007/978-1-4757-4015-8}, mr={1876169}}
\bptnote{check year}%
\end{bbook}
%

\bptok{imsref}%
\endbibitem

\bibitem{Kre95}
%
\begin{barticle}[mr]
\bauthor{\bsnm{Krebs},~\bfnm{W.~B.}\binits{W.~B.}}
(\byear{1995}).
\btitle{Brownian motion on the continuum tree}.
\bjournal{Probab. Theory Related Fields}
\bvolume{101}
\bpages{421--433}.
\bid{doi={10.1007/BF01200505}, issn={0178-8051}, mr={1324094}}
\end{barticle}
%

\bptok{imsref}%
\endbibitem

\bibitem{Law93}
%
\begin{barticle}[mr]
\bauthor{\bsnm{Lawler},~\bfnm{Gregory~F.}\binits{G.~F.}}
(\byear{1993}).
\btitle{A discrete analogue of a theorem of {M}akarov}.
\bjournal{Combin. Probab. Comput.}
\bvolume{2}
\bpages{181--199}.
\bid{doi={10.1017/S0963548300000584}, issn={0963-5483}, mr={1249129}}
\end{barticle}
%

\bptok{imsref}%
\endbibitem

\bibitem{LG89}
%
\begin{bincollection}[mr]
\bauthor{\bsnm{Le Gall},~\bfnm{Jean-Fran{\c{c}}ois}\binits{J.-F.}}
(\byear{1989}).
\btitle{Marches al\'eatoires, mouvement brownien et processus de branchement}.
In \bbooktitle{S\'eminaire de {P}robabilit\'es, {XXIII}}.
\bseries{Lecture Notes in Math.}
\bvolume{1372}
\bpages{258--274}.
\bpublisher{Springer},
\blocation{Berlin}.
\bid{doi={10.1007/BFb0083978}, mr={1022916}}
\end{bincollection}
%

\bptok{imsref}%
\endbibitem

\bibitem{probasur}
%
\begin{barticle}[mr]
\bauthor{\bsnm{Le Gall},~\bfnm{Jean-Fran{\c{c}}ois}\binits{J.-F.}}
(\byear{2005}).
\btitle{Random trees and applications}.
\bjournal{Probab. Surv.}
\bvolume{2}
\bpages{245--311}.
\bid{doi={10.1214/154957805100000140}, issn={1549-5787}, mr={2203728}}
\end{barticle}
%

\bptok{imsref}%
\endbibitem

\bibitem{Lin}
%
\begin{barticle}[mr]
\bauthor{\bsnm{Lin},~\bfnm{Shen}\binits{S.}}
(\byear{2014}).
\btitle{The harmonic measure of balls in critical {G}alton--{W}atson
trees with infinite variance offspring distribution}.
\bjournal{Electron. J. Probab.}
\bvolume{19}
\bpages{1--35}.
\bid{doi={10.1214/EJP.v19-3498}, issn={1083-6489}, mr={3272331}}
\end{barticle}
%

\bptok{imsref}%
\endbibitem

\bibitem{Lin2}
%
\begin{bmisc}[auto:parserefs-M02]
\bauthor{\bsnm{Lin},~\bfnm{S.}\binits{S.}}
(\byear{2015}).
\bhowpublished{Typical behavior of the harmonic measure in critical
Galton-Watson trees.
Preprint. Available at \arxivurl{arXiv:1502.05584}.}
\end{bmisc}
%

\bptok{imsref}%
\endbibitem

\bibitem{L00}
%
\begin{barticle}[mr]
\bauthor{\bsnm{Lyons},~\bfnm{Russell}\binits{R.}}
(\byear{2000}).
\btitle{Singularity of some random continued fractions}.
\bjournal{J. Theoret. Probab.}
\bvolume{13}
\bpages{535--545}.
\bid{doi={10.1023/A:1007837306508}, issn={0894-9840}, mr={1778585}}
\end{barticle}
%

\bptok{imsref}%
\endbibitem

\bibitem{LPP95}
%
\begin{barticle}[mr]
\bauthor{\bsnm{Lyons},~\bfnm{Russell}\binits{R.}},
\bauthor{\bsnm{Pemantle},~\bfnm{Robin}\binits{R.}} \AND
\bauthor{\bsnm{Peres},~\bfnm{Yuval}\binits{Y.}}
(\byear{1995}).
\btitle{Ergodic theory on {G}alton--{W}atson trees: Speed of random
walk and dimension of harmonic measure}.
\bjournal{Ergodic Theory Dynam. Systems}
\bvolume{15}
\bpages{593--619}.
\bid{doi={10.1017/S0143385700008543}, issn={0143-3857}, mr={1336708}}
\end{barticle}
%

\bptok{imsref}%
\endbibitem

\bibitem{LPP96}
%
\begin{barticle}[mr]
\bauthor{\bsnm{Lyons},~\bfnm{Russell}\binits{R.}},
\bauthor{\bsnm{Pemantle},~\bfnm{Robin}\binits{R.}} \AND
\bauthor{\bsnm{Peres},~\bfnm{Yuval}\binits{Y.}}
(\byear{1996}).
\btitle{Biased random walks on {G}alton--{W}atson trees}.
\bjournal{Probab. Theory Related Fields}
\bvolume{106}
\bpages{249--264}.
\bid{doi={10.1007/s004400050064}, issn={0178-8051}, mr={1410689}}
\end{barticle}
%

\bptok{imsref}%
\endbibitem

\bibitem{LP10}
%
\begin{bbook}[auto:parserefs-M02]
\bauthor{\bsnm{Lyons},~\bfnm{R.}\binits{R.}} \AND
\bauthor{\bsnm{Peres},~\bfnm{Y.}\binits{Y.}}
(\byear{2015}).
\btitle{Probability on Trees and Networks}.
\bnote{Preprint. Available at \surl{http://mypage.iu.edu/\textasciitilde
rdlyons/}.}
\end{bbook}
%

\bptok{imsref}%
\endbibitem

\bibitem{Mak95}
%
\begin{barticle}[mr]
\bauthor{\bsnm{Makarov},~\bfnm{N.~G.}\binits{N.~G.}}
(\byear{1985}).
\btitle{On the distortion of boundary sets under conformal mappings}.
\bjournal{Proc. Lond. Math. Soc. (3)}
\bvolume{51}
\bpages{369--384}.
\bid{doi={10.1112/plms/s3-51.2.369}, issn={0024-6115}, mr={0794117}}
\end{barticle}
%

\bptok{imsref}%
\endbibitem

\bibitem{Pit06}
%
\begin{bbook}[mr]
\bauthor{\bsnm{Pitman},~\bfnm{J.}\binits{J.}}
(\byear{2006}).
\btitle{Combinatorial Stochastic Processes}.
\bseries{Lecture Notes in Math.}
\bvolume{1875}.
\bpublisher{Springer},
\blocation{Berlin}.
\bid{mr={2245368}}
\end{bbook}
%

\bptok{imsref}%
\endbibitem

\bibitem{RY}
%
\begin{bbook}[mr]
\bauthor{\bsnm{Revuz},~\bfnm{Daniel}\binits{D.}} \AND
\bauthor{\bsnm{Yor},~\bfnm{Marc}\binits{M.}}
(\byear{1991}).
\btitle{Continuous Martingales and {B}rownian Motion}.
\bpublisher{Springer},
\blocation{Berlin}.
\bid{doi={10.1007/978-3-662-21726-9}, mr={1083357}}
\end{bbook}
%

\bptok{imsref}%
\endbibitem

\bibitem{Zu75}
%
\begin{barticle}[mr]
\bauthor{\bsnm{Zubkov},~\bfnm{A.~M.}\binits{A.~M.}}
(\byear{1975}).
\btitle{Limit distributions of the distance to the nearest common ancestor}.
\bjournal{Teor. Verojatnost. i Primenen.}
\bvolume{20}
\bpages{614--623}.
\bid{issn={0040-361X}, mr={0397915}}
\end{barticle}
%

\bptok{imsref}%
\endbibitem
\end{thebibliography}
\end{document}